\documentclass[11pt]{amsart}

\usepackage[T1]{fontenc}
\usepackage[utf8]{inputenc}
\usepackage{lmodern}

\usepackage{layout}
\usepackage[american]{babel}
\usepackage[babel, final]{microtype}
\usepackage{amssymb}
\usepackage{amscd}
\usepackage{mathtools}
\usepackage{tensor}
\usepackage[mathscr]{eucal}
\usepackage{textcomp}
\usepackage[all,knot,arc]{xy}

\usepackage{xifthen}
\usepackage{chngcntr}
\usepackage{enumitem}
\usepackage{multicol}

\usepackage[pdftex,%
dvipsnames,%
]{xcolor}
\definecolor{bluearr}{RGB}{58,106,161}
\definecolor{redarr}{RGB}{158,49,53}
\definecolor{greenarr}{RGB}{62,136,74}
\usepackage[pdftex, final]{graphicx}
\usepackage{caption}
\usepackage{subcaption}
\usepackage{pinlabel}
\usepackage[pdftex,%
letterpaper,%
includehead,%
includefoot,%
nomarginpar,%
lmargin=1in,%
rmargin=1in,%
tmargin=1in,%
bmargin=1in,%
]{geometry}
\usepackage[pdftex,%
final,%
colorlinks=true,%
linkcolor=NavyBlue,%
citecolor=NavyBlue,%
filecolor=NavyBlue,%
menucolor=NavyBlue,%
urlcolor=NavyBlue,%
bookmarks=true,%
bookmarksdepth=3,%
bookmarksnumbered=true,%
bookmarksopen=true,%
bookmarksopenlevel=2,%
]{hyperref}
\hypersetup{
  pdfinfo={
    Title={An unoriented skein relation for bordered--sutured Floer 
      homology},
    Author={David Shea Vela-Vick, C.-M. Michael Wong},
    Subject={Tangle invariants},
    Keywords={Tangles, knot Floer homology, bordered Floer homology, skein 
      relation}
  }
}

\usepackage{tikz}
\usetikzlibrary{arrows, automata}

\input{macros}

\newcommand{\set}[1]{\mathchoice%
  {\left\lbrace #1 \right\rbrace}%
  {\lbrace #1 \rbrace}%
  {\lbrace #1 \rbrace}%
  {\lbrace #1 \rbrace}%
}
\newcommand{\setc}[2]{\mathchoice%
  {\left\lbrace #1 \, \middle\vert \, #2 \right\rbrace}%
  {\lbrace #1 \, \vert \, #2 \rbrace}%
  {\lbrace #1 \, \vert \, #2 \rbrace}%
  {\lbrace #1 \, \vert \, #2 \rbrace}%
}
\newcommand{\paren}[1]{\mathchoice%
  {\left( #1 \right)}%
  {( #1 )}%
  {( #1 )}%
  {( #1 )}%
}

\newcommand{\abs}[1]{\mathchoice%
  {\left\lvert #1 \right\rvert}%
  {\lvert #1 \rvert}%
  {\lvert #1 \rvert}%
  {\lvert #1 \rvert}%
}

\newcommand{\vphi}{\varphi}
\newcommand{\eset}{\varnothing}

\newcommand{\dirsum}{\oplus}
\newcommand{\bigdirsum}{\bigoplus}
\renewcommand{\tensor}{\otimes}

\newcommand{\union}{\cup}

\newcommand{\intersect}{\cap}

\newcommand{\comp}{\circ}
\newcommand{\cross}{\times}

\newcommand{\isom}{\cong}

\newcommand{\numset}[1]{\mathbb{#1}}

\newcommand{\Z}{\numset{Z}}

\newcommand{\R}{\numset{R}}

\newcommand{\F}[1]{\numset{F}_{#1}}
\newcommand{\cycgrp}[1]{\Z / #1 \Z}
\newcommand{\field}[1]{\mathbf{#1}}

\renewcommand{\k}{\field{k}}

\DeclareMathOperator{\id}{id}
\DeclareMathOperator{\Id}{Id}

\DeclareMathOperator{\im}{Im}
\renewcommand{\Im}{\im}

\DeclareMathOperator{\Cone}{Cone}
\DeclareMathOperator{\rk}{rk}

\newcommand{\bdy}{\partial}

\newcommand{\nbhd}[1]{\nu (#1)}

\DeclareMathOperator{\CFK}{CFK}
\DeclareMathOperator{\HFK}{HFK}
\DeclareMathOperator{\CFL}{CFL}

\DeclareMathOperator{\Kh}{Kh}

\newcommand{\ophat}[1]{\widehat{#1}}
\newcommand{\optilde}[1]{\widetilde{#1}}

\newcommand{\CFKh}{\ophat{\CFK}}

\newcommand{\HFKh}{\ophat{\HFK}}

\newcommand{\CFLt}{\optilde{\CFL}}

\newcommand{\Kht}{\optilde{\Kh}}

\newcommand{\HD}{\mathcal H}

\newcommand{\curves}[1]{\boldsymbol{#1}}
\newcommand{\alphas}[1][]{%
  \ifthenelse{\equal{#1}{}}{\curves{\alpha}}{\curves{\alpha^{#1}}}
}
\newcommand{\betas}[1][]{%
  \ifthenelse{\equal{#1}{}}{\curves{\beta}}{\curves{\beta_{#1}}}
}

\newcommand{\gen}[1]{\mathbf{#1}}
\newcommand{\genset}[1]{\mathfrak{#1}}
\newcommand{\x}{\gen{x}}
\newcommand{\y}{\gen{y}}
\newcommand{\z}{\gen{z}}
\newcommand{\w}{\gen{w}}

\DeclareMathOperator{\spin}{Spin}
\newcommand{\SpinC}{\spin^c}

\newcommand{\spinc}{\mathfrak s}

\DeclareMathOperator{\ind}{ind}
\DeclareMathOperator{\gr}{gr}

\newcommand{\emptypoly}[1]{#1^\circ}

\DeclareMathOperator{\Tri}{Tri}

\DeclareMathOperator{\Quad}{Quad}
\DeclareMathOperator{\Hex}{\Hex}

\tikzset{cdlabel/.style={above,sloped,
    execute at begin node=$\scriptstyle,execute at end node=$}}
\tikzset{algarrow/.style={->, thick}}
\tikzset{blgarrow/.style={->, thick}}
\tikzset{clgarrow/.style={->, thick}}
\tikzset{tensoralgarrow/.style={double, double equal sign distance, -implies}}
\tikzset{tensorblgarrow/.style={double, double equal sign distance, -implies}}
\tikzset{tensorclgarrow/.style={double, double equal sign distance, -implies}}
\tikzset{tensorelgarrow/.style={double, double equal sign distance, -implies}}
\tikzset{modarrow/.style={->, dashed}}
\tikzset{Amodar/.style={->, dashed}}
\tikzset{Dmodar/.style={->, dashed}}
\tikzset{DAmodar/.style={->, dashed}}

\tikzset{al/.style={->, bend right=20, thick}}
\tikzset{ar/.style={->, bend left=20, thick}}
\tikzset{als/.style={->, bend right=15, thick}}
\tikzset{ars/.style={->, bend left=15, thick}}

\newcommand{\alg}[1]{\mathcal{#1}}
\newcommand{\module}[1]{\mathcal{#1}}
\newcommand{\Ainf}{\alg{A}_\infty}

\newcommand{\moduletype}[1]{\mathit{#1}}

\newcommand{\DA}{\moduletype{DA}}

\renewcommand{\AA}{\moduletype{AA}}

\newcommand{\M}{\module{M}}

\renewcommand{\emptypoly}[2][]{%
  #2^{\ifthenelse{\equal{#1}{}}{\circ}{\circ, #1}}
}

\newcommand{\position}[2]{#1^#2}

\newcommand{\onbdy}[1]{\position{#1}{\bdy}}

\newcommand{\bdybdy}{\onbdy{\bdy}}

\newcommand{\arcdiag}{\mathcal{Z}}
\newcommand{\matchedpts}{\mathbf{a}}

\DeclareMathOperator{\inv}{inv}

\DeclareMathOperator{\SFC}{SFC}
\DeclareMathOperator{\SFH}{SFH}

\DeclareMathOperator{\BSD}{BSD}
\DeclareMathOperator{\BSA}{BSA}

\DeclareMathOperator{\BSDA}{BSDA}

\newcommand{\BSDh}{\ophat{\BSD}}
\newcommand{\BSAh}{\ophat{\BSA}}

\newcommand{\BSDAh}{\ophat{\BSDA}}

\newcommand{\orarcs}{\mathbf{Z}}

\newcommand*{\oplbar}[1]{\underline{#1}}
\DeclareMathOperator{\Gr}{Gr}
\newcommand*{\Grr}{\oplbar{\Gr}}
\newcommand*{\grr}{\oplbar{\gr}}

\newcommand*{\grP}{\mathcal{P}}
\newcommand*{\grPr}{\oplbar{\grP}}

\renewcommand*{\Id}{\mathbb{I}}

\newcommand*{\glue}[1]{\mathchoice%
  {\mathbin{\operatorname*{\union}_{#1}}}%
  {\union_{#1}}%
  {\union_{#1}}%
  {\union_{#1}}%
}

\newcommand*{\HDMan}{\HD^{\mathrm{Man}}}
\newcommand*{\HDMank}{\HDMan_k}
\newcommand*{\HDMannext}{\HDMan_{k+1}}

\newcommand*{\FMan}{F^{\mathrm{Man}}}
\newcommand*{\FMank}{\FMan_k}

\newcommand*{\HDManp}{\HD^{\mathrm{Man},\mathrm{st}}}
\newcommand*{\HDManpinfty}{\HDManp_{\infty}}
\newcommand*{\HDManpzero}{\HDManp_{0}}
\newcommand*{\HDManpone}{\HDManp_{1}}

\newcommand*{\dg}{\mathit{dg}}
\newcommand*{\DAA}{\moduletype{DAA}}
\newcommand*{\Mod}{\mathsf{Mod}}

\newcommand*{\A}{\alg{A}}
\newcommand*{\AZ}{\alg{A} (\arcdiag)}
\newcommand*{\AZi}{\alg{A} (\arcdiag, i)}
\newcommand*{\AZfive}{\alg{A} (\arcdiag, 5)}
\newcommand*{\IZfive}{\alg{I} (\arcdiag, 5)}

\newcommand*{\Iocc}[1]{I_{\overline{#1}}}
\newcommand*{\rhos}{\boldsymbol{\rho}}

\newcommand*{\Istart}{I_\mathrm{s}}
\newcommand*{\Iend}{I_\mathrm{e}}

\newcommand*{\closure}[1]{\overline{#1}}

\newcommand*{\elT}{T^{\mathrm{el}}}
\newcommand*{\Tan}{\mathcal{T}}
\newcommand*{\elTan}{\Tan^{\mathrm{el}}}

\newcommand*{\sut}{\Gamma}
\newcommand*{\mersut}{\sut^{\mathrm{el}}}

\newcommand*{\elparam}{\phi^{\mathrm{el}}}

\newcommand*{\B}{\mathcal{B}}
\newcommand*{\Bk}{\B_k}
\newcommand*{\Bnext}{\B_{k+1}}
\newcommand*{\Bprev}{\B_{k+2}}
\newcommand*{\Y}{\mathcal{Y}}
\newcommand*{\Yk}{\Y_k}
\newcommand*{\Ynext}{\Y_{k+1}}
\newcommand*{\Yprev}{\Y_{k+2}}

\newcommand*{\sutF}{\mathcal{F}}

\newcommand*{\diff}{\delta}

\newcommand*{\into}{\hookrightarrow}

\newcommand*{\source}{S^{\triangleright}}
\newcommand*{\moduli}{\mathcal{M}}
\newcommand*{\embmoduli}{\moduli_{\mathrm{emb}}}

\DeclareMathOperator{\Mor}{Mor}

\newcommand*{\homeq}{\simeq}

\newcommand*{\boxtensor}{\boxtimes}

\newcommand*{\Eta}{\boldsymbol{\eta}}

\DeclareMathOperator{\Int}{Int}

\renewcommand*{\genset}{\mathcal{G}}

\DeclareMathOperator{\Conf}{Conf}

\newcommand*{\args}[1]{\mathopen{}\left(%
    #1%
  \right)}

\newcommand*{\deltat}{\tilde{\delta}}

\newcommand*{\blank}{\scalebox{0.75}[1.0]{$\mathord{-}$}} 

\newcommand*{\vkappa}{\kappa}

\newcommand*{\betak}{\beta^k}
\newcommand*{\betanext}{\beta^{k+1}}
\newcommand*{\betaprev}{\beta^{k+2}}

\newcommand*{\betask}{\betas^k}
\newcommand*{\betasnext}{\betas^{k+1}}
\newcommand*{\betasprev}{\betas^{k+2}}
\newcommand*{\betasone}{\betas^1}
\newcommand*{\betasinfty}{\betas^\infty}
\newcommand*{\betaszero}{\betas^0}

\newcommand*{\Ham}[1]{{#1_{\mathrm{H}}}}

\newcommand*{\betakH}{\beta^{\Ham{k}}}
\newcommand*{\betanextH}{\beta^{\Ham{(k+1)}}}

\newcommand*{\betaskH}{\betas^{\Ham{k}}}
\newcommand*{\betasnextH}{\betas^{\Ham{(k+1)}}}
\newcommand*{\betasprevH}{\betas^{\Ham{(k+2)}}}
\newcommand*{\betasoneH}{\betas^{\Ham{1}}}
\newcommand*{\betasinftyH}{\betas^{\Ham{\infty}}}
\newcommand*{\betaszeroH}{\betas^{\Ham{0}}}

\newcommand*{\etak}{\eta^{k, k+1}}
\newcommand*{\etanext}{\eta^{k+1, k+2}}

\newcommand*{\etaone}{\eta^{1, \infty}}
\newcommand*{\etainfty}{\eta^{\infty, 0}}
\newcommand*{\etazero}{\eta^{0, 1}}

\newcommand*{\etaHone}{\eta^{\Ham{1}, \infty}}

\newcommand*{\etakH}{\eta^{k, \Ham{(k+1)}}}

\newcommand*{\etazeroH}{\eta^{0, \Ham{1}}}

\newcommand*{\Etaknext}{\Eta^{k < k+1}}
\newcommand*{\Etanextnext}{\Eta^{k+1 < k+2}}
\newcommand*{\Etaprevnext}{\Eta^{k+2 < k}}

\newcommand*{\Etakprev}{\Eta^{k < k+2}}
\newcommand*{\Etanextprev}{\Eta^{k+1 < k}}

\newcommand*{\EtaknextH}{\Eta^{k < \Ham{(k+1)}}}
\newcommand*{\EtanextnextH}{\Eta^{k+1 < \Ham{(k+2)}}}
\newcommand*{\EtaprevnextH}{\Eta^{k+2 < \Ham{k}}}

\newcommand*{\EtakprevH}{\Eta^{k < \Ham{(k+2)}}}

\renewcommand*{\Theta}{\boldsymbol{\theta}}

\newcommand*{\thetaHk}{\theta^{\Ham{k} < k}}

\newcommand*{\thetaHone}{\theta^{\Ham{1} < 1}}

\newcommand*{\thetakH}{\theta^{k < \Ham{k}}}

\newcommand*{\thetaoneH}{\theta^{1 < \Ham{1}}}

\newcommand*{\ThetaHk}{\Theta^{\Ham{k} < k}}
\newcommand*{\ThetaHnext}{\Theta^{\Ham{(k+1)} < k+1}}

\newcommand*{\ThetakH}{\Theta^{k < \Ham{k}}}

\newcommand*{\HDbH}{\HD_{\betas}^{\mathrm{H}}}
\newcommand*{\HDab}{\HD_{\alphas, \betas}}
\newcommand*{\HDabH}{\HD_{\alphas, \betas}^{\mathrm{H}}}

\newcommand*{\fkH}{f_{k, \mathrm{H}}}

\newcommand*{\fprevH}{f_{k+2, \mathrm{H}}}
\newcommand*{\vphikH}{\vphi_{k, \mathrm{H}}}
\newcommand*{\vphinextH}{\vphi_{k+1, \mathrm{H}}}

\newcommand*{\PhiHk}{\Phi_{\Ham{k} < k}}

\newcommand*{\Xik}{\Xi_{k}}

\newcommand*{\Xiprev}{\Xi_{k+2}}

\newcommand*{\PhikH}{\Phi_{k < \Ham{k}}}

\newcommand*{\fcomb}{{\overline{f}}}
\newcommand*{\vphicomb}{{\overline{\vphi}}}

\newcommand*{\fcombk}{\fcomb_k}

\newcommand*{\vphicombk}{\vphicomb_k}

\newcommand*{\fcombki}{\fcomb_{k,i}}

\newcommand*{\vphicombkij}{\vphicomb_{k,ij}}

\DeclareMathOperator{\BSDAA}{BSDAA}
\newcommand*{\BSDAAh}{\ophat{\BSDAA}}

\newcommand*{\calF}{\mathcal{F}}
\newcommand*{\HC}{\mathsf{H}}

\newcommand*{\betaspH}{\betas'^{,\mathrm{H}}}
\newcommand*{\HDpH}{{\HD}'^{,\mathrm{H}}}

\DeclareMathOperator{\BSDDA}{BSDDA}
\newcommand*{\BSDDAh}{\ophat{\BSDDA}}

\newcommand*{\piece}{\mathcal{P}}

\newcommand*{\rcoset}{\backslash}

\newcommand*{\grbox}{\grr^{\mathord{\boxtimes}}}
\newcommand*{\grboxsk}{\grr^{\mathord{\boxtimes}, \mathrm{sk}}}

\newcommand*{\pihom}{\pi_{\mathrm{hom}}}

\newcommand*{\Grrsk}{\Grr^{\mathrm{sk}}_r}
\newcommand*{\grrsk}{\grr^{\mathrm{sk}}}

\newcommand*{\grrr}{\grr_r}
\newcommand*{\grrrsk}{\grrsk_r}

\newcommand*{\grPsk}{\grPr^{\mathrm{sk}}}
\newcommand*{\grrPsk}{\grPsk_r}

\newcommand*{\grPmin}{\grPr^{\mathrm{min}}}
\newcommand*{\grrPmin}{\grPmin_r}

\newcommand{\subgrp}[1]{\mathchoice%
  {\left\langle #1 \right\rangle}%
  {\langle #1 \rangle}%
  {\langle #1 \rangle}%
  {\langle #1 \rangle}%
}

\newcommand*{\Grrrr}{\Grr_{r_1, r_2, r}}
\newcommand*{\Grrrrsk}{\Grr^{\mathrm{sk}}_{r_1, r_2, r}}
\newcommand*{\grrrrboxsk}{\grboxsk_{r_1, r_2, r}}

\newcommand*{\wocc}[1]{\w_{\overline{#1}}'}

\newcommand*{\Msk}{M^{\mathrm{sk}}}

\newcommand*{\pisk}{\pi^{\mathrm{sk}}}
\newcommand*{\eqrel}{\mathord{\sim}}
\newcommand*{\eqrelsk}{\eqrel_{\mathrm{sk}}}

\DeclareMathOperator*{\PD}{PD}

\renewcommand*{\div}{\operatorname{div}}

\newcommand*{\eval}[2]{\mathchoice%
  {\left\langle #1, #2 \right\rangle}%
  {\langle #1, #2 \rangle}%
  {\langle #1, #2 \rangle}%
  {\langle #1, #2 \rangle}%
}

\begin{document}

\title[An unoriented skein relation via bordered--sutured Floer homology]{An 
  unoriented skein relation \\ via bordered--sutured Floer homology}

\author[David Shea Vela-Vick]{David Shea Vela-Vick}
\address{Department of Mathematics \\ Louisiana State University \\ Baton Rouge, LA 70803}
\email{\href{mailto:shea@math.lsu.edu}{shea@math.lsu.edu}}
\urladdr{\url{http://www.math.lsu.edu/~shea/}}

\author[C.-M. Michael Wong]{C.-M. Michael Wong}
\address{Department of Mathematics \\ Louisiana State University \\ Baton Rouge, LA 70803}
\email{\href{mailto:cmmwong@lsu.edu}{cmmwong@lsu.edu}}
\urladdr{\url{http://www.math.lsu.edu/~cmmwong/}}

\thanks{DSV was partially supported by NSF Grant DMS-1249708 and Simons Foundation Grant 524876.}



\begin{abstract}

  We show that the bordered--sutured Floer invariant of the complement of a 
  tangle in an arbitrary $3$-manifold $Y$, with minimal conditions on the 
  bordered--sutured structure, satisfies an unoriented skein exact triangle.  
  This generalizes a theorem by Manolescu \cite{Man07:HFKUnorientedSkein} for 
  links in $S^3$. We give a theoretical proof of this result by adapting 
  holomorphic polygon counts to the bordered--sutured setting, and also give a 
  combinatorial description of all maps involved and explicitly compute them.  
  We then show that, for $Y = S^3$, our exact triangle coincides with 
  Manolescu's. Finally, we provide a graded version of our result, explaining 
  in detail the grading reduction process involved.

\end{abstract}


\maketitle


\section{Introduction}
\label{sec:introduction}

Knot Floer homology, defined by Ozsv\'ath and Szab\'o \cite{OzsSza04:HFK}, and 
independently by Rasmussen \cite{Ras03:HFK}, is an invariant of oriented, 
null-homologous knots $K$ in an oriented, connected, closed $3$-manifold $Y$.  
It is then extended by Ozsv\'ath and Szab\'o \cite{OzsSza08:HFL} to an 
invariant for oriented links $L$.  The simplest flavor $\HFKh (Y, L)$ is a 
bigraded module over $\F{2} = \cycgrp{2}$ or $\Z$, which categorifies the 
Alexander polynomial of $L$ when $Y = S^3$; for this reason, it is often 
compared to Khovanov homology \cite{Kho00:Kh}, a link invariant $\Kht (L)$ that 
categorifies the Jones polynomial. Knot Floer homology is known to detect the 
fiberedness \cite{Ni07:HFKFibered} and the genus 
\cite{OzsSza04:HFThurstonNormHFKGenus} of a knot, and has given rise to a 
number of concordance invariants \cite{OzsSza03:HFKg4, Hom14:Epsilon, 
  OzsStiSza17:Upsilon}.  Among other applications, it has also led to 
invariants of Legendrian and transverse links in a contact $3$-manifold 
\cite{NgOzsThu08:GRIDEffective, LisOzsSti09:LOSS}.

It is well known that the Alexander and Jones polynomials satisfy skein 
relations, which have played a very important role in our understanding of 
these invariants.  In particular, both polynomials satisfy an oriented skein 
relation that relates two links that differ by a crossing and their common 
oriented resolution, while only the Jones polynomial satisfies an unoriented 
skein relation that relates a link with its two resolutions at a crossing.  
There are categorified statements of some of these skein relations, in the form 
of long exact sequences: Knot Floer homology satisfies an oriented skein exact 
triangle \cite{OzsSza04:HFK}, while Khovanov homology satisfies an unoriented 
one \cite{Kho00:Kh, Vir04:KhUnorientedSkein, Ras05:HFKKhExpo}.

Curiously, while the Alexander polynomial does not satisfy an unoriented skein 
relation, Manolescu \cite{Man07:HFKUnorientedSkein} shows that $\HFKh$ does 
satisfy an unoriented skein exact triangle over $\F{2}$, when $Y = S^3$. This 
is extended by the second author \cite{Won17:GHUnorientedSkein} to 
$\Z$-coefficients, using grid homology, a combinatorial version of link Floer 
homology.

This exact sequence has a number of consequences: First, quasi-alternating 
links are a large class of links in $S^3$ that include all alternating links. 
Manolescu's result immediately implies $\rk \HFKh (S^3, L) = 2^{\ell - 1} \det 
(L)$ for a quasi-alternating link $L$ with $\ell$ components, explaining a rank 
equality between $\HFKh$ and $\Kht$ for many small links in $S^3$. By 
calculating the shifts in the $\delta$-grading (obtained by diagonally 
collapsing the bigrading) in the exact sequence, Manolescu and Ozsv\'ath 
\cite{ManOzs08:QALinks} show that a quasi-alternating link $L$ is $\sigma$-thin 
over $\F{2}$, meaning that $\HFKh (S^3, L)$ is supported only in 
$\delta$-grading $\sigma (L) / 2$; this in turn implies that $\HFKh (S^3, L)$ 
is completely determined by the signature and the Alexander polynomial of $L$.  
By the second author's work on grid homology, this fact is also true over $\Z$.

Second, the exact triangle is iterated by Baldwin and Levine 
\cite{BalLev12:HFKUnorientedSkeinIterate} to obtain a cube-of-resolutions 
complex, giving a combinatorial description of link Floer homology that is 
distinct from grid homology.

Third, using the grid homology version of the exact triangle, Lambert-Cole 
\cite{Lam17:GHMutation} computes the rank of the maps involved, and uses this 
computation to prove that $\HFKh$ remains invariant under mutation by a large 
class of tangles. This mutation invariance applies to the Kinoshita--Terasaka 
and Conway families, as well as all mutant knots with crossing number at most 
$12$, giving a partial answer to the Mutation Invariance Conjecture 
\cite{BalLev12:HFKUnorientedSkeinIterate}.

The goal of the present paper is to generalize Manolescu's result to obtain an 
exact triangle for an unoriented skein triple of tangles in a $3$-manifold with 
boundary, in an appropriate sense. There are currently several theories related 
to Heegaard Floer homology that provide a suitable gluing theorem; in this 
paper, we focus on one such theory, bordered--sutured Floer homology, defined 
by Zarev \cite{Zar11:BSFH}.  Petkova and the second author 
\cite{PetWon18:TFHSkein} have proven a similar result for tangle Floer 
homology, a combinatorial tangle invariant defined by Petkova and V\'ertesi 
\cite{PetVer16:TFH}, similar to grid homology, for tangles in $S^2 \cross [0, 
1]$; meanwhile, Zibrowius \cite{Zib17:Dissertation} has given an alternative 
proof of Manolescu's result for links in $S^3$ using peculiar modules, which 
are invariants of tangles in $B^3$.  The exact triangle in the present paper 
applies to tangles in any $3$-manifold.

Sutured Floer homology, defined by Juh\'asz \cite{Juh06:SFH}, is a variant of 
Heegaard Floer homology \cite{OzsSza04:HF} for sutured manifolds. It recovers 
the hat-version of link Floer homology via the isomorphism
\[
  \SFH (Y \setminus \nbhd{L}, \sut_{\mu}) \isom \HFKh (Y, L),
\]
where $\sut_{\mu}$ denotes a pair of oppositely oriented, meridional sutures 
for each link component, at the expense of breaking a uniform homological 
grading into a grading for each $\SpinC$-summand. (For this reason, this will 
be our view of the gradings on $\HFKh$ for the rest of the paper.) This 
isomorphism also gives a definition of $\HFKh$ for links that are not 
rationally null-homologous.  Inspired by the bordered Heegaard Floer theory 
developed by Lipshitz, Ozsv\'ath, and Thurston \cite{LipOzsThu18:BFH, 
  LipOzsThu15:BFHBimodules}, Zarev \cite{Zar11:BSFH} defines an invariant of 
manifolds whose boundaries are partly sutured and partly bordered: One may then 
glue two manifolds together by identifying the bordered parts of their 
boundaries, and use a pairing theorem to calculate the invariant associated to 
the glued manifold. Since this theory is currently not defined over $\Z$, we 
shall work over $\F{2} = \cycgrp{2}$ throughout the paper.

For $k \in \set{\infty, 0, 1}$, let $\elTan_k = (\closure{B^3}, \elT_k)$ be the 
tangles shown in \fullref{fig:tangles}, which we shall refer to as 
\emph{elementary tangles}. Here, $\bdy \elT_k \subset \bdy \closure{B^3}$ are 
identical $0$-manifolds for $k \in \set{\infty, 0, 1}$.

\begin{figure}[!htbp]
  \captionsetup{aboveskip={\dimexpr10pt+2pt+\Lactualfontsize\relax}}
  \labellist
  \Large\hair 2pt
  \pinlabel $\elTan_1$ [t] at 38 0
  \pinlabel $\elTan_\infty$ [t] at 145 0
  \pinlabel {$\elTan_0$} [t] at 252 0
  \endlabellist
  \includegraphics{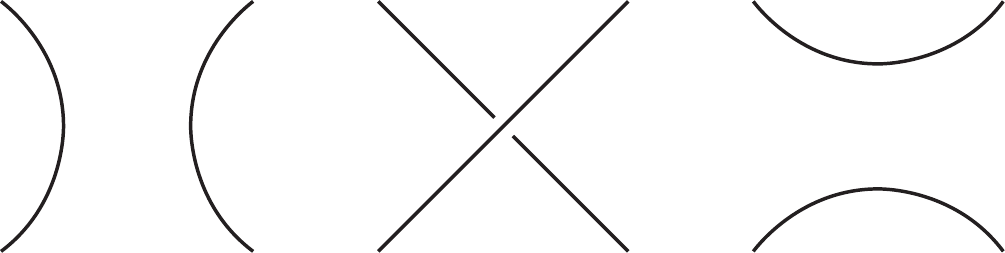}
  \caption{The elementary tangle $\elTan_\infty$ and its resolutions $\elTan_0$ 
    and $\elTan_1$.}
  \label{fig:tangles}
\end{figure}

Consider now a triple of more general tangles $\Tan_k = (Y, T_k)$, where
$Y$ is a compact, oriented $3$-manifold (possibly) with boundary and each 
$(T_k, \bdy T_k) \subset (Y, \bdy Y)$ is a smoothly embedded compact 
$1$-dimensional submanifold, such that
\begin{itemize}
  \item There is an embedded $\closure{B^3}$ in the interior of $Y$ in which 
    $T_k$ coincides with $\elT_k$; and
  \item For $k \in \set{\infty, 0, 1}$, the tangles $T_k$ are identical outside 
    $\closure{B^3}$.
\end{itemize}
Suppose each tangle complement $Y_k = Y \setminus \nbhd{T_k}$ is equipped with 
a bordered--sutured structure $(\sut_k, \arcdiag_1 \union \arcdiag_2, \phi_k)$, 
such that
\begin{itemize}
  \item The sutures in $\sut_k$ do not intersect $\closure{B^3}$, and are 
    identical
    for $k \in \set{\infty, 0, 1}$;
  \item The intersection of $\bdy Y_k$ and $\closure{B^3}$ belongs to the 
    positive region $R_+ (\sut)$; and
  \item The bordered boundary, i.e.\ the image of the parametrization $\phi_k 
    \colon G (\arcdiag_1 \union \arcdiag_2) \into \bdy Y_k$, does not intersect 
    $\closure{B^3}$, and the parametrization maps are identical
    for $k \in \set{\infty, 0, 1}$.
\end{itemize}
Note that either of $\arcdiag_1$ and $\arcdiag_2$ may be empty or not 
connected.  To each bordered--sutured manifold $\Y_k = (Y_k, \sut_k, \arcdiag_1 
\union \arcdiag_2, \phi_k)$, then, Zarev \cite{Zar11:BSFH} associates a 
type~$\DA$ bimodule
\[
  \itensor[^{\A (-\arcdiag_1)}]{\BSDAh (\Yk)}{_{\A (\arcdiag_2)}}.
\]
over $\A (-\arcdiag_1)$ and $\A (\arcdiag_2)$. A similar construction is used 
to obtain a bordered--sutured manifold $\Y_\sut (\Tan)$ by Alishahi and 
Lipshitz \cite{AliLip17:BFHCompressingDisks}, who prove that $\BSDh (\Y_\sut 
(\Tan))$ detects (partly) boundary parallel tangles. Our hypothesis differs 
from their construction in that we do not allow longitudinal sutures that 
intersect $\closure{B^3}$, but we do allow non--null-homologous $\Tan$ and 
place no restrictions on the sutures or the bordered boundaries outside 
$\closure{B^3}$.

\begin{maintheorem}
  \label{thm:main}
  Let $\Yk$ be as described above. There exist type~$\DA$ homomorphisms $F_k 
  \colon \BSDAh (\Yk) \to \BSDAh (\Ynext)$ such that
  \[
    \BSDAh (\Yk) \homeq \Cone (F_{k+1} \colon \BSDAh (\Ynext) \to \BSDAh 
    (\Yprev))
  \]
  as type~$\DA$ structures.  Moreover, the homomorphisms $F_k$ and the homotopy 
  equivalence above respect the relative gradings on the bimodules in a sense 
  to be made precise in \fullref{sec:gradings}; see \fullref{thm:gradings}.
\end{maintheorem}

Because of the more complicated nature of the gradings, we generally defer 
their discussions to \fullref{sec:gradings}.

The strategy to prove \fullref{thm:main} is to consider bordered--sutured 
manifolds $\Bk$ corresponding to the complements of the elementary tangles 
$\elTan_k$, first defining homomorphisms $f_k \colon \BSDh (\Bk) \to \BSDh 
(\Bnext)$, and then using the pairing theorem \cite[Theorem~8.5.1]{Zar11:BSFH} 
to extend the result to our setting. In fact, it is evident from our strategy 
that the skein relation above holds for more general bordered--sutured manifolds 
$\Yk$ and not just tangle complements; for example, one may derive from 
\fullref{thm:main} a skein relation for (a generalization of) graph Floer 
homology, defined by Harvey and O'Donnol \cite{HarODo17:HFG}, and independently 
by Bao \cite{Bao17:HFG}, for certain kinds of spatial graphs. The statement of 
\fullref{thm:main} is chosen to emphasize our interest in the application of 
bordered--sutured Floer theory to the study of knots, links, and tangles. See 
\fullref{thm:general} for the more general statement.

Our strategy illustrates and harnesses the power of a bordered theory, which 
allows one to use cut-and-paste techniques to prove general results by working 
locally.  For example, in a similar spirit, Lipshitz, Ozsv\'ath, and Thurston 
\cite{LipOzsThu14:BFHBranchedDoubleSSI, LipOzsThu16:BFHBranchedDoubleSSII} have 
used bordered Floer homology to compute explicitly the spectral sequence from 
Khovanov homology of a link to the Heegaard Floer homology of its branched 
double cover \cite{OzsSza05:HFBranchedDoubleSS}.

In particular, our strategy offers an advantage over the proofs in 
\cite{Man07:HFKUnorientedSkein, Won17:GHUnorientedSkein, PetWon18:TFHSkein}, as 
follows. In Manolescu's original paper \cite{Man07:HFKUnorientedSkein}, the 
maps involved in the exact sequence are defined by counting holomorphic 
polygons in a high-dimensional manifold, which are of theoretical interest but 
not practically computable.  In contrast, the maps in 
\cite{Won17:GHUnorientedSkein} are combinatorially defined, which is a feature 
specifically exploited in the application to mutation invariance 
\cite{Lam17:GHMutation}; the maps in \cite{PetWon18:TFHSkein} are likewise 
combinatorial. However, while we believe the maps in 
\cite{Won17:GHUnorientedSkein, PetWon18:TFHSkein} coincide with holomorphic 
polygon counts, it is not immediately obvious that they are so; moreover, the 
complexes involved there are fairly large.  The present paper unites the two 
approaches: We will first define $f_k$ by counts of holomorphic triangles in 
\fullref{sec:skein_via_polygon}, and then explicitly compute them 
combinatorially in \fullref{sec:skein_via_computation}, taking advantage of the 
fact that $\BSDh (\Bk)$ have at most three generators with our choice of 
bordered--sutured Heegaard diagrams $\HD_k$. The reader is encouraged to take a 
cursory look at \fullref{fig:comm_diag_f} and \fullref{fig:comm_diag_phi}, 
which contain the entire computation.

A special case of \fullref{thm:main} is when $Y$ is a compact $3$-manifold 
\emph{without} boundary, with $\arcdiag_1 = \arcdiag_2 = \eset$.
In this case, each $T_k$ is an embedded $1$-dimensional submanifold without 
boundary, and is hence a link in $Y$; we thus denote it more conventionally by 
$L_k$. Let $\ell_k$ be the number of components of $L_k$, and set $m = \max 
\set{\ell_\infty, \ell_0, \ell_1}$.

\begin{maincorollary}
  \label{cor:arbitrary}
  There exists an exact triangle
  \[
    \dotsb \xrightarrow{(F_1)_*} \HFKh (Y, L_\infty) \tensor V^{m - 
      \ell_\infty} \xrightarrow{(F_\infty)_*} \HFKh (Y, L_0) \tensor V^{m - 
      \ell_0} \xrightarrow{(F_0)_*} \HFKh (Y, L_1) \tensor V^{m - \ell_1}
    \xrightarrow{(F_1)_*} \dotsb,
  \]
  where $V$ is a vector space of dimension $2$. Moreover, the maps involved 
  respect the relative gradings on the homology groups in a sense to be made 
  precise in \fullref{sec:gradings}; see \fullref{thm:gradings}.
\end{maincorollary}

The reader familiar with gradings in sutured Floer homology may find it useful 
to examine \fullref{eg:iden_fibers_2} for an illustration of the graded version 
of \fullref{cor:arbitrary}.

The exact triangle in \cite{Man07:HFKUnorientedSkein} is the special case $Y = 
S^3$. In \cite{Man07:HFKUnorientedSkein}, Manolescu constructs, from a given 
connected link projection of $L_\infty \subset S^3$ and a choice of edges in 
the projection, a \emph{special Heegaard diagram} $\HDMan_\infty$ for 
$L_\infty$, modifies $\HDMan_\infty$ locally to obtain Heegaard diagrams 
$\HDMan_0$ and $\HDMan_1$ for $L_0$ and $L_1$, and uses these Heegaard diagrams 
to prove the skein exact triangle.  Here, we have added a superscript to 
distinguish these Heegaard diagrams defined in \cite{Man07:HFKUnorientedSkein} 
from those constructed in the present article; likewise, we denote by $\FMank$ 
the maps involved in the skein exact triangle, which are denoted $f_k$ in 
\cite{Man07:HFKUnorientedSkein}.  

Note that one might be able to prove \fullref{cor:arbitrary} by directly 
generalizing Manolescu's construction; however, that would require constructing 
Heegaard diagrams of a given form, for all $(Y, L_\infty)$, which could be 
somewhat cumbersome before the advent of bordered--sutured Floer theory. (See, 
for contrast, the discussion after \fullref{thm:comm_spec}.)

\begin{maintheorem}
  \label{thm:agree}
  For $Y = S^3$, the exact triangle in \fullref{cor:arbitrary} agrees with the 
  exact triangle in \cite{Man07:HFKUnorientedSkein}. See \fullref{thm:agree_2} 
  for the precise statement.
\end{maintheorem}

In \cite{ManOzs08:QALinks} (and in \cite{Won17:GHUnorientedSkein, 
  PetWon18:TFHSkein}), the exact triangles are equipped with an absolute 
$\delta$-grading in $\Z$.  While we expect the grading in 
\fullref{cor:arbitrary} to be related to this $\delta$-grading when applied to 
$Y = S^3$, we do not prove this in this paper.  See, however, 
\fullref{rmk:gr_01}.

\fullref{thm:agree} is inspired by \cite{LipOzsThu16:BFHBranchedDoubleSSII}, in 
which the theory of bordered polygon counts is developed to prove that the 
exact triangle, and in fact the spectral sequence, constructed in 
\cite{LipOzsThu14:BFHBranchedDoubleSSI} using bordered Floer homology, agree 
with the original exact triangle and spectral sequence constructed directly 
using Heegaard Floer homology \cite{OzsSza05:HFBranchedDoubleSS}.

\subsection*{Organization}
The rest of the paper is organized as follows. In \fullref{sec:preliminaries}, 
we fix some notation and provide a lemma in homological algebra that will be 
useful in later sections.  Then, in \fullref{sec:setup}, we describe the 
bordered boundary and the bordered--sutured manifolds $\Bk$, which correspond 
to the complements of the elementary tangles $\elT_k$, and explicitly compute 
the algebra and modules associated to these objects. Stating and assuming the 
specialization of \fullref{thm:main} to $\Bk$, we complete the general proof of 
\fullref{thm:main} and \fullref{cor:arbitrary}. We provide two proofs of the 
specialization. We first give a theoretical proof by holomorphic polygon counts 
in \fullref{sec:skein_via_polygon}, in the process adapting the technique of 
polygon counts from \cite{LipOzsThu16:BFHBranchedDoubleSSII} to the 
bordered--sutured setting. Then, in \fullref{sec:skein_via_computation}, we 
give a combinatorial proof, with explicit computations, and show that the two 
proofs are equivalent.  We then compare our maps with Manolescu's maps 
\cite{Man07:HFKUnorientedSkein} and establish \fullref{thm:agree} in 
\fullref{sec:comparison}. Finally, in \fullref{sec:gradings}, we discuss the 
gradings in bordered--sutured Floer theory in some detail and prove a graded 
version of \fullref{thm:main}.

\subsection*{Acknowledgements}
The authors would like to thank Robert Lipshitz for helpful discussions.

\section{Algebraic preliminaries}
\label{sec:preliminaries}

We shall assume that the reader is familiar with the formal algebraic 
structures in the bordered Floer package \cite{LipOzsThu18:BFH, 
  LipOzsThu15:BFHBimodules} and the bordered--sutured Floer package 
\cite{Zar11:BSFH}. (For the reader familiar with the former but not the latter, 
the algebraic structures involved are the same.) These include type~$D$ and 
type~$A$ modules, bimodules of various types, morphisms of modules and their 
boundaries and compositions, and the box tensor operation of modules. In 
general, we follow the notation in \cite{Zar11:BSFH}, which may differ from the 
notation in \cite{LipOzsThu18:BFH}; see, however, \fullref{rmk:order}.

In the following, we will always use the calligraphic typeface (e.g.\ $\M$ and 
$\module{N}$) to denote an unspecified type~$D$ or type~$A$ module over some 
algebra $\alg{A}$, and reserve the regular italic typeface (e.g.\ $M$ and $N$) 
for the underlying $\k$-module.  However, we will denote both the type~$D$ 
bordered--sutured Floer module over the algebra $\AZ$, and the underlying 
$\k$-module, by $\BSDh (\Y)$. We will also distinguish between the identity 
type~$D$ homomorphism $\Id_{\M}$ of $\M$ and the identity $\k$-module 
isomorphism $\id_{M}$ of $M$.

Given a type~$D$ morphism $f \colon \M \to \module{N}$, let $\bdy f$ denote its 
boundary in $\Mor^{\A} (\M, \module{N})$, given by
\[
  \bdy f = (\mu \tensor \id_N) \comp (\id_A \tensor \delta_N) \comp f + (\mu 
  \tensor \id_N) \comp (\id_A \tensor f) \comp \delta_M + (d \tensor \id_N) 
  \comp f.
\]
For convenience, we sometimes denote by $d f$ the term $(d \tensor \id_N) \comp 
f$.

The proof of \fullref{thm:main} uses a lemma in homological algebra whose 
version for chain complexes first appeared in 
\cite{OzsSza05:HFBranchedDoubleSS}. Let $\A = (A, d, \mu)$ be a differential 
graded algebra, and let $\M$ and $\module{N}$ be type~$D$ modules over $\A$.

\begin{lemma}
  \label{lem:hom_alg}
  Let $\M_k = \set{(M_k, \delta_{M_k})}_{k \in \set{\infty, 0, 1}}$ be a 
  collection of type~$D$ modules over a unital differential graded algebra $\A 
  = (A, d, \mu)$ over a base ring $\k$ of characteristic $2$, and let $f_k 
  \colon \M_k \to \M_{k+1}$, $\vphi_k \colon \M_k \to \M_{k+2}$, and $\vkappa_k 
  \colon \M_k \to \M_k$ be morphisms satisfying the following conditions for 
  each $k$:
  \begin{enumerate}
    \item The morphism $f_k \colon \M_k\to \M_{k+1}$ is a type~$D$ 
      homomorphism, i.e.
      \[
        \partial f_k = 0;
      \]
    \item The morphism $f_{k+1} \circ f_k$ is homotopic to zero via homotopy 
      $\vphi_k$, i.e.
      \[
        f_{k+1} \circ f_k + \partial \vphi_k = 0;
      \]
    \item The morphism $f_{k+2} \circ \vphi_k + \vphi_{k+1} \circ f_k$ is 
      homotopic to the identity $\Id_k$ via homotopy $\vkappa_k$, i.e.
      \[
        f_{k+2} \circ \vphi_k + \vphi_{k+1} \circ f_k + \partial \vkappa_k = 
        \Id_k.
      \]
  \end{enumerate}
  A graphical representation of the conditions above is given in 
  \fullref{fig:hom_alg}. Then for each $k$, the type~$D$ modules $\M_k$ is 
  homotopy equivalent to the mapping cone $\Cone (f_{k+1})$.
  \begin{figure}
    \begin{gather}
      \mathcenter{
        \includegraphics{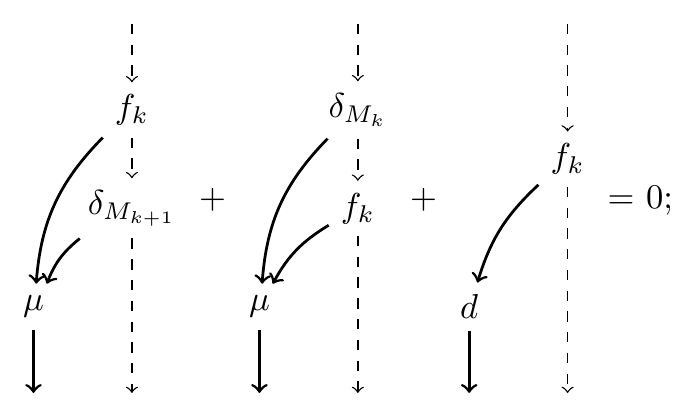}\tag{\emph{(1)}}
      }\\
      \mathcenter{
        \includegraphics{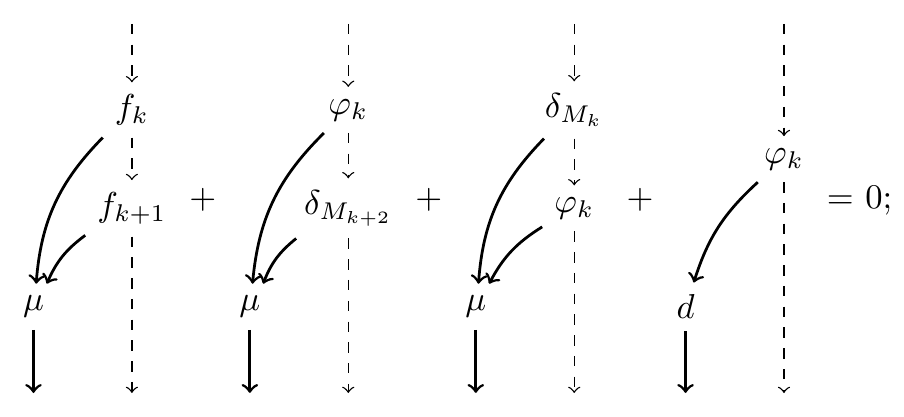}\tag{\emph{(2)}}
      }\\
      \mathcenter{
        \includegraphics{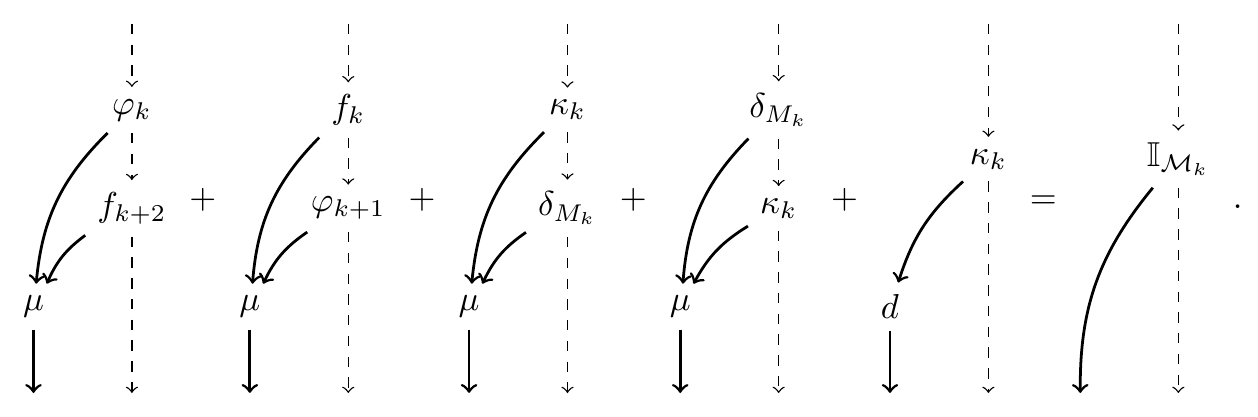}\tag{\emph{(3)}}
      }
    \end{gather}
    \caption{Graphical representations of the conditions in 
      \fullref{lem:hom_alg}.}
    \label{fig:hom_alg}
  \end{figure}
\end{lemma}

\begin{proof}
  This is a special case of \cite[Lemma~2.1]{PetWon18:TFHSkein}, where $B = 
  \F{2}$ is the trivial algebra.
\end{proof}

\section{The bordered--sutured Floer package of elementary tangle complements} 
\label{sec:setup}

In this section, we explicitly compute the differential graded algebra and 
type~$D$ modules that we will be working with in the rest of the paper. For the 
sake of economy, we do not provide the definitions of the algebras and modules 
in bordered--sutured Floer theory, but direct the reader to \cite{Zar11:BSFH}.

\subsection{The algebra associated to a 4-punctured sphere}
\label{ssec:4puncS2}

Let $F$ be the $4$-punctured sphere, and let $\sutF$ be the sutured surface 
$(F, \Lambda)$, where $\Lambda$ consists of $2$ distinct points on each 
component of $\bdy F$. In other words, each boundary circle of $F$ is divided 
into a positive and a negative arc.

In the context of tangles, $\sutF$ will be the bordered part of the boundary of 
a tangle complement, along which another bordered--sutured manifold can be 
glued. Thus, our first task is to parametrize $\sutF$ by an arc diagram $- 
\arcdiag$. (The orientation reversal is appropriate for type~$D$ structures.)

Let $\orarcs = \set{Z_1, Z_2, Z_3, Z_4}$ be a collection of oriented arcs, and let 
$\matchedpts = \set{a_1, \dotsc, a_{12}}$ be a collection of distinct points in $\orarcs$, 
such that
\[
  a_1 \in Z_1, \quad a_2, a_3, a_4, a_5 \in Z_2, \quad a_6, a_7, a_8, a_9 \in 
  Z_3, \quad a_{10}, a_{11}, a_{12} \in Z_4,
\]
in this order, if we traverse the arcs $Z_i$ according to their orientations.  
Let $M \colon \matchedpts \to \set{1, \dotsc, 6}$ be the matching
\begin{align*}
  M (a_1) & = M (a_3) = 1, & M (a_2) & = M (a_5) = 2, & M (a_4) & = M (a_7) = 
  3,\\
  M (a_6) & = M (a_9) = 4, & M (a_8) & = M (a_{11}) = 5, & M (a_{10}) & = M 
  (a_{12}) = 6.
\end{align*}
If we define $\arcdiag = (\orarcs, \matchedpts, M)$, then $- \arcdiag$ 
parametrizes $\sutF (- \arcdiag) = \sutF$, as in \fullref{fig:arc_diag}.

\begin{figure}[!htbp]
	\labellist
	\small\hair 2pt
	\pinlabel {\scriptsize \textcolor{red}{$a_1$}} [r] at 28 10
	\pinlabel {\scriptsize \textcolor{red}{$a_2$}} [r] at 28 35
	\pinlabel {\scriptsize \textcolor{red}{$a_3$}} [r] at 28 45
	\pinlabel {\scriptsize \textcolor{red}{$a_4$}} [r] at 28 55
	\pinlabel {\scriptsize \textcolor{red}{$a_5$}} [r] at 28 64
	\pinlabel {\scriptsize \textcolor{red}{$a_6$}} [r] at 28 85
	\pinlabel {\scriptsize \textcolor{red}{$a_7$}} [r] at 28 95
	\pinlabel {\scriptsize \textcolor{red}{$a_8$}} [r] at 28 104
	\pinlabel {\scriptsize \textcolor{red}{$a_9$}} [r] at 28 114
	\pinlabel {\scriptsize \textcolor{red}{$a_{10}$}} [r] at 28 137
	\pinlabel {\scriptsize \textcolor{red}{$a_{11}$}} [r] at 28 148
	\pinlabel {\scriptsize \textcolor{red}{$a_{12}$}} [r] at 28 159
	
	\pinlabel {\footnotesize $\rho_1$} [r] at 21 39
	\pinlabel {\footnotesize $\rho_2$} [r] at 21 50
	\pinlabel {\footnotesize $\rho_3$} [r] at 21 60
	\pinlabel {\footnotesize $\rho_4$} [r] at 21 90
	\pinlabel {\footnotesize $\rho_5$} [r] at 21 100
	\pinlabel {\footnotesize $\rho_6$} [r] at 21 109
	\pinlabel {\footnotesize $\rho_7$} [r] at 21 142
	\pinlabel {\footnotesize $\rho_8$} [r] at 21 153
	
	\pinlabel {\large $Z_1$} [r] at 1 10
	\pinlabel {\large $Z_2$} [r] at 1 50
	\pinlabel {\large $Z_3$} [r] at 1 100
	\pinlabel {\large $Z_4$} [r] at 1 148

	\pinlabel \textcolor{red}{$e_1$} [r] at 62 26
	\pinlabel \textcolor{red}{$e_2$} [r] at 62 50
	\pinlabel \textcolor{red}{$e_3$} [r] at 62 75
	\pinlabel \textcolor{red}{$e_4$} [r] at 62 100
	\pinlabel \textcolor{red}{$e_5$} [r] at 62 125
	\pinlabel \textcolor{red}{$e_6$} [r] at 62 148
	
  \pinlabel {\Large $\cong$} [r] at 153 78
  \endlabellist
  \mbox {\phantom{\large $Z_1$\!\!}}
  \includegraphics{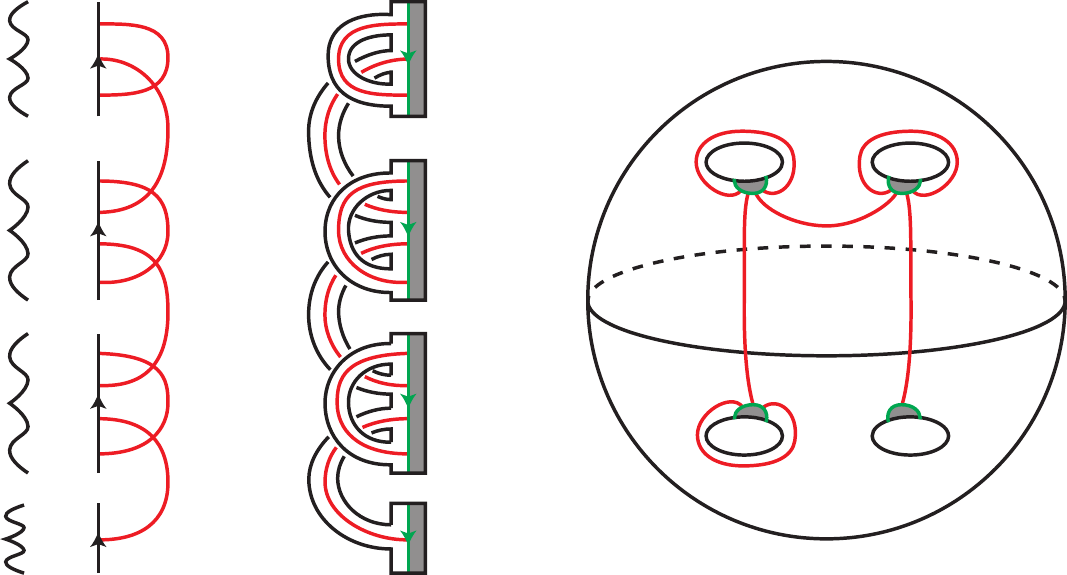}
  \caption{Left: The arc diagram $\arcdiag$. Center and right: The sutured 
    surface $\sutF (-\arcdiag) = (F, \Lambda)$, where $F$ is a $4$-punctured 
    sphere.}
  \label{fig:arc_diag}
\end{figure}
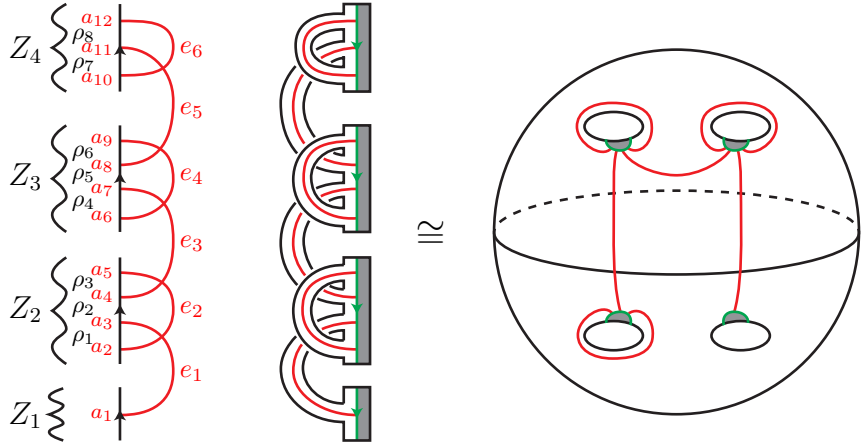

In bordered--sutured Floer theory, an algebra $\AZ$ is associated to the arc 
diagram $\arcdiag$. Each generator of $\AZ$ corresponds to a \emph{strands 
  diagram}.  Multiplication is defined by concatenation of strands diagrams 
with the convention that the product is zero if they cannot be concatenated.  
The differential is given by the sum of all possible ways to resolve the 
crossings in a given strands diagram. In both operations, there is an 
additional condition that the resulting strands diagram must not have double 
crossings between any two strands; any offending strands diagram is set to be 
zero.

The algebra $\AZ$ decomposes into a direct sum
\[
  \AZ = \bigdirsum_{i=0}^6 \AZi,
\]
corresponding to the number of \emph{occupied} arcs (for type~$A$ modules) or 
\emph{unoccupied} arcs (for type~$D$ modules). We will be working with type~$D$ 
modules defined by bordered--sutured diagrams with no $\alpha$-circles and 
exactly one $\beta$-circle, and so we will always have five unoccupied arcs.  
Thus, for our purposes, it is sufficient to consider the $5$-summand $\AZfive$. 

The algebra $\AZfive$ has six idempotents $I_{12345}$, $I_{12346}$, 
$I_{12356}$, $I_{12456}$, $I_{13456}$, and $I_{23456}$; we denote the set of 
these idempotents by $\IZfive$.  To simplify notation, we will instead denote 
these by $\Iocc{6}$, $\Iocc{5}$, $\Iocc{4}$, $\Iocc{3}$, $\Iocc{2}$, and 
$\Iocc{1}$ respectively, with the subscripts indicating the \emph{occupied} arc 
rather than the \emph{unoccupied} arcs. We will denote by $I$ the sum $\Iocc{1} 
+ \dotsb + \Iocc{6}$.

In $\arcdiag$, there are $15$ Reeb chords: $\rho_1$ from $a_2$ to $a_3$, 
$\rho_2$ from $a_3$ to $a_4$, $\rho_3$ from $a_4$ to $a_5$, and their 
concatenations $\rho_{12}$, $\rho_{23}$, and $\rho_{123}$; $\rho_4$ from $a_6$ 
to $a_7$, $\rho_5$ from $a_7$ to $a_8$, $\rho_6$ from $a_8$ to $a_9$, and their 
concatenations $\rho_{45}$, $\rho_{56}$, and $\rho_{456}$; $\rho_7$ from 
$a_{10}$ to $a_{11}$, $\rho_8$ from $a_{11}$ to $a_{12}$, and their 
concatenation $\rho_{78}$. Therefore, we define the index set
\[
  J' = \set{1, 2, 3, 12, 23, 123, 4, 5, 6, 45, 56, 456, 7, 8, 78}.
\]
In fact, we will soon be working with a $12$-element subset $J \subset J'$, 
defined by
\[
  J = \setc{j \in J'}{M (\rho_j^-) \neq M (\rho_j^+)} = \set{1, 2, 3, 12, 23, 
    4, 5, 6, 45, 56, 7, 8}.
\]
Here, $\rho^-$ and $\rho^+$ denote the starting and ending points of $\rho$ 
respectively. Similarly, given a collection $\rhos$ of Reeb chords, we use the 
notation $\rhos^- = \setc{\rho_j^-}{\rho_j \in \rhos}$
and $\rhos^+ = \setc{\rho_j^+}{\rho_j \in \rhos}$.

Recall from \cite[Definition~2.3.1]{Zar11:BSFH} the definition of a 
$p$-completion of a collection $\rhos$ of Reeb chords. Let $\abs{\rhos} = n 
\leq 5$. Since we are working with $\AZfive$, we are interested in the 
$5$-completions of $\rhos$; in our context, such a $5$-completion is a choice 
of a $(5-n)$-element subset of $\set{1, \dotsc, 6}$, corresponding to a choice 
of $5-n$ unoccupied arcs not labeled by $M (\rhos^-)$ or $M (\rhos^+)$.  As 
explained in \cite{Zar11:BSFH}, every $5$-completion $s$ of $\rhos$ defines an 
element $a (\rhos, s) \in \AZfive$, and the \emph{associated element} of 
$\rhos$ in $\AZfive$ is
\[
  a_5 (\rhos) = \sum_{s \text{ a $5$-completion of } \rhos} a (\rhos, s).
\]
The algebra $\AZfive$ is then generated over $\IZfive$ by the elements $a_5 
(\rhos)$ for all $5$-completable $\rhos$.

Since we will be working with the type~$D$ invariant, we would like to draw our 
attention to the special case where $\rhos = \set{\rho_j}$ for some fixed $j 
\in J'$. To simplify our notation, we denote the associated element of 
$\set{\rho_j}$ by
\[
  (j) = a_5 (\set{\rho_j}).
\]
More generally, for $j_1, \dotsc, j_\ell \in J'$, we let
\[
  (j_1, \dotsc, j_\ell) = a_5 (\set{\rho_{j_1}}) \dotsm a_5 
  (\set{\rho_{j_\ell}}),
\]
the product of these associated elements. We will be using this notation 
throughout the rest of this paper.

\begin{remark}
  \label{rmk:no_repeat}
  It will be helpful for us to observe that, if $(j_1, \dots, j_\ell) \neq 0$, 
  then each of $1, 3, 4, 6, 7, 8$ can appear in $(j_1, \dotsc, j_\ell)$ at most 
  once, not only as indices in $J'$, but in fact as actual digits. For example, 
  we have $(45, 4) = 0$; here, even though the index $4 \in J'$ only appears 
  once, the digit $4$ appears twice. To verify this claim, suppose the digit 
  $4$ appears more than once. Writing the product $(j_1, \dotsc, j_\ell)$ as 
  $\Istart \cdot a_5 (\rhos) \cdot \Iend$, where $\Istart$ and $\Iend$ are 
  idempotents, we see that the coefficient of $[\rho_4]$ in the homology class 
  $[\rhos] \in H_1 (\orarcs, \matchedpts)$ is greater than $1$. This means that 
  $\rho_j^- = \rho_{j'}^- = \rho_4^-$ for at least two Reeb chords $\rho_j$ and 
  $\rho_{j'}$ in $\rhos$, and hence that $\rhos$ is not $5$-completable. The 
  same analysis works for the digits $1$ and $7$, while a similar argument 
  involving $\rho_j^+$ works for $3$, $6$, and $8$.

  An easy generalization of the argument shows that the digits $2$ and $5$ can 
  appear in $(j_1, \dotsc, j_\ell)$ at most twice.
\end{remark}

Suppose further that $\rhos = \set{\rho_j}$ for some $j \in J$. For such 
$\rhos$, observe that since $M (\rho_j^-) \neq M (\rho_j^+)$, there is a unique 
$4$-element subset of $\set{1, \dotsc, 6}$ disjoint from $\set{M (\rho_j^-), M 
  (\rho_j^+)}$; in other words, $\set{\rho_j}$ has a unique $5$-completion 
$s_j$. Thus, we see that
\[
  (j) = a_5 (\set{\rho_j}) = a (\set{\rho_j}, s_j).
\]
Unless otherwise stated, from now on, we will work only with the index set $J$ 
rather than $J'$; see \fullref{rmk:no_full_reeb} (and \fullref{rmk:no_repeat}) 
for some justification.

It is evident that our notation gives a compact way to represent multiplication 
of many, but not all, algebra elements. (As we shall see, this will be 
sufficient.) Note that, however, for a given element $a \in \AZfive$, there may 
be more than one way to represent $a$ in this notation.  For example, 
\fullref{fig:alg_mult} illustrates the fact that
\[
  (12, 3, 2) = a_5 (\set{\rho_{12}}) a_5 (\set{\rho_3}) a_5 
  (\set{\rho_2}) = a_5 (\set{\rho_{123}, \rho_2}) = a_5 (\set{\rho_2}) 
  a_5 (\set{\rho_1}) a_5 (\set{\rho_{23}}) = (2, 1, 23).
\]
A similar calculation shows that $(45, 6, 5) = (5, 4, 56)$.

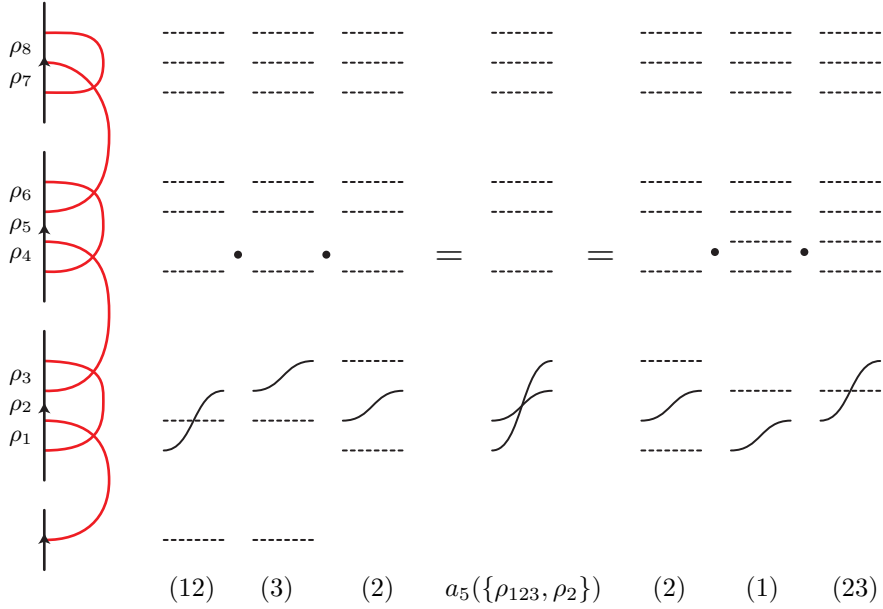
\begin{figure}[!htbp]
	\captionsetup{aboveskip={\dimexpr10pt+2pt+\sactualfontsize\relax}}
  \labellist
  \small\hair 2pt
	
  \pinlabel {\footnotesize $\rho_1$} [r] at 0 50
  \pinlabel {\footnotesize $\rho_2$} [r] at 0 62
  \pinlabel {\footnotesize $\rho_3$} [r] at 0 73
  \pinlabel {\footnotesize $\rho_4$} [r] at 0 119
  \pinlabel {\footnotesize $\rho_5$} [r] at 0 130
  \pinlabel {\footnotesize $\rho_6$} [r] at 0 142
  \pinlabel {\footnotesize $\rho_7$} [r] at 0 185
  \pinlabel {\footnotesize $\rho_8$} [r] at 0 197	

  \pinlabel {\Large $=$} at 155 118
  \pinlabel {\Large $=$} at 212 118

  \pinlabel $(12)$ [t] at 58 0
  \pinlabel $(3)$ [t] at 90 0
  \pinlabel $(2)$ [t] at 128 0
  \pinlabel $a_5(\set{\rho_{123},\rho_2})$ [t] at 183 0
  \pinlabel $(2)$ [t] at 238 0
  \pinlabel $(1)$ [t] at 273 0
  \pinlabel $(23)$ [t] at 308 0

  \endlabellist
  \mbox{\phantom{\footnotesize $\rho_1$\!}}
  \includegraphics{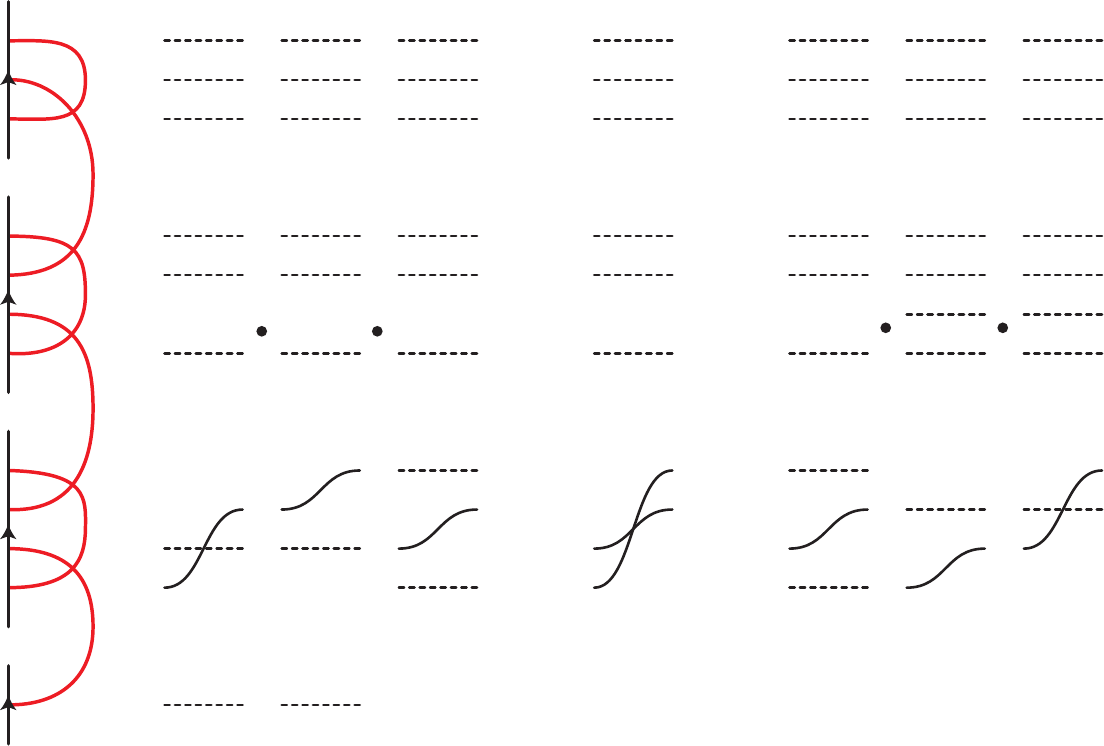}
  \caption{In our notation, $(12, 3, 2)$ and $(2, 1, 23)$ both represent the 
    same algebra element, $a_5 (\set{\rho_{123}, \rho_2})$.}
  \label{fig:alg_mult}
\end{figure}

Consider now an element $(j_1, \dotsc, j_\ell) \in \AZfive$, with $j_1, \dotsc, 
j_\ell \in J$.  For such an element, our notation gives a simple algorithm that 
computes its differential: It is the sum of all possible ways to insert a comma 
(such that the resulting notation makes sense) and exchange the indices 
adjacent to the new comma.  For example, \fullref{fig:alg_diff} illustrates the 
fact that
\[
  d (12, 3, 4, 56, 2) = (2, 1, 3, 4, 56, 2) + (12, 3, 4, 6, 5, 2).
\]
This observation follows from the Leibniz rule.

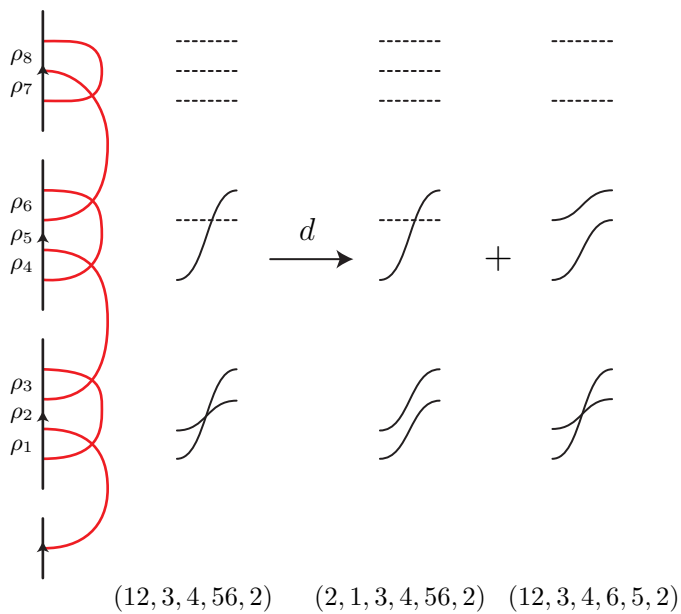
\begin{figure}[!htbp]
	\captionsetup{aboveskip={\dimexpr10pt+2pt+\sactualfontsize\relax}}

  \labellist \small\hair 2pt
  \pinlabel {\footnotesize $\rho_1$} [r] at 1 50
  \pinlabel {\footnotesize $\rho_2$} [r] at 1 62
  \pinlabel {\footnotesize $\rho_3$} [r] at 1 73
  \pinlabel {\footnotesize $\rho_4$} [r] at 1 118
  \pinlabel {\footnotesize $\rho_5$} [r] at 1 129
  \pinlabel {\footnotesize $\rho_6$} [r] at 1 141
  \pinlabel {\footnotesize $\rho_7$} [r] at 1 185
  \pinlabel {\footnotesize $\rho_8$} [r] at 1 197

  \pinlabel {\large $d$} at 102 132
  \pinlabel {\Large $+$} at 174 121

  \pinlabel $(12,3,4,56,2)$ [t] at 59 0
  \pinlabel $(2,1,3,4,56,2)$ [t] at 137 0
  \pinlabel $(12,3,4,6,5,2)$ [t] at 210 0

  \endlabellist
  \mbox{\phantom{\footnotesize $\rho_8$\!\!}}
  \includegraphics{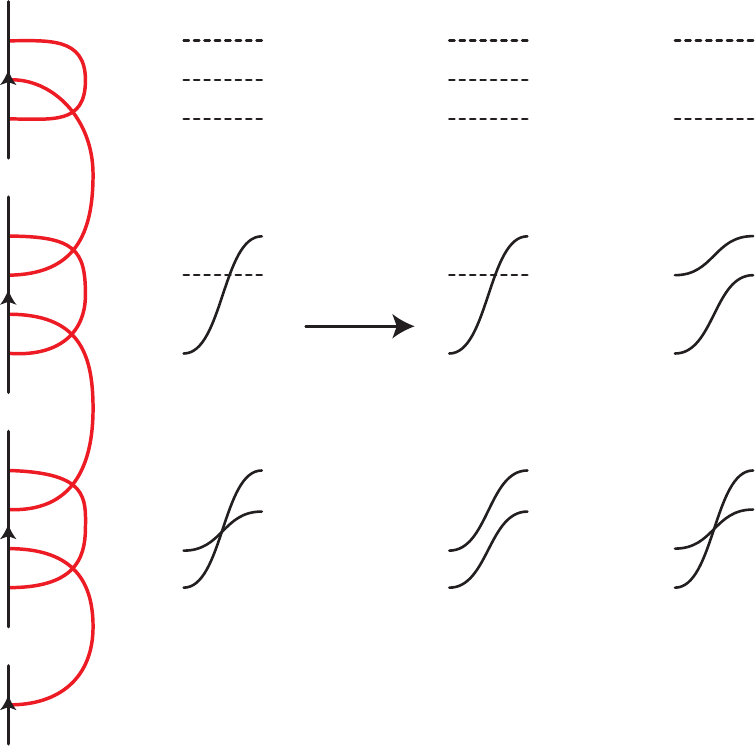}
  \mbox{\phantom{\small \!$,5,2)$}}
  \caption{The differential of the algebra element $(12, 3, 4, 56, 2)$.}
  \label{fig:alg_diff}
\end{figure}

\subsection{The modules associated to the three tangle complements}
\label{ssec:tancomp}

For $k \in \set{\infty, 0, 1}$, consider the tangle complement $B_k = B^3 
\setminus \nbhd{\elT_k}$. Let $\Bk$ be the bordered--sutured manifold $(B_k, 
\mersut_k, - \arcdiag, \elparam_k)$, where $\mersut_k$ consists of two pairs of 
oppositely oriented meridional sutures, one for each arc of $\elT_k$, and 
$\elparam_k \colon G (- \arcdiag) \into \bdy B_k$ gives the parametrization of 
the $4$-punctured sphere $\bdy B_k \intersect \bdy B^3$ shown in 
\fullref{fig:tancomp}.  In other words, $\elparam_k$ allows us to identify 
$\sutF (- \arcdiag)$ with the sutured surface $\paren{\bdy B_k \intersect \bdy 
  B^3, \mersut_k \intersect \bdy B^3 \intersect \closure{\nbhd{\elT_k}}}$ that 
constitutes the bordered part of the boundary, and we will make this 
identification from now on. The ``sutured'' part of the boundary, consisting of 
the two inner cylinders, is divided by $\mersut_k$ into one positive and two 
negative regions.  

\begin{figure}[!htbp]
  \captionsetup{aboveskip=\dimexpr10pt+5pt+\Lactualfontsize\relax}
  \labellist
  \Large\hair 5pt

  \pinlabel $\B_1$ [t] at 58 0
  \pinlabel $\B_\infty$ [t] at 191 0
  \pinlabel $\B_0$ [t] at 320 0

  \endlabellist
  \includegraphics{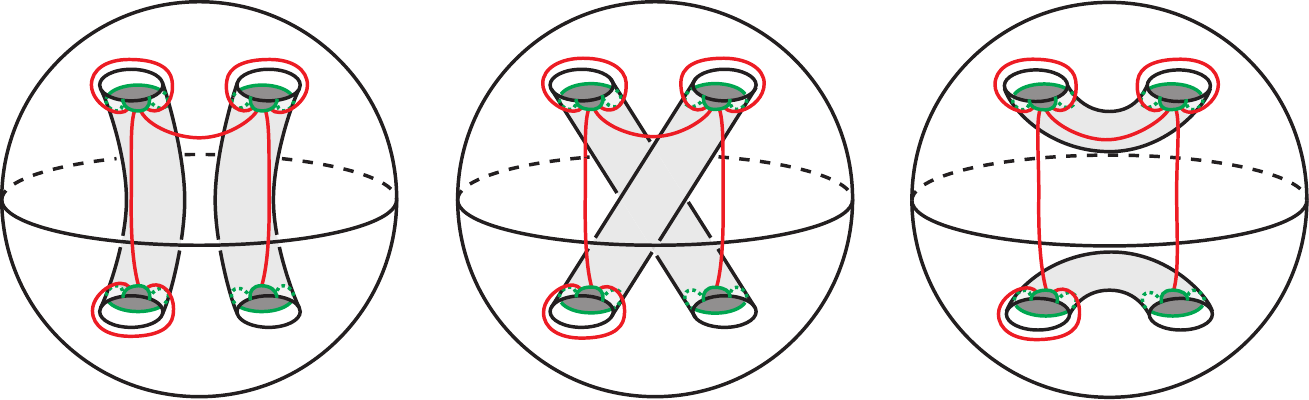}
  \caption{The bordered--sutured manifolds $\Bk$.}
  	\label{fig:tancomp}
\end{figure}

Below, we will compute the type~$D$ modules $\BSDh (\Bk)$. Recall that for a 
bordered--sutured manifold $\Y = (Y, \Gamma, \arcdiag, \phi)$, $\BSDh (\Y)$ is a 
type~$D$ module over $\alg{A} (- \arcdiag)$. Therefore, $\BSDh (\Bk)$ is a 
type~$D$ module over $\AZ$.

\begin{figure}[!htbp]
  \begin{subfigure}[b]{\textwidth}
    \captionsetup{aboveskip={\dimexpr10pt+2pt+\sactualfontsize\relax}, 
      belowskip={\floatsep}}
    \centering
    \labellist
    \scriptsize\hair 2pt
    \pinlabel \textcolor{red}{4} [t] at 65 1
    \pinlabel \textcolor{red}{5} [t] at 82 1
    \pinlabel \textcolor{red}{3} [t] at 109 1
    \pinlabel \textcolor{red}{4} [t] at 124 1
    \pinlabel \textcolor{red}{2} [t] at 270 1
    \pinlabel \textcolor{red}{3} [t] at 284 1
    \pinlabel \textcolor{red}{1} [t] at 311 1
    \pinlabel \textcolor{red}{2} [t] at 327 1
    \pinlabel \textcolor{red}{6} at 43 48
    \pinlabel \textcolor{red}{5} at 53 47
    \pinlabel \textcolor{red}{6} at 63 48
    \pinlabel \textcolor{red}{1} at 341 47

    \pinlabel {\small 6} [t] at 73 0
    \pinlabel {\small 5} [t] at 95 0
    \pinlabel {\small 4} [t] at 117 0
    \pinlabel {\small 3} [t] at 277 0
    \pinlabel {\small 2} [t] at 298 0
    \pinlabel {\small 1} [t] at 320 0
    \pinlabel {\small 7} at 48 49
    \pinlabel {\small 8} at 58 49

    \pinlabel $x$ at 146 32

    \endlabellist
    \includegraphics[scale=1.0]{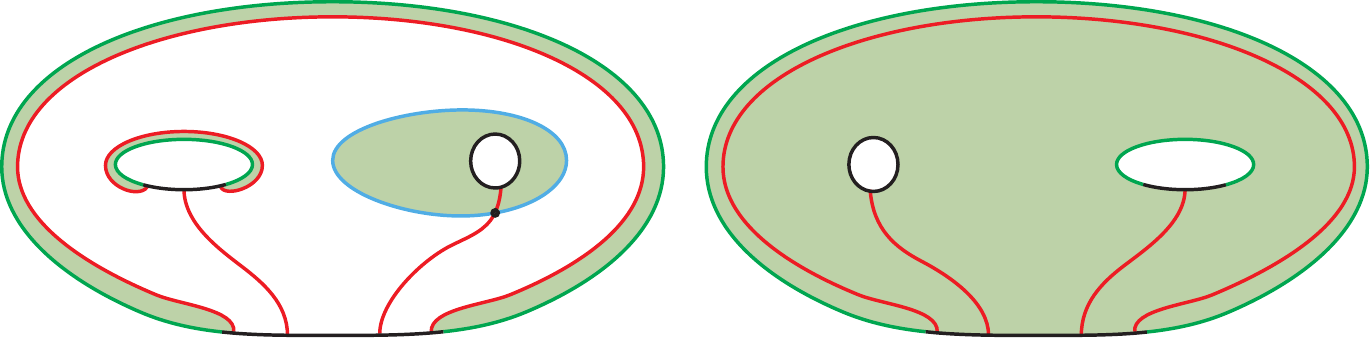}
    \caption{The bordered--sutured diagram $\HD_1$.}
    \label{fig:hd_one}
  \end{subfigure}
  \begin{subfigure}[b]{\textwidth}
    \captionsetup{aboveskip={\dimexpr10pt+2pt+\sactualfontsize\relax}, 
      belowskip={\floatsep}}
    \centering
    \labellist
    \scriptsize\hair 2pt
    \pinlabel \textcolor{red}{4} [t] at 65 1
    \pinlabel \textcolor{red}{5} [t] at 82 1
    \pinlabel \textcolor{red}{3} [t] at 109 1
    \pinlabel \textcolor{red}{4} [t] at 124 1
    \pinlabel \textcolor{red}{2} [t] at 270 1
    \pinlabel \textcolor{red}{3} [t] at 284 1
    \pinlabel \textcolor{red}{1} [t] at 311 1
    \pinlabel \textcolor{red}{2} [t] at 327 1
    \pinlabel \textcolor{red}{6} at 43 48
    \pinlabel \textcolor{red}{5} at 53 47
    \pinlabel \textcolor{red}{6} at 63 48
    \pinlabel \textcolor{red}{1} at 341 47

    \pinlabel {\small 6} [t] at 73 0
    \pinlabel {\small 5} [t] at 95 0
    \pinlabel {\small 4} [t] at 117 0
    \pinlabel {\small 3} [t] at 277 0
    \pinlabel {\small 2} [t] at 298 0
    \pinlabel {\small 1} [t] at 320 0
    \pinlabel {\small 7} at 48 49
    \pinlabel {\small 8} at 58 49

    \pinlabel $y_2$ at 115 22
    \pinlabel $y_1$ at 77 22
    \pinlabel $y_3$ at 318 22

    \endlabellist
    \includegraphics[scale=1.0]{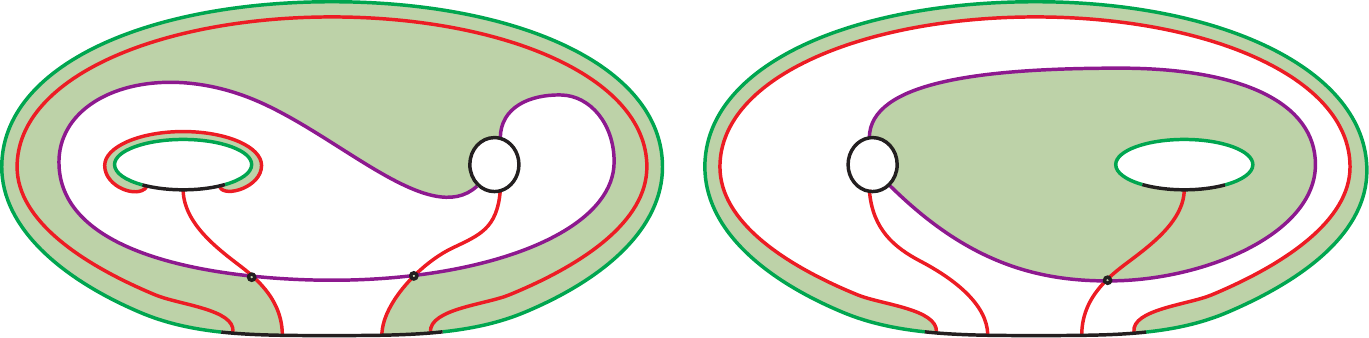}
    \caption{The bordered--sutured diagram $\HD_\infty$.}
    \label{fig:hd_infty}
  \end{subfigure}
  \begin{subfigure}[b]{\textwidth}
    \captionsetup{aboveskip={\dimexpr10pt+2pt+\sactualfontsize\relax}}
    \centering
    \labellist
    \scriptsize\hair 2pt
    \pinlabel \textcolor{red}{4} [t] at 65 1
    \pinlabel \textcolor{red}{5} [t] at 82 1
    \pinlabel \textcolor{red}{3} [t] at 109 1
    \pinlabel \textcolor{red}{4} [t] at 124 1
    \pinlabel \textcolor{red}{2} [t] at 270 1
    \pinlabel \textcolor{red}{3} [t] at 284 1
    \pinlabel \textcolor{red}{1} [t] at 311 1
    \pinlabel \textcolor{red}{2} [t] at 327 1
    \pinlabel \textcolor{red}{6} at 43 48
    \pinlabel \textcolor{red}{5} at 53 47
    \pinlabel \textcolor{red}{6} at 63 48
    \pinlabel \textcolor{red}{1} at 341 47

    \pinlabel {\small 6} [t] at 73 0
    \pinlabel {\small 5} [t] at 95 0
    \pinlabel {\small 4} [t] at 117 0
    \pinlabel {\small 3} [t] at 277 0
    \pinlabel {\small 2} [t] at 298 0
    \pinlabel {\small 1} [t] at 320 0
    \pinlabel {\small 7} at 48 49
    \pinlabel {\small 8} at 58 49

    \pinlabel $z_1$ at 64 32
    \pinlabel $z_2$ at 341 26

    \endlabellist
    \includegraphics[scale=1.0]{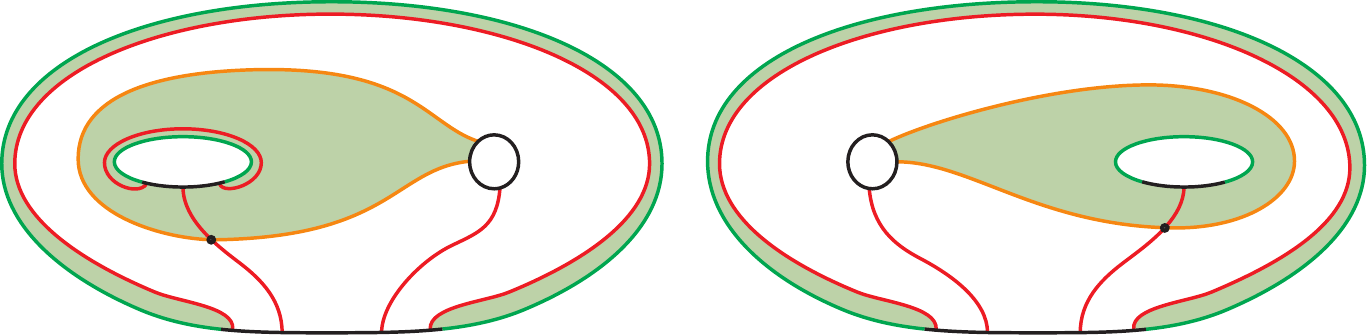}
    \caption{The bordered--sutured diagram $\HD_0$.}
    \label{fig:hd_zero}
  \end{subfigure}
  \caption{The bordered--sutured diagrams $\HD_k$ for $\Bk$. Red curves 
    represent $\alphas = \alphas^a$, green curves represent $\bdy \Sigma 
    \setminus \orarcs$, and the blue, purple, and gold curves represent 
    $\betasone$, $\betasinfty$, and $\betaszero$ respectively. A small red 
    integer $i$ indicates the arc $\alpha^a_i = \psi (e_i)$, and a larger black 
    digit $j$ indicates a Reeb chord $\rho_j$. (Only those Reeb chords where 
    $j$ consists of a single digit are shown.)
    The regions adjacent to $\bdy \Sigma \setminus \orarcs$ are shaded in light 
    green.}
  \label{fig:hd}
\end{figure}
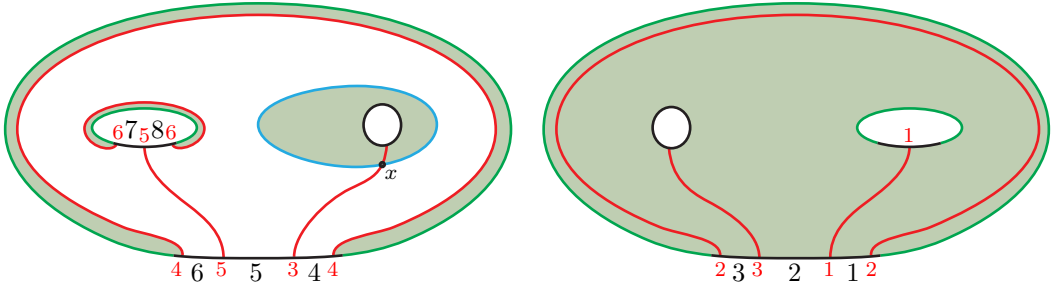
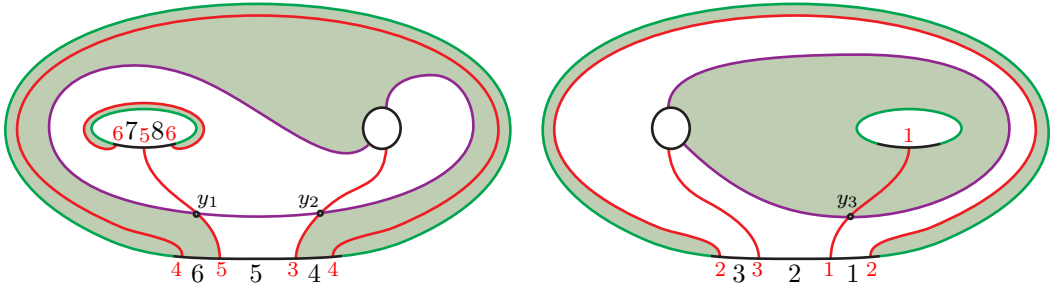
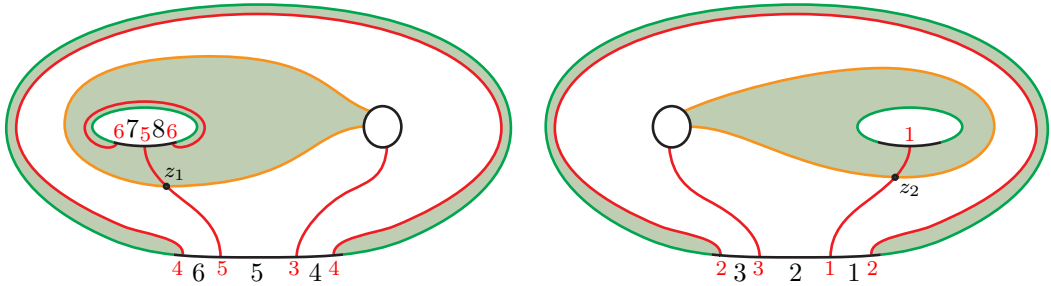

Bordered sutured diagrams $\HD_k = (\Sigma, \alphas, \betas^k, -\arcdiag, 
\psi)$ for $\Bk$ are shown in \fullref{fig:hd}. Here, the three diagrams share 
the same surface with boundary $\Sigma$, collection $\alphas = \alphas^a$ of 
$\alpha$-arcs, arc diagram $\arcdiag$, and embedding $\psi \colon G (\arcdiag) 
\into \Sigma$; the only difference is in the collections $\betas^k$, each of 
which consists of a single $\beta$-circle $\beta^k$. Note that $\HD_1$, 
$\HD_\infty$, and $\HD_0$ are provincially admissible but not admissible. (See 
\cite[Definition~4.4.1]{Zar11:BSFH} for the definitions of admissibility.)

Since there is a single $\beta$-circle in each of these diagrams, there will 
always be five unoccupied arcs in $\arcdiag$, and we may view $\BSDh (\Bk)$ as 
a type~$D$ module over $\A = \AZfive$, confirming our claim in 
\fullref{ssec:4puncS2}.

\subsubsection{The module \texorpdfstring{$\BSDh (\B_1)$}{BSD(B1)}}
\label{sssec:BSDY1}

We begin by computing the type~$D$ module $\BSDh (\B_1)$. In the Heegaard 
diagram $\HD_1$, there is exactly one intersection point $x$ between $\alphas = 
\alphas^a$ and $\betas^1 = \set{\beta^1}$, and so $\BSDh (\B_1)$ is generated 
by a single element $\x$, with idempotent
\[
  \Iocc{3} \cdot \x = \x.
\]
There is only one region $R$ in $\HD_1$ not adjacent to $\bdy \Sigma \setminus 
\orarcs$, and so all domains are of the form $n R$. Note that $\bdy^\bdy R = 
[-\rho_4] + [-\rho_5] + [-\rho_6] + [-\rho_7] + [-\rho_8]$. If there exists a 
holomorphic curve $u$, with homology class $[u] = B \in \pi_2 (\x, \x)$, that 
contributes to $\diff (\x)$, then $u \in \embmoduli^B (\x, \x; \source; 
\vec{P})$, where
\begin{itemize}
  \item $\vec{P}$ is a discrete ordered partition;
  \item The domain $[B]$ of $B$ satisfies $\bdy^\bdy [B] = [\vec{\rho} 
    (\vec{P})]$;
  \item $\ind (B, \vec{\rho} (\vec{P})) = 1$; and
  \item $a (- \vec{\rho} (\vec{P})) \neq 0$.
\end{itemize}
The second and the last of these conditions, together with 
\fullref{rmk:no_repeat}, imply that the coefficients of $[-\rho_4]$, 
$[-\rho_6]$, $[-\rho_7]$, and $[-\rho_8]$ in $\bdy^\bdy [B]$ cannot be greater 
than $1$; therefore, the only domain that can possibly contribute to $\diff 
(\x)$ is $R$, and we may consider only the homology class $B$ with $[B] = R$.  
Now a simple calculation using \cite[Proposition~5.3.5]{Zar11:BSFH} and the 
first and the third conditions above shows that the decorated source $\source$ 
must satisfy $\chi (\source) = 1$ and $\# \vec{\rho} (\vec{P})$ must be $5$.  
Hence, $\source$ must be a topological disk, with one $(+)$-puncture, one 
$(-)$-puncture, and five $e$-punctures labeled by $-\rho_4$, $-\rho_6$, 
$-\rho_7$, $-\rho_8$, and $-\rho_5$, in that order with the standard 
orientation. Here, one may deduce the order of the punctures by inspecting 
$\HD_1$; in fact, this will also reveal the order of the elements of $\vec{P}$.  
With these choices of $B$, $\source$, and $\vec{P}$, the Riemann mapping 
theorem implies that $\embmoduli^B (\x, \x; \source; \vec{P})$ is indeed 
nonempty and in fact consists of exactly one holomorphic representative $u$.  
This shows that
\[
  \diff (\x) = a_5 (\rho_4) a_5 (\rho_6) a_5 (\rho_7) a_5 (\rho_8) a_5 (\rho_5) 
  \tensor \x = (4, 6, 7, 8, 5) \tensor \x.
\]
One may easily see from our notation that
\[
  d (4, 6, 7, 8, 5) = 0 = (4, 6, 7, 8, 5)^2,
\]
where the second equality follows from \fullref{rmk:no_repeat}, with forbidden 
digits $4, 6, 7, 8$ repeated; this confirms that $\BSDh (\B_1)$ does indeed 
satisfy the type~$D$ structure equation.

\begin{remark}
  \label{rmk:computations}
  In the following discussion, we will generally not provide justification for 
  the boundary operator computations, as each of them follows from a line of 
  reasoning similar to the above: In general, \fullref{rmk:no_repeat} provides 
  an easy way to identify the domains that can possibly contribute, of which 
  there are finitely many. The index formula 
  \cite[Proposition~5.3.5]{Zar11:BSFH} implies that $\source$ must be a 
  topological disk, and the Riemann mapping theorem then implies the existence 
  of a unique holomorphic representative. See \fullref{rmk:complication} below 
  for the only complication.
\end{remark}

\begin{remark}
  \label{rmk:no_full_reeb}
  To aid in one's computations, it is also worth noting that if a holomorphic 
  curve with decorated source $\source$ contributes to any map defined by 
  counting holomorphic curves, none of the $e$-punctures of $\source$ can be 
  labeled by $-\rho_j$ with $j \in J' \setminus J = \set{123, 456, 78}$. To see 
  this, suppose $q$ is an $e$-puncture labeled by, say, $- \rho_{123}$; then by 
  examining $\HD_k$, the next puncture on $\bdy \source$ after $q$, in the 
  standard orientation, must be an $e$-puncture labeled by $- \rho_{j_0}$, 
  where $j_0$ must contain the digit $3$. This violates the restriction in 
  \fullref{rmk:no_repeat}.
\end{remark}

\subsubsection{The module \texorpdfstring{$\BSDh (\B_\infty)$}{BSD(Binfinity)}}
\label{sssec:BSDYi}

In $\HD_\infty$, there are three intersection points between $\alphas$ and 
$\betas^{\infty}$, which we label by $y_1$, $y_2$, and $y_3$ as in 
\fullref{fig:hd}. The corresponding generators have idempotents
\[
  \Iocc{5} \cdot \y_1 = \y_1, \quad \Iocc{3} \cdot \y_2 = \y_2, \quad \Iocc{1} 
  \cdot \y_3 = \y_3.
\]
The boundary operation is given by
\begin{align*}
  \delta (\y_1) & = (5) \tensor \y_2 + (7, 8, 5, 12, 3) \tensor \y_2 + (7, 8, 
  5, 2) \tensor \y_3,\\
  \delta (\y_2) & = 0,\\
  \delta (\y_3) & = (1, 3) \tensor \y_2.
\end{align*}

\subsubsection{The module \texorpdfstring{$\BSDh (\B_0)$}{BSD(B0)}}
\label{sssec:BSDY0}

In $\HD_0$, there are two intersection points between $\alphas$ and $\betas^0$, 
which we label by $z_1$ and $z_2$ as in \fullref{fig:hd}. The corresponding 
generators have idempotents
\[
  \Iocc{5} \cdot \z_1 = \z_1, \quad \Iocc{1} \cdot \z_2 = \z_2.
\]
The boundary operation is given by
\begin{align*}
  \delta (\z_1) & = (5, 12, 3, 4, 6) \tensor \z_1 + (5, 2) \tensor \z_2 + (5, 
  12, 3, 4, 56, 2) \tensor \z_2,\\
  \delta (\z_2) & = (1, 3, 4, 6) \tensor \z_1 + (1, 3, 4, 56, 2) \tensor \z_2.
\end{align*}

\begin{remark}
  \label{rmk:complication}
  We illustrate one consideration we must take in the calculation above, beyond 
  the argument outlined in \fullref{rmk:computations}. Let $B \in \pi_2 (\z_2, 
  \z_1)$ be the homology class from $\z_1$ to $\z_2$ that gives rise to the 
  term $(5, 12, 3, 4, 56, 2) \tensor \z_2$ above; then $[B] = R_1 + 2 R_2$, 
  where $R_2$ is the unique region adjacent to $- \rho_2$ and $- \rho_5$, and 
  $R_1$ is the only other non-forbidden region.  Since $g = 1$ and $e (B) = - 5 
  / 2$, applying \cite[Proposition~5.3.5]{Zar11:BSFH} to $B$ shows that we must 
  have $\chi (\source) = \# \vec{P} - 5$. Because $\vec{P}$ is discrete, 
  inspecting $\HD_0$, we see that $\# \vec{P} \geq 4$. For all other domains 
  involved in the calculations in this subsection, the analogous restriction on 
  $\# \vec{P}$ and the fact that $\chi (\source) \leq 1$ (as $\source$ is a 
  connected surface with at least one boundary component) together imply that 
  $\chi (\source) = 1$; here, however, we obtain two possibilities: $\chi 
  (\source) = 1$ and $\# \vec{P} = 6$; and $\chi (\source) = -1$ and $\# 
  \vec{P} = 4$. We must argue that we may discard the latter possibility, as 
  follows. We note that $\# \vec{P} = 4$ implies that $\vec{\rho} (\vec{P})$ 
  contains, in some order, either $- \rho_{123}$ and $- \rho_{2}$, or $- 
  \rho_{12}$ and $- \rho_{23}$; and two other Reeb chords with index digits 
  $4$, $5$, and $6$. We know that $- \rho_{123} \notin \vec{\rho} (\vec{P})$ by 
  \fullref{rmk:no_full_reeb}, and so $\vec{\rho} (\vec{P})$ must contain $- 
  \rho_{12}$ and $- \rho_{23}$; but $(12, 23) = (23, 12) = 0$. (Alternatively, 
  we may apply \fullref{rmk:no_full_reeb} to $- \rho_{456}$ and observe that 
  $(45, 56) = (56, 45) = 0$.) We may then proceed with the usual argument for 
  the former possibility, where $\chi (\source) = 1$ and $\# \vec{P} = 6$, 
  which gives the aforementioned term in $\delta (\z_1)$.

  In \fullref{sec:skein_via_computation} and in particular the proof of 
  \fullref{prop:poly_comp}, this complication will again arise whenever the 
  domain $[B]$ has multiplicity greater than $1$ at some region.
\end{remark}

\subsection{Proof of the main theorem without gradings}
\label{ssec:proof_main}

The key step to proving \fullref{thm:main} will be the following special case:

\begin{proposition}
  \label{prop:elementary}
  There exist type~$D$ homomorphisms $f_k \colon \BSDh (\Bk) \to \BSDh 
  (\Bnext)$ such that
  \[
    \BSDh (\Bk) \homeq \Cone (f_{k+1} \colon \BSDh (\Bnext) \to \BSDh (\Bprev))
  \]
  as type~$D$ structures.
\end{proposition}

We will give two proofs of \fullref{prop:elementary}, using counts of bordered 
holomorphic polygons in \fullref{sec:skein_via_polygon} and by direct 
computation in \fullref{sec:skein_via_computation}. Using 
\fullref{prop:elementary}, we may prove the following generalization of 
\fullref{thm:main}:

\begin{maintheorem}
  \label{thm:general}
  Let $\Y' = (Y', \sut', \arcdiag_1 \union \arcdiag_2 \union \arcdiag, \phi')$ 
  be any bordered--sutured manifold, and let $\Yk = \Y' \glue{\sutF (\arcdiag)} 
  \Bk$. There exist type~$\DA$ homomorphisms $F_k \colon \BSDAh (\Yk) \to 
  \BSDAh (\Ynext)$ such that
  \[
    \BSDAh (\Yk) \homeq \Cone (F_{k+1} \colon \BSDAh (\Ynext) \to \BSDAh 
    (\Yprev))
  \]
  as type~$\DA$ structures.  Moreover, the homomorphisms $F_k$ and the homotopy 
  equivalence above respect the relative gradings on the bimodules in a sense 
  to be made precise in \fullref{sec:gradings}; see \fullref{thm:gradings}.
\end{maintheorem}

\begin{proof}[Proof of \fullref{thm:general}, without gradings]
  To the bordered--sutured manifold $\Y'$, Zarev \cite{Zar11:BSFH} associates 
  (the homotopy type of) a type~$\DA$ bimodule
  \[
    \itensor[^{\A (-\arcdiag_1)}]{\BSDAh (\Y')}{_{\A (\arcdiag_2 \union 
        \arcdiag)}} = \itensor[^{\A (-\arcdiag_1)}]{\BSDAh (\Y')}{_{\A 
        (\arcdiag_2) \tensor \AZ}}
  \]
  By choosing an admissible diagram $\HD'$ for $\Y'$, we may assume $\BSDAh 
  (\Y')$ to be bounded.
  
  Unlike in the familiar setting for $\dg$ modules, the $\dg$ category 
  $\itensor{\Mod}{_{\A_1 \tensor \A_2}}$ of (right) type~$A$ modules over $\A_1 
  \tensor \A_2$ is not the same as the $\dg$ category $\itensor{\Mod}{_{\A_1, 
      \A_2}}$
  of (right-right) type~$\AA$ bimodules over $\A_1$ and $\A_2$.  There is, 
  nonetheless, an equivalence
  \[
    \calF_{\AA} \colon \HC (\itensor{\Mod}{_{\A_1 \tensor \A_2}}) \to \HC 
    (\itensor{\Mod}{_{\A_1, \A_2}})
  \]
  between the homotopy categories; see 
  \cite[Proposition~2.4.11]{LipOzsThu15:BFHBimodules} and 
  \cite[Section~8.3]{Zar11:BSFH}.  In our context, there is an analogous 
  functor $\calF_{DAA}$ that allows us to view $\BSDAh (\Y')$ as a type~$\DAA$ 
  trimodule
  \[
    \itensor[^{\A (-\arcdiag_1)}]{\BSDAAh (\Y')}{_{\A (\arcdiag_2), \AZ}} = 
    \calF_{\DAA} \paren{\itensor[^{\A (-\arcdiag_1)}]{\BSDAh (\Y')}{_{\A 
          (\arcdiag_2) \tensor \AZ}}},
  \]
  allowing us to write
  \begin{align*}
    \BSDAh (\Yk) & \simeq \BSDAAh (\Y') \boxtensor \BSDh (\Bk)\\
    & \simeq \BSDAAh (\Y') \boxtensor \Cone (f_{k+1} \colon \BSDh (\Bnext) \to 
    \BSDh (\Bprev))\\
    & \simeq \Cone (\Id_{\BSDAAh (\Y')} \boxtensor f_{k+1} \colon \BSDAAh (\Y') 
    \boxtensor \BSDh (\Bnext) \to \BSDAAh (\Y') \boxtensor \BSDh (\Bprev))\\
    & \simeq \Cone (F_{k+1} \colon \BSDAh (\Ynext) \to \BSDAh (\Yprev))
  \end{align*}
  for some type~$\DA$ homomorphism $F_{k+1}$, where we have respectively used 
  the pairing theorem \cite[Theorem~8.5.1]{Zar11:BSFH}, 
  \fullref{prop:elementary}, 
  \cite[Lemma~2.9]{LipOzsThu14:BFHBranchedDoubleSSI}, and the pairing theorem 
  again.
\end{proof}

\begin{proof}[Proof of \fullref{thm:main}, without gradings]
  Define $Y' = Y_k \setminus B^3$, which is a compact $3$-manifold with 
  boundary. Considering $\sut_k$, note that by the definition of a 
  bordered--sutured structure, there are sutures on each boundary component of 
  $Y_k$; thus, we may isotope each $\sut_k$ such that it coincides with 
  $\mersut_k$ in the interior of $B^3$.  Let $\sut' = (\sut_k \intersect Y') 
  \union (- \mersut_k \intersect \bdy \closure{B^3})$, and let $\phi' \colon G 
  (\arcdiag_1 \union \arcdiag_2 \union \arcdiag) \into \bdy Y'$ be the map 
  obtained by setting $\phi'|_{G (\arcdiag_1 \union \arcdiag_2)} = \phi_k$ and 
  defining $\phi'|_{G (\arcdiag)}$ to be $\elparam_k$ with the opposite 
  orientation.  (Recall that, with the codomain restricted to the image, 
  neither $\phi_k$ nor $\elparam_k$ depends on $k$.) Then $\Y' = (Y', \sut', 
  \arcdiag_1 \union \arcdiag_2 \union \arcdiag, \phi')$ is a bordered--sutured 
  manifold, such that $\Yk = \Y' \glue{\sutF (\arcdiag)} \Bk$. The theorem now 
  follows from \fullref{thm:general}.
\end{proof}

\begin{proof}[Proof of \fullref{cor:arbitrary}, without gradings]
  For $Y$ without boundary and $\arcdiag_1 = \arcdiag_2 = \eset$, the 
  bordered--sutured manifold $\Yk = (Y_k, \sut_k, \eset, \phi_k)$ is simply a 
  sutured manifold $(Y_k, \sut_k)$ (where $Y_k = Y \setminus \nbhd{L_k}$), and 
  the type~$\DA$ structure $\BSDAh (\Yk)$ reduces to the sutured Floer chain 
  group $\SFC (Y_k, \sut_k)$; see, for example, \cite[Section~9.1]{Zar11:BSFH}.  
  Note that this is true for all $\set{\sut_k}_{k \in \set{\infty, 0, 1}}$ 
  satisfying the conditions described in \fullref{sec:introduction}; below, we 
  choose one such set to prove our corollary.

  Observe that $\elT_k \subset L_k$ intersects exactly one component of $L_k$ 
  for two values $k_1$, $k_2$ of $k \in \set{\infty, 0, 1}$, and intersects two 
  components of $L_k$ for the other value $k_3$; we shall refer to these link 
  components as the \emph{special link components}. Noting that each component 
  of $\bdy Y_k$ corresponds to a link component of $L_k$, choose $\set{\sut_k}$ 
  such that
  \begin{itemize}
    \item For each point $p$ in $\bdy \elT_k$, there is a distinct meridional 
      suture in $\sut_k$ that is a push-off of the meridian of $p$ into the 
      interior of $\bdy Y_k \intersect Y'$;
    \item There are no other sutures on the components of $\bdy Y_k$ that 
      correspond to the special link components; and
    \item There is exactly one pair of meridional sutures on each component of 
      $\bdy Y_k$ that does not correspond to one of the special link 
      components.
  \end{itemize}
  With this choice, $(Y_k, \sut_k)$ is exactly the link complement with one 
  pair of meridional sutures for each link component, and an additional pair of 
  meridional sutures for the special link components when $k \in \set{k_1, 
    k_2}$. By \cite[Proposition~9.2]{Juh06:SFH} and an easy extension of 
  \cite[Proposition~2.3]{ManOzsSar09:GH} to links in arbitrary $3$-manifolds, 
  we see that
  \[
    \SFH (Y_k, \sut_k) \isom
    \begin{cases}
      \HFKh (Y, L_k) \tensor V & \text{if } k = k_1, k_2,\\
      \HFKh (Y, L_k) & \text{if } k = k_3.
    \end{cases}
  \]

  Applying \fullref{thm:main} to $\SFC (Y_k, \sut_k)$, the mapping cone of 
  chain complexes naturally gives rise to a short exact sequence, and hence a 
  long exact triangle on homology.
\end{proof}

\section{The skein relation via holomorphic polygon counts}
\label{sec:skein_via_polygon}

In this section, we give a theoretical proof of \fullref{prop:elementary} using 
a technique of counting holomorphic polygons in the bordered setting, developed 
by Lipshitz, Ozsv\'ath, and Thurston \cite{LipOzsThu16:BFHBranchedDoubleSSII}.  
We shall use these polygon counts to define type~$D$ morphisms, which satisfy a 
type~$D$ version of some well-known $\Ainf$-relations. That the morphisms 
satisfy the conditions in \fullref{lem:hom_alg} will then be an easy 
consequence.

\subsection{Polygon counting in bordered--sutured Floer theory}
\label{ssec:bordered_polygons}

Lipshitz, Ozsv\'ath, and Thurston 
\cite[Section~4]{LipOzsThu16:BFHBranchedDoubleSSII} describe a theory of 
holomorphic polygon counting for bordered Floer homology.  The same technique 
translates to the bordered--sutured Floer setting with small changes.  Below, we 
outline the important definitions and results in this direction.

\begin{definition}[cf.\ 
  {\cite[Definition~4.1]{LipOzsThu16:BFHBranchedDoubleSSII}}]
  Let $\Sigma$ be a compact, oriented surface with boundary. Fix the data 
  $(\arcdiag, \psi)$, where $\arcdiag$ is an arc diagram and $\psi \colon G 
  (\arcdiag) \to \Sigma$ is an orientation-reversing embedding as in 
  \cite[Definition~4.1]{Zar11:BSFH}. (See \cite[Definition~2.4]{Zar11:BSFH} for 
  the definition of $G (\arcdiag)$, and of $e_i$ below.) A \emph{set of 
    bordered attaching curves compatible with $\arcdiag$} is a collection 
  $\alphas = \alphas^a \union \alphas^c$ of curves in $\Sigma$ such that
  \begin{itemize}
    \item The collection $\alphas^a = \set{\alpha_1^a, \dotsc, \alpha_q^r}$ 
      consists exactly of the arcs $\alpha_i^a = \psi (e_i)$;
    \item The collection $\alphas^c = \set{\alpha_1^c, \dotsc, \alpha_r^s}$ 
      consists of simple closed curves in $\Int (\Sigma)$;
    \item The curves in $\alphas$ are pairwise disjoint; and
    \item The map $\pi_0 (\bdy \Sigma \setminus \orarcs) \to \pi_0 (\Sigma 
      \setminus \alphas)$ is surjective.
  \end{itemize}
\end{definition}

\begin{definition}[cf.\ 
  {\cite[Definition~3.1]{LipOzsThu16:BFHBranchedDoubleSSII}}]
  Let $\Sigma$ be a compact, oriented surface with boundary. A \emph{set of 
    attaching circles} is a collection $\betas = \set{\beta_1, \dotsc, 
    \beta_t}$ of simple closed curves in $\Int (\Sigma)$ such that
  \begin{itemize}
    \item The curves in $\betas$ are pairwise disjoint; and
    \item The map $\pi_0 (\bdy \Sigma) \to \pi_0 (\Sigma \setminus \betas)$ is 
      surjective.
  \end{itemize}
\end{definition}

\begin{definition}[cf.\ 
  {\cite[Definition~4.2]{LipOzsThu16:BFHBranchedDoubleSSII}}]
  A \emph{bordered--sutured Heegaard multi-diagram} is the data $(\Sigma, 
  \alphas, \set{\betas^k}_{k=1}^m, \arcdiag, \psi)$, where $\alphas$ is a set 
  of bordered attaching curves compatible with $\arcdiag$ in $\Sigma$, and  
  $\set{\betas^k}_{k=1}^m$ is some $m$-tuple of sets of $t$ attaching circles, 
  for some fixed $t$.  Without loss of generality, we may assume that all 
  curves involved are pairwise transverse, and that there are no triple 
  intersection points.  For the sake of economy, we often omit the word 
  ``Heegaard''.

  A \emph{generalized multi-periodic domain} is a relative homology class $B 
  \in H_2 (\Sigma, \alphas \union \betas^1 \union \dotsb \union \betas^m \union 
  \bdy \Sigma)$ whose boundary $\bdy B$, viewed as an element of
  \[
    H_1 (\alphas \union \betas^1 \union \dotsb \union \betas^m \union \bdy 
    \Sigma, \bdy \Sigma) \isom H_1 (\alphas \union \betas^1 \union \dotsb 
    \union \betas^m, \matchedpts),
  \]
  is contained in the image of the inclusion
  \[
    H_1 (\alphas, \matchedpts) \dirsum H_1 (\betas^1) \dirsum \dotsb \dirsum 
    H_1 (\betas^m) \to H_1 (\alphas \union \betas^1 \union \dotsb \union 
    \betas^m, \matchedpts).
  \]
  It is a \emph{multi-periodic domain} if all of its local multiplicities at 
  the regions adjacent to $\bdy \Sigma \setminus \orarcs$ are zero. A 
  multi-periodic domain is \emph{provincial} if all of its local multiplicities 
  at the regions adjacent to $\bdy \Sigma$ are zero. The bordered--sutured 
  multi-diagram $(\Sigma, \alphas, \set{\betas^k}_{k=1}^m, \arcdiag, \psi)$ is 
  \emph{admissible} (resp.\ \emph{provincially admissible}) if every non-zero 
  multi-periodic domain (resp.\ non-zero provincial multi-periodic domain) has 
  both positive and negative local multiplicities.
\end{definition}

Given a bordered--sutured multi-diagram, one could always find an 
\emph{admissible collection of almost-complex structures} $\set{J_j}_{j \in 
  \Conf (D_n)}$; see \cite[Definition~4.4]{LipOzsThu16:BFHBranchedDoubleSSII}.  
($\Conf (D_n)$ denotes the moduli space of complex structures on the disk $D_n$ 
with $n$ labeled punctures on its boundary, i.e.\ an $n$-sided polygon.) We fix 
this data.

If $(\Sigma, \alphas, \set{\betas^k}_{k=1}^m, \arcdiag, \psi)$ is a 
provincially admissible bordered--sutured multi-diagram, then for all $k$, the 
bordered--sutured diagram $\HD_k = (\Sigma, \alphas, \betas^k, \arcdiag, \psi)$ 
is provincially admissible in the sense of \cite[Definition~4.10]{Zar11:BSFH}. 
Thus, we may consider the type~$D$ structure $\BSDh (\HD_k)$; we denote this 
type~$D$ structure by $\BSDh (\alphas, \betas^k)$. (For the expert, the 
implicit choice of a suitable almost-complex structure is specified by the data 
$\set{J_j}$ here.) Observe further that $\HD_{k_1, k_2} = (\Sigma, 
\betas^{k_1}, \betas^{k_2})$ in fact is a sutured Heegaard diagram that is 
admissible in the sense of \cite[Definition~3.11]{Juh06:SFH}; we denote by 
$\SFC (\betas^{k_1}, \betas^{k_2})$ its sutured Floer chain complex $\SFC 
(\HD_{k_1, k_2})$.

The discussion in \cite[Section~4]{LipOzsThu16:BFHBranchedDoubleSSII} then 
carries over to the present context, allowing us to define maps by polygon 
counts. More precisely, if $(\Sigma, \alphas, \set{\betas^k}_{k=1}^m, 
-\arcdiag, \psi)$ is a provincially admissible bordered--sutured multi-diagram, 
we obtain maps
\[
  \delta_n \colon \SFC (\betas^{k_{n-1}}, \betas^{k_n}) \tensor \dotsb \tensor 
  \SFC (\betas^{k_{1}}, \betas^{k_2}) \tensor \BSDh (\alphas, \betas^{k_1})\to 
  \AZ \tensor \BSDh (\alphas, \betas^{k_n}),
\]
defined, for generators $\Eta^{k_i} \in \genset (\betas^{k_{i}}, 
\betas^{k_{i+1}})$ and $\x \in \genset (\alphas, \betas^{k_1})$, by
\[
  \delta_n (\Eta^{k_{n-1}} \tensor \dotsb \tensor \Eta^{k_1} \tensor \x) = 
  \sum_{\substack{
      \y \in \genset (\alphas, \betas^{k_n})\\
      \source, \text{ discrete } \vec{P}\\
      \ind (B, \vec{\rho} (\vec{P})) = 2 - n
    }}
  \args{\# \embmoduli^B (\y, \Eta^{k_{n-1}}, \dotsc, \Eta^{k_{1}}, \x; \source; 
    \vec{P})}
  a \args{-\vec{\rho} (\vec{P})} \tensor \y.
\]
Here, $\embmoduli^B (\y, \Eta^{k_{n-1}}, \dotsc, \Eta^{k_{1}}, \x; \source; 
\vec{P})$ denotes the moduli space of pairs $(j, u)$, where $j$ is a complex 
structure on $D_{n+1}$, and $u$ is an embedded $J_j$-holomorphic representative 
of the homology class $B \in \pi_2 (\y, \Eta^{k_{n-1}}, \dotsc, \Eta^{k_{1}}, 
\x)$, with decorated source $\source$, and compatible with a discrete ordered 
partition $\vec{P}$ of the $e$-punctures of $\source$. Note that $\delta_1$ is 
the usual structure map on $\BSDh (\alphas, \betas^{k_1})$.

\begin{remark}
  \label{rmk:order}
  We have chosen to follow \cite{LipOzsThu16:BFHBranchedDoubleSSII} in 
  presenting the arguments of $\delta_n$ in the order above. As a result, the 
  generator arguments of the moduli spaces $\moduli$ are presented in the order 
  reverse to that in \cite{Zar11:BSFH}. See \cite[Convention~3.4 and 
  Remark~4.22]{LipOzsThu16:BFHBranchedDoubleSSII} for some justification.
\end{remark}

Recall that in the non-bordered case, there are well-known maps
\[
  m_n \colon \SFC (\betas^{k_{n-1}}, \betas^{k_{n}}) \tensor \dotsb \tensor 
  \SFC (\betas^{k_{1}}, \betas^{k_{2}}) \tensor \SFC (\betas^{k_{0}}, 
  \betas^{k_{1}}) \to \SFC (\betas^{k_{0}}, \betas^{k_{n}}),
\]
defined by holomorphic polygon counts, which are well known to satisfy some 
$\Ainf$-relations; see, for example, \cite{Sei08:FukBook} and 
\cite{FukOhOht09:LFHBookI,FukOhOht09:LFHBookII}. One may view the map $m_n$ as 
a special case of $\delta_n$, where $\alphas = \betas^{k_{0}}$ is a set of 
bordered attaching curves compatible with the empty arc diagram $\arcdiag = 
\eset$, and $\AZ = \A (\eset) = \F{2}$. Conversely, $\delta_n$ can be viewed as 
the bordered version of $m_n$. Note also that $m_1$ is the usual differential 
on $\SFC (\betas^{k_0}, \betas^{k_1})$.

The map $\delta_n$ can be naturally extended to a map
\[
  \deltat_n \colon \SFC (\betas^{k_{n-1}}, \betas^{k_n}) \tensor \dotsb \tensor 
  \SFC (\betas^{k_{1}}, \betas^{k_2}) \tensor \AZ \tensor \BSDh (\alphas, 
  \betas^{k_1})\to \AZ \tensor \BSDh (\alphas, \betas^{k_n})
\]
defined by
\[
  \deltat_n (\Eta^{k_{n-1}} \tensor \dotsb \tensor \Eta^{k_{1}} \tensor a 
  \tensor \x) = \paren{\mu (a, \blank) \tensor \id_{\BSDh (\alphas, 
      \betas^{k_n})}} \comp \delta_n (\Eta^{k_{n-1}} \tensor \dotsb \tensor 
  \Eta^{k_{1}} \tensor \x).
\]
Then the result we need, by adapting from 
\cite[Section~4]{LipOzsThu16:BFHBranchedDoubleSSII}, is the following 
proposition, which is the bordered version of the $\Ainf$-relations for $m_n$.

\begin{proposition}[cf.\ 
  {\cite[Proposition~4.23]{LipOzsThu16:BFHBranchedDoubleSSII}}]
  \label{prop:ainf}
  The holomorphic polygon counts defined above satisfy the $\Ainf$-relations:
  \begin{equation*}
    \begin{split}
      0& = \sum_{1 \leq p \leq n} \deltat_{n-p+1} (\Eta^{k_{n-1}} \tensor 
      \dotsb \tensor \Eta^{k_{p}} \tensor \delta_p (\Eta^{k_{p-1}} \tensor 
      \dotsb \tensor \Eta^{k_{1}} \tensor \x))\\
      & \quad + \sum_{1 \leq p < q \leq n} \delta_{n-q+p+1} (\Eta^{k_{n-1}} 
      \tensor \dotsb \tensor \Eta^{k_{q}} \tensor m_{q-p} (\Eta^{k_{q-1}} 
      \tensor \dotsb \tensor \Eta^{k_{p}}) \tensor \Eta^{k_{p-1}} \tensor 
      \dotsb \tensor \Eta^{k_{1}} \tensor \x)\\
      & \quad + \paren{d \tensor \id_{\BSDh (\alphas, \betas^{k_n})}} \comp 
      \delta_n (\Eta^{k_{n-1}} \tensor \dotsb \tensor \Eta^{k_{1}} \tensor \x).
    \end{split}
  \end{equation*}
\end{proposition}

The maps $\delta_n$ have analogous versions for type~$A$ modules and for 
bimodules of various types. For example, suppose $(\Sigma, \alphas, 
\set{\betask}_{k=1}^m, \arcdiag_1 \union \arcdiag_2, \psi)$ is a provincially 
admissible bordered--sutured multi-diagram, and let $\HD_k = (\Sigma, \alphas, 
\betask, \arcdiag_1 \union \arcdiag_2, \psi)$; denoting by $\BSDAh (\alphas, 
\betas^k)$ the type~$\DA$ structure
\[
  \itensor[^{\A (- \arcdiag_1)}]{\BSDAh (\HD_k)}{_{\A (\arcdiag_2)}},
\]
there are maps
\[
  \begin{split}
    \itensor[_n]{\delta}{_\ell} & \colon \SFC (\betas^{k_{n-1}}, \betas^{k_n}) 
    \tensor \dotsb \tensor \SFC (\betas^{k_1}, \betas^{k_2}) \tensor \BSDAh 
    (\alphas, \betas^{k_1}) \tensor \underbrace{\A (\arcdiag_2) \tensor \dotsb 
      \tensor \A (\arcdiag_2)}_\ell\\
    & \quad \to \A (- \arcdiag_1) \tensor \BSDAh (\alphas, \betas^{k_n}),
  \end{split}
\]
defined by holomorphic polygon counts, which satisfy $\Ainf$-relations similar 
to those in \fullref{prop:ainf}. We will not need these generalized maps in 
this section, but in the proof of \fullref{thm:agree} in 
\fullref{sec:comparison}, we will use the analogous maps for type~$\DAA$ 
trimodules, which are similarly defined.

\subsection{Counting provincial polygons}
\label{ssec:prov_polygons}

Before we define the type~$D$ morphisms needed in \fullref{lem:hom_alg}, we 
prepare ourselves by counting some provincial polygons and computing the 
associated maps $m_n$.

\begin{figure}[!htbp]
	\labellist
	\scriptsize\hair 2pt
  \pinlabel $\etainfty_1$ at 66 133
  \pinlabel $\etazeroH_1$ at 134 124
  \pinlabel $\etazero_1$ at 148 107
  \pinlabel $\etaone_1$ at 164 107
  \pinlabel $\etaHone_1$ at 180 117
  \pinlabel $\etaHone_0$ at 112 114
  \pinlabel $\thetaoneH$ at 102 91
  \pinlabel $\thetaHone$ at 100 68
  \pinlabel $\etaone_0$ at 127 103
  \pinlabel $\etainfty_0$ at 143 82
  \pinlabel $\etazeroH_0$ at 105 45
  \pinlabel $\etazero_0$ at 125 70
	
	\endlabellist
  \includegraphics[scale=1.0]{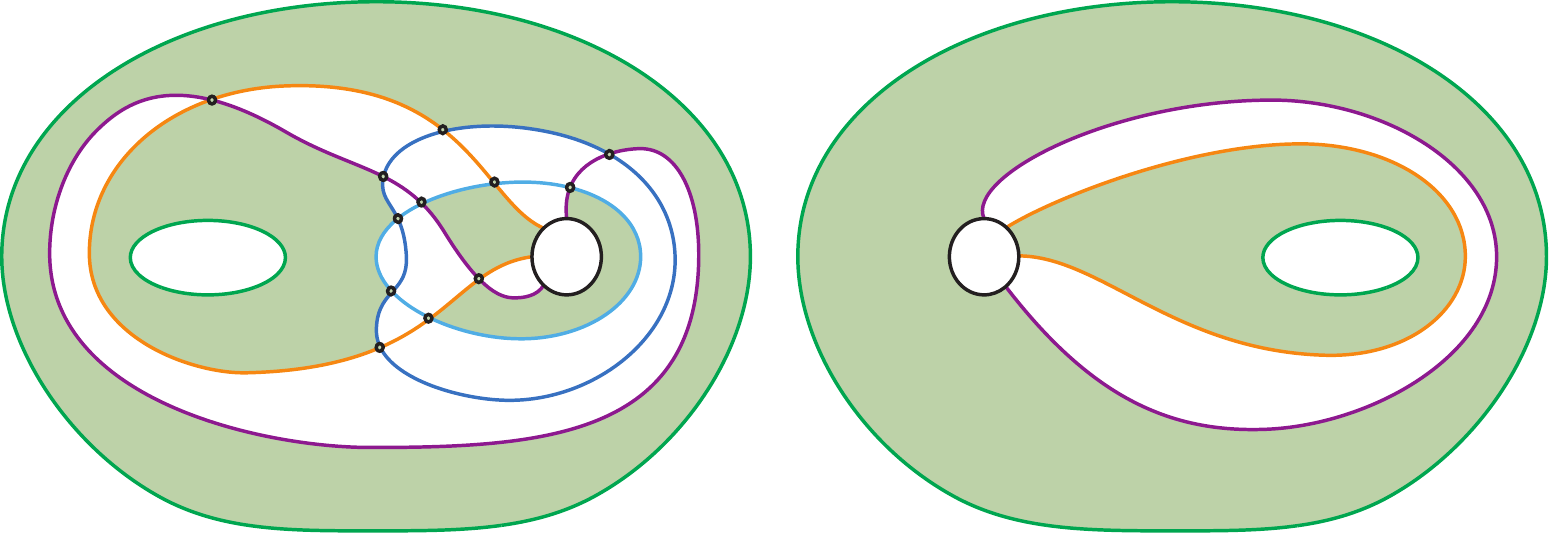}
  \caption{The sutured multi-diagram $\HDbH = (\Sigma, \set{\betask, 
      \betaskH}_{k \in \set{\infty, 0, 1}}$).}
 	\label{fig:mult_hd}
\end{figure}

Consider the bordered--sutured diagrams $\HD_k = (\Sigma, \alphas, \betask, - 
\arcdiag, \psi)$ in \fullref{ssec:tancomp}, where $k \in \set{\infty, 0, 1}$.  
We shall view $\betasone$, $\betasinfty$, and $\betaszero$ as sets of attaching 
circles on $\Sigma$, and combine them into one sutured multi-diagram as in 
\fullref{fig:mult_hd}. Without loss of generality, we may assume, as in 
\fullref{fig:mult_hd}, that the circles $\betak$ intersect transversely, 
pairwise, at two points; we label the two intersection points between $\betak$ 
and $\betanext$ by $\etak_0$ and $\etak_1$. (The choice of which intersection 
point is $\etak_0$ and which is $\etak_1$, as indicated in 
\fullref{fig:mult_hd}, is made to simpify the notation in \fullref{lem:m_123} 
below, but is itself not material in the proof of \fullref{prop:elementary}.) 
Each $\etak_i$ is associated to \emph{two} generators, one in $\genset 
(\betask, \betasnext)$ and one in $\genset (\betasnext, \betask)$; we denote 
the former by $\Etaknext_i$ and the latter by $\Etanextprev_i$.

For each $k \in \set{\infty, 0, 1}$, let $\betakH$ be a small Hamiltonian 
perturbation of $\betak$, such that
\begin{itemize}
  \item $\betak$ and $\betakH$ intersect transversely at two points $\thetakH$ 
    and $\thetaHk$, both of which lie on the boundary $\bdy p$ of the embedded 
    triangle $p$ with vertices $\etaone_0$, $\etainfty_0$, and $\etazero_0$; 
    and
  \item $\betakH$ intersects $\betanext$ (resp.\ $\betaprev$) transversely 
    close to $\betak \intersect \betanext$ (resp.\ $\betak \intersect 
    \betaprev$).
\end{itemize}
Let $\betaskH = \set{\betakH}$. In \fullref{fig:mult_hd}, we have drawn 
$\betasoneH$; to prevent cluttering our diagrams, however, we will not draw 
$\betasinftyH$ or $\betaszeroH$ here or in the sequel. As a consequence, 
whenever an argument below involves $\betaskH$ (and its intersection points 
with other curves), the reader is invited to check the veracity of the argument 
for the case $k = 1$, and to observe that the cases $k = \infty$ and $k = 0$ 
are entirely analogous.

We now fix the notation for the intersection points on $\betakH$. First, we 
label the intersection points between $\betak$ and $\betanextH$ such that 
$\etakH_i$ is close to $\etak_i$. Second, note that $(\Sigma, \betask, 
\betaskH)$ is a sutured Heegaard diagram for the complement of a $2$-component 
unlink in $S^3$, with a pair of meridional sutures on each unlink component.  
This means that
\[
  \SFC (\betask, \betaskH) \isom \CFKh (S^3, 2\text{-component unlink}),
\]
and so is generated by two generators whose (relative) gradings differ by $1$.  
Denote by $\thetakH$ the intersection point between $\betak$ and $\betakH$ that 
corresponds to the top-graded generator $\ThetakH \in \genset (\betask, 
\betaskH)$. The other intersection point, which we denote by $\thetaHk$, will 
then correspond to the top-graded generator $\ThetaHk \in \genset (\betaskH, 
\betask)$.

It is easy to check that the sutured multi-diagram $\HDbH = (\Sigma, 
\set{\betask, \betaskH}_{k \in \set{\infty, 0, 1}})$ is admissible. We now 
compute some $\Ainf$-relations associated to $\HDbH$. 

\begin{lemma}
  \label{lem:m_123}
  For $k \in \set{\infty, 0, 1}$ and $i, j, \ell \in \cycgrp{2}$, the maps 
  $m_n$ associated to $\HDbH$ satisfy:
  \begin{gather}
    m_1 \equiv 0; \label{eqn:m_1}\\
    m_2 (\Etanextnext_j \tensor \Etaknext_i) = \Etakprev_{i+j}; 
    \label{eqn:m_2}\\
    m_2 (\EtanextnextH_j \tensor \Etaknext_i) = \EtakprevH_{i+j}; 
    \label{eqn:m_2H}\\
    m_2 (\ThetaHk \tensor \EtaprevnextH_i) = \Etaprevnext_i; \label{eqn:m_3a}\\
    m_3 (\ThetaHk \tensor \EtaprevnextH_j \tensor \Etanextnext_i) = 0; 
    \label{eqn:m_3b}\\
    m_3 (\EtaprevnextH_\ell \tensor \Etanextnext_j \tensor \Etaknext_i) = 
    \begin{cases}
      \ThetakH & \text{if } i = j = \ell = 0,\\
      0 & \text{otherwise.}
    \end{cases} \label{eqn:m_3c}
  \end{gather}
\end{lemma}

\begin{proof}
  All calculations here follow from considerations similar to that in 
  \fullref{sssec:BSDY1}, summarized in \fullref{rmk:computations}.  
  (\fullref{rmk:complication} does not apply here.) Below, we mention some 
  facts that may help the reader to verify the calculations.

  For \fullref{eqn:m_1}, note that, for each pair of two sets of attaching 
  circles $(\betas, \betas')$ except $(\betask, \betaskH)$ and $(\betaskH, 
  \betask)$, it is clear that there are no domains in $(\Sigma, \betas, 
  \betas')$ that could possibly contribute to $m_1$. For each of $(\betask, 
  \betaskH)$ and $(\betaskH, \betask)$, there are two bigons that cancel out.  
  (All Heegaard diagrams involved are \emph{nice} in the sense defined by 
  Sarkar and Wang \cite{SarWan10:Nice}, and so the bigon counts are 
  combinatorial.)

  For \fullref{eqn:m_2}, note that if we completely ignore the three 
  Hamiltonian perturbations $\betaskH$, then there are exactly four regions 
  that are not adjacent to $\bdy \Sigma \setminus \orarcs$, each a topological 
  disk made up of three arcs (i.e.\ a triangle).  Each of the twelve cases in 
  \fullref{eqn:m_2} arises from exactly one of these four regions, and each 
  region gives rise to exactly three equations. \fullref{eqn:m_2H} is 
  completely analogous.

  There is exactly one region that gives rise to each case in 
  \fullref{eqn:m_3a}, and \fullref{eqn:m_3b} follows from the fact that there 
  are no regions of the required index. Finally, for \fullref{eqn:m_3c}, 
  observe that there are no regions of the required index except when $i = j = 
  \ell = 0$. In the case $i = j = \ell = 0$, there is exactly one quadrilateral 
  formed by the four relevant intersection points, which reduces to one of the 
  four triangles in the previous paragraph as the small Hamiltonian 
  perturbation $\betakH$ is chosen to approach $\betak$; this is the triangle   
  $p$ mentioned in our description of $\betakH$, in the center of the left of 
  \fullref{fig:mult_hd}.
\end{proof}

\subsection{Proof of \texorpdfstring{\fullref{prop:elementary}}{the key 
    proposition} by polygon counting}
\label{ssec:proof_prop}

\begin{figure}[!htbp]
  \captionsetup{aboveskip={\dimexpr10pt+2pt+\sactualfontsize\relax}}
  \labellist
	\scriptsize\hair 2pt
	\pinlabel \textcolor{red}{4} [t] at 77 1
	\pinlabel \textcolor{red}{5} [t] at 95 1
	\pinlabel \textcolor{red}{3} [t] at 128 1
	\pinlabel \textcolor{red}{4} [t] at 145 1
	\pinlabel \textcolor{red}{2} [t] at 317 1
	\pinlabel \textcolor{red}{3} [t] at 337 1
	\pinlabel \textcolor{red}{1} [t] at 366 1
	\pinlabel \textcolor{red}{2} [t] at 385 1
	\pinlabel \textcolor{red}{6} at 50 78
	\pinlabel \textcolor{red}{5} at 63 76
	\pinlabel \textcolor{red}{6} at 74 78
	\pinlabel \textcolor{red}{2} at 402 76

  	\pinlabel {\small 6} [t] at 86 0
	\pinlabel {\small 5} [t] at 111 0
	\pinlabel {\small 4} [t] at 137 0
	\pinlabel {\small 3} [t] at 326 0
	\pinlabel {\small 2} [t] at 352 0
	\pinlabel {\small 1} [t] at 377 0
	\pinlabel {\small 7} at 57 77
	\pinlabel {\small 8} at 68 77

  \pinlabel $\etainfty_1$ at 68 138
  \pinlabel $\etazeroH_1$ at 142 129
  \pinlabel $\etazero_1$ at 154 112
  \pinlabel $\etaone_1$ at 170 111
  \pinlabel $\etaHone_1$ at 190 120
  \pinlabel $\etaHone_0$ at 117 117
  \pinlabel $\thetaoneH$ at 105 95
  \pinlabel $\thetaHone$ at 103 71
  \pinlabel $\etaone_0$ at 133 106
  \pinlabel $\etainfty_0$ at 150 87
  \pinlabel $\etazeroH_0$ at 110 46
  \pinlabel $\etazero_0$ at 130 73
	\pinlabel $x$ at 164 63
  \pinlabel $\Ham{x}$ at 144 45
	\pinlabel $y_2$ at 133 32
	\pinlabel $z_1$ at 76 53
	\pinlabel $y_1$ at 86 22
	\pinlabel $z_2$ at 390 58
	\pinlabel $y_3$ at 371 36
	
	\endlabellist
  \includegraphics{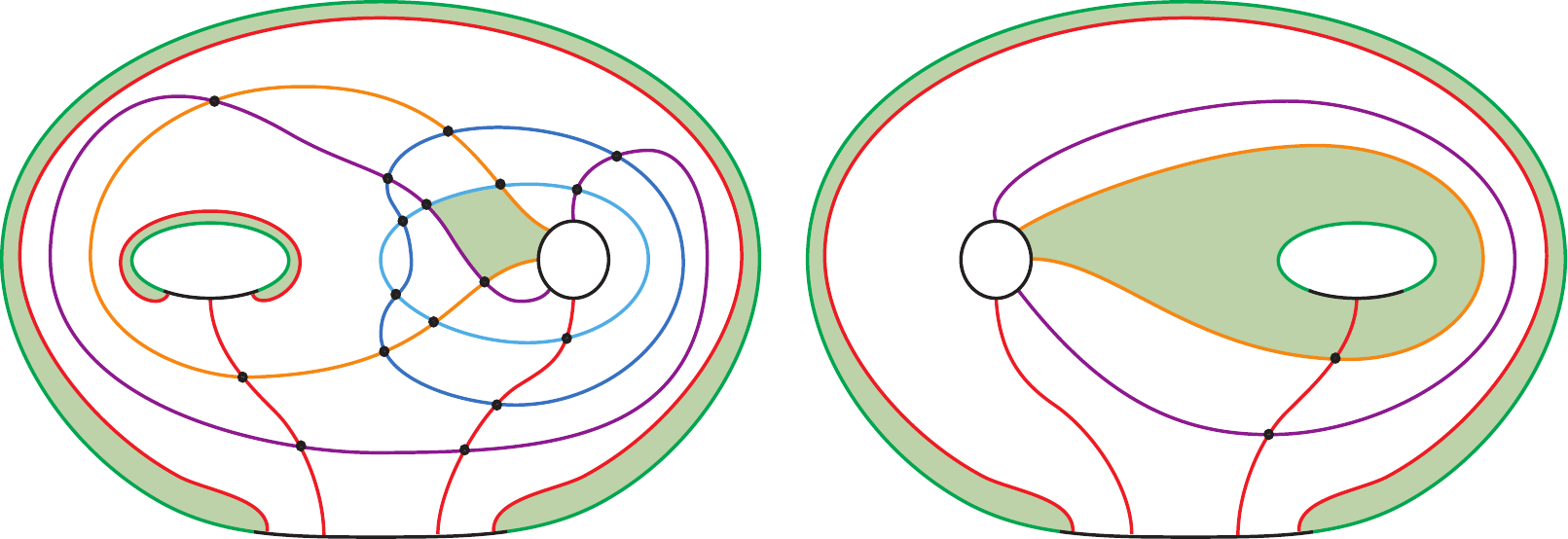}
  \caption{The bordered--sutured multi-diagram $\HDabH = (\Sigma, \alphas, 
    \set{\betask, \betaskH}_{k \in \set{\infty, 0, 1}}, \arcdiag, \psi)$.}
  \label{fig:mult_hd_2}
\end{figure}

We now define the type~$D$ morphisms needed in \fullref{lem:hom_alg}.

Consider the bordered--sutured multi-diagram $\HDabH = (\Sigma, \alphas, 
\set{\betask, \betaskH}_{k \in \set{\infty, 0, 1}}, \arcdiag, \psi)$, shown in 
\fullref{fig:mult_hd_2}, which is provincially admissible but not admissible.  
Note that $\BSDh (\alphas, \betask) = \BSDh (\Bk)$ as in 
\fullref{ssec:tancomp}. For each $k \in \set{\infty, 0, 1}$, let
\[
  \Etaknext = \Etaknext_0 \dirsum \Etaknext_1,
\]
and define
\begin{align*}
  f_k& \colon \BSDh (\alphas, \betask) \to \AZ \tensor \BSDh (\alphas, 
  \betasnext), & f_k& (\w) = \delta_2 (\Etaknext \tensor \w),\\
  \vphi_k& \colon \BSDh (\alphas, \betask) \to \AZ \tensor \BSDh (\alphas, 
  \betasprev), & \vphi_k& (\w) = \delta_3 (\Etanextnext \tensor \Etaknext 
  \tensor \w).
\end{align*}
Then $f_k$ and $\vphi_k$ are type~$D$ morphisms.

\begin{lemma}
  \label{lem:cond1_poly}
  The morphisms $f_k \colon \BSDh (\Bk) \to \BSDh (\Bnext)$ are type~$D$ 
  homomorphisms and thus satisfy Condition~(1) of \fullref{lem:hom_alg}.
\end{lemma}

\begin{proof}
  This and following lemmas in this subsection are proven by combining the 
  $\Ainf$-relations in \fullref{prop:ainf} and the computations in 
  \fullref{lem:m_123}. We show the details explicitly in this proof and omit 
  them in the sequel.
  
  With inputs $\Etaknext$ and $\w$, the $\Ainf$-relation in \fullref{prop:ainf} 
  is
  \[
    \begin{split}
      0& = \deltat_2 (\Etaknext \tensor \delta_1 (\w)) + \deltat_1 (\delta_2 
      (\Etaknext \tensor \w))\\
      & \quad + \delta_2 (m_1 (\Etaknext) \tensor \w) + \paren{d \tensor 
        \id_{\BSDh (\Bnext)}} \comp \delta_2 (\Etaknext \tensor \w).
    \end{split}
  \]
  \fullref{eqn:m_1} now implies the third term is zero, and the above reduces 
  to
  \[
    0 = f_k \comp \delta (\w) + \delta \comp f_k (\w) + d f_k (\w) = \bdy f_k 
    (\w),
  \]
  which is the condition that $f_k$ is a homomorphism.
\end{proof}

\begin{lemma}
  \label{lem:cond2_poly}
  The morphisms $f_k \colon \BSDh (\Bk) \to \BSDh (\Bnext)$ and $\vphi_k \colon 
  \BSDh (\Bk) \to \BSDh (\Bprev)$ satisfy Condition~(2) of 
  \fullref{lem:hom_alg}.
\end{lemma}

\begin{proof}
  This is a combination of \fullref{prop:ainf} with inputs $\Etanextnext$, 
  $\Etaknext$, and $\w$, together with \fullref{eqn:m_1} and the equality
  \begin{equation}
    \label{eqn:m_2_sum}
    m_2 (\Etanextnext \tensor \Etaknext) = 0,
  \end{equation}
  which is obtained by taking the sum of \fullref{eqn:m_2} over all $i, j \in 
  \set{0, 1}$.
\end{proof}

\begin{lemma}
  \label{lem:cond3_poly}
  The morphisms $f_k \colon \BSDh (\Bk) \to \BSDh (\Bnext)$ and $\vphi_k \colon 
  \BSDh (\Bk) \to \BSDh (\Bprev)$ satisfy Condition~(3) of 
  \fullref{lem:hom_alg}.
\end{lemma}

\begin{proof}
  We begin by letting
  \[
    \EtaknextH = \EtaknextH_0 \dirsum \EtaknextH_1,
  \]
  and defining the morphisms
  \begin{align*}
    \fkH& \colon \BSDh (\alphas, \betask) \to \AZ \tensor \BSDh (\alphas, 
    \betasnextH), & \fkH& (\w) = \delta_2 (\EtaknextH \tensor \w),\\
    \vphikH& \colon \BSDh (\alphas, \betask) \to \AZ \tensor \BSDh (\alphas, 
    \betasprevH), & \vphikH& (\w) = \delta_3 (\EtanextnextH \tensor \Etaknext 
    \tensor \w),
  \end{align*}
  which are analogous to $f_k$ and $\vphi_k$. We also define the morphisms
  \begin{align*}
    \PhikH& \colon \BSDh (\alphas, \betask) \to \AZ \tensor \BSDh (\alphas, 
    \betaskH), & \PhikH& (\w) = \delta_2 (\ThetakH \tensor \w),\\
    \PhiHk& \colon \BSDh (\alphas, \betaskH) \to \AZ \tensor \BSDh (\alphas, 
    \betask), & \PhiHk& (\w) = \delta_2 (\ThetaHk \tensor \w),\\
    \Xik& \colon \BSDh (\alphas, \betask) \to \AZ \tensor \BSDh (\alphas, 
    \betasnext), & \Xik& (\w) = \delta_3 (\ThetaHnext \tensor \EtaknextH 
    \tensor \w),
  \end{align*}

  We first claim that $f_{k+2}$ is homotopic to $\PhiHk \comp \fprevH$. Indeed, 
  this follows from \fullref{prop:ainf} with inputs $\ThetaHk$, 
  $\EtaprevnextH$, and $\w$, together with \fullref{eqn:m_1} and the equation
  \begin{equation}
    \label{eqn:m_3a_sum}
    m_2 (\ThetaHk \tensor \EtaprevnextH) = \Etaprevnext,
  \end{equation}
  which is obtained by taking the sum of \fullref{eqn:m_3a} over $i \in \set{0, 
    1}$.

  Similarly, we see that $\vphi_{k+1}$ is homotopic to $\PhiHk \comp \vphinextH 
  + \Xiprev \comp f_{k+1}$. (Note that $\vphi_{k+1}$ is not a homomorphism.) 
  This follows from \fullref{prop:ainf} with inputs $\ThetaHk$, 
  $\EtaprevnextH$, $\Etanextnext$, and $\w$, together with \fullref{eqn:m_1}, 
  \fullref{eqn:m_3a_sum}, and the equations
  \begin{gather}
    m_2 (\EtaprevnextH \tensor \Etanextnext) = 0, \label{eqn:m_2H_sum}\\
    m_3 (\ThetaHk \tensor \EtaprevnextH \tensor \Etanextnext) = 0, 
    \label{eqn:m_3b_sum}
  \end{gather}
  which are obtained respectively from \fullref{eqn:m_2H} and 
  \fullref{eqn:m_3b}.

  Thus, we see that $f_{k+2} \comp \vphi_k + \vphi_{k+1} \comp f_k$ is 
  homotopic to
  \[
    \PhiHk \comp (\fprevH \comp \vphi_k + \vphinextH \comp f_k) + \Xiprev \comp 
    f_{k+1} \comp f_k.
  \]
  By \fullref{lem:cond2_poly}, the last term is nullhomotopic. Now we claim 
  that $\fprevH \comp \vphi_k + \vphinextH \comp f_k$ is homotopic to $\PhikH$.  
  Indeed, this follows from \fullref{prop:ainf} with inputs $\EtaprevnextH$, 
  $\Etanextnext$, and $\Etaknext$, together with \fullref{eqn:m_1}, 
  \fullref{eqn:m_2_sum}, \fullref{eqn:m_2H_sum}, and the equation
  \begin{equation}
    \label{eqn:m_3c_sum}
    m_3 (\EtaprevnextH \tensor \Etanextnext \tensor \Etaknext) = \ThetakH,
  \end{equation}
  which is obtained from \fullref{eqn:m_3c}.

  This means that $f_{k+2} \comp \vphi_k + \vphi_{k+1} \comp f_k$ is homotopic 
  to $\PhiHk \comp \PhikH$. Now recall that the Hamiltonian isotopy between 
  $\betask$ and $\betaskH$ can be equivalently viewed as a deformation of the 
  almost complex structure on $\Sigma \cross D_2$, and 
  \cite[Proposition~11.4]{Lip06:HFCylindrical}, applied to the bordered--sutured 
  Floer setting, implies that $\PhikH$ and $\PhiHk$ are each homotopic to the 
  homomorphism induced by the corresponding deformation. (Note that in 
  \cite{Lip06:HFCylindrical}, a small Hamiltonian isotopy of $\betas$ is 
  denoted by ${\betas}'$, while ${\betas}^{H}$ instead denotes the result of a 
  handleslide.) By \cite[Lemma~9.4 and Lemma~9.5]{Lip06:HFCylindrical}, $\PhiHk 
  \comp \PhikH$ is then homotopic to the homomorphism induced by the identity 
  deformation, and hence homotopic to the identity.
\end{proof}

\begin{proof}[Proof of \fullref{prop:elementary}, via holomorphic polygon 
  counts]
  We apply \fullref{lem:hom_alg}, the conditions of which are satisfied 
  according to \fullref{lem:cond1_poly}, \fullref{lem:cond2_poly}, and 
  \fullref{lem:cond3_poly}.
\end{proof}

\section{The skein relation via direct computation}
\label{sec:skein_via_computation}

In this section, we give an alternative proof of \fullref{prop:elementary} by 
direct computation.

Consider again the bordered--sutured multi-diagram $\HDab$ in 
\fullref{fig:mult_hd_2}, and let $i, j \in \set{0, 1}$.

\begin{figure}[!htbp]
  \captionsetup{aboveskip={\dimexpr10pt+2pt+\sactualfontsize\relax}}
  \labellist
	\scriptsize\hair 2pt
	\pinlabel \textcolor{red}{4} [t] at 77 1
	\pinlabel \textcolor{red}{5} [t] at 95 1
	\pinlabel \textcolor{red}{3} [t] at 128 1
	\pinlabel \textcolor{red}{4} [t] at 145 1
	\pinlabel \textcolor{red}{2} [t] at 317 1
	\pinlabel \textcolor{red}{3} [t] at 337 1
	\pinlabel \textcolor{red}{1} [t] at 366 1
	\pinlabel \textcolor{red}{2} [t] at 385 1
	\pinlabel \textcolor{red}{6} at 50 78
	\pinlabel \textcolor{red}{5} at 63 76
	\pinlabel \textcolor{red}{6} at 74 78
	\pinlabel \textcolor{red}{2} at 402 76

  	\pinlabel {\small 6} [t] at 86 0
	\pinlabel {\small 5} [t] at 111 0
	\pinlabel {\small 4} [t] at 137 0
	\pinlabel {\small 3} [t] at 326 0
	\pinlabel {\small 2} [t] at 352 0
	\pinlabel {\small 1} [t] at 377 0
	\pinlabel {\small 7} at 57 77
	\pinlabel {\small 8} at 68 77

  \pinlabel $\etainfty_1$ at 68 138
  \pinlabel $\etazero_1$ at 154 112
  \pinlabel $\etaone_1$ at 170 111
  \pinlabel $\etaone_0$ at 133 106
  \pinlabel $\etainfty_0$ at 150 87
  \pinlabel $\etazero_0$ at 130 73
	\pinlabel $x$ at 164 63
  \pinlabel $y_2$ at 133 32
	\pinlabel $z_1$ at 76 53
	\pinlabel $y_1$ at 86 22
	\pinlabel $z_2$ at 390 58
	\pinlabel $y_3$ at 371 36
	
	\endlabellist
  \includegraphics{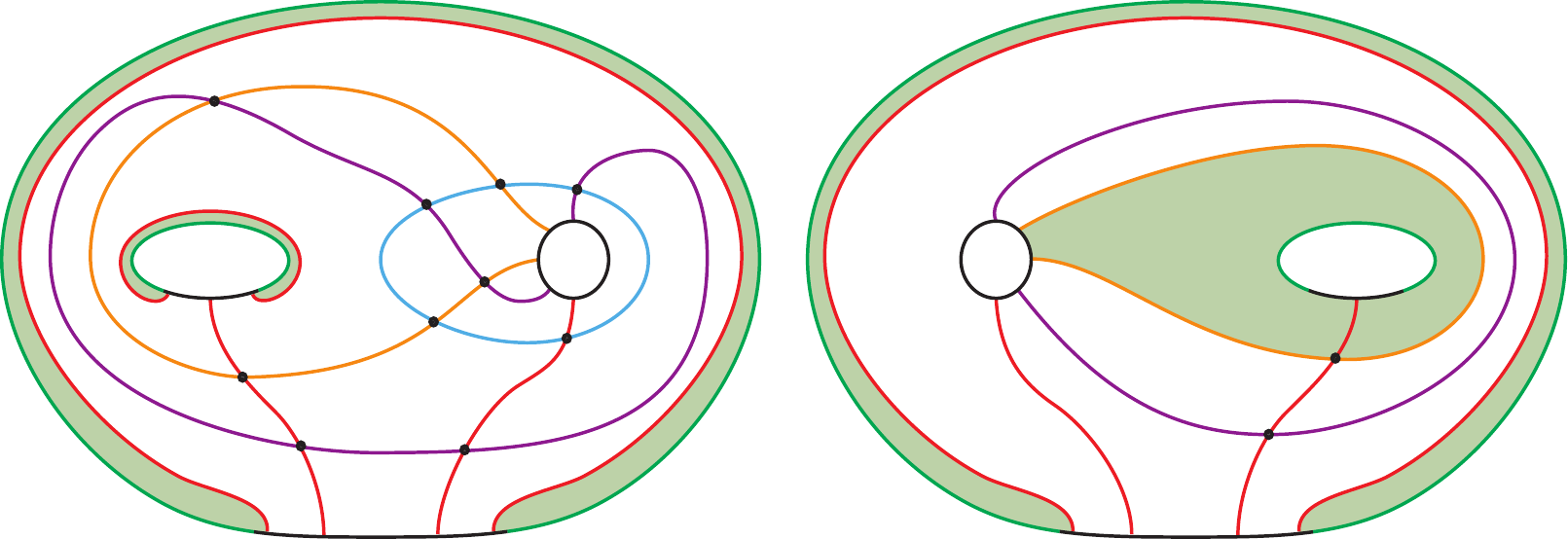}
  \caption{The bordered--sutured multi-diagram $\HDab = (\Sigma, \alphas, 
    \set{\betask}_{k \in \set{\infty, 0, 1}}, \arcdiag, \psi)$.}
  \label{fig:mult_hd_3}
\end{figure}

Given generators $\w^k \in \genset (\HD_k)$, with corresponding intersection 
points $w^k$, and a sequence of Reeb chords $\vec{\rho} = (\rho^1, \dotsc, 
\rho^m)$, let $\Tri_i (\w^{k+1}, \w^k; \vec{\rho})$ be the space of immersed 
disks $p$ in $\Sigma$ with the following properties:
\begin{itemize}
  \item $\bdy p \subset \alphas \union \betask \union \betasnext \union 
    \betasprev \union \orarcs$;
  \item If we denote by $\alpha^a$ an $\alpha$-arc in $\alphas^a$, then, 
    traversing $\bdy p$ in the standard orientation, we encounter
    \[
      w^k, \alpha^a, \rho^1, \alpha^a, \rho^2, \alpha^a, \dotsc, \alpha^a, 
      \rho^m, \alpha^a, w^{k+1}, \betanext, \etak_i, \betak, w^k,
    \]
    in this order; and
  \item At each of the turning points above, i.e.\ $w^k$, $w^{k+1}$, $\etak_i$, 
    and $\set{\bdy \rho^\ell} \subset \matchedpts$, the angle is acute.
\end{itemize}
Similarly, let $\Quad_{ij} (\w^{k+2}, \w^k; \vec{\rho})$ be the space of 
immersed disks $p$ in $\Sigma$ with the following properties:
\begin{itemize}
  \item $\bdy p \subset \alphas \union \betask \union \betasnext \union 
    \betasprev \union \orarcs$;
  \item If we denote by $\alpha^a$ an $\alpha$-arc in $\alphas^a$, then, 
    traversing $\bdy p$ in the standard orientation, we encounter
    \[
      w^k, \alpha^a, \rho^1, \alpha^a, \rho^2, \alpha^a, \dotsc, \alpha^a, 
      \rho^m, \alpha^a, w^{k+2}, \betaprev, \etanext_j, \betanext, \etak_i, 
      \betak, w^k,
    \]
    in this order; and
  \item At each of the turning points above, i.e.\ $w^k$, $w^{k+2}$, 
    $\etanext_j$, $\etak_i$, and $\set{\bdy \rho^\ell} \subset \matchedpts$, 
    the angle is acute.
\end{itemize}

We use the spaces above to define the morphisms
\begin{align*}
  \fcombki& \colon \BSDh (\Bk) \to \BSDh (\Bnext), & \fcombki (\w^k)& = 
  \sum_{\w^{k+1} \in \genset (\HD_{k+1})} \sum_{\Tri_i (\w^{k+1}, \w^k; 
    \vec{\rho})} a \paren{- \vec{\rho}} \tensor \w^{k+1},\\
  \vphicombkij& \colon \BSDh (\Bk) \to \BSDh (\Bprev), & \vphicombkij (\w^k)& = 
  \sum_{\w^{k+2} \in \genset (\HD_{k+2})} \sum_{\Quad_{ij} (\w^{k+2}, \w^k; 
    \vec{\rho})} a \paren{- \vec{\rho}} \tensor \w^{k+2},
\end{align*}
and the morphisms
\begin{align*}
  \fcombk& \colon \BSDh (\Bk) \to \BSDh (\Bnext), & \fcombk& = \fcomb_{k,0} + 
  \fcomb_{k,1},\\
  \vphicombk& \colon \BSDh (\Bk) \to \BSDh (\Bprev), & \vphicombk& = 
  \vphicomb_{k,00} + \vphicomb_{k,01} + \vphicomb_{k,10} + \vphicomb_{k,11}.\\
\end{align*}

Since the definitions of $\fcombki$ and $\vphicombkij$ refer to actual immersed 
disks rather than holomorphic representatives, by inspecting $\HDab$, we may 
directly compute:
\begin{align*}
  \fcomb_{1,0}(\x)& = (4, 6) \tensor \y_1 + (12, 3) \tensor \y_2 + (4, 56) 
  \tensor \y_2 + (2) \tensor \y_3,\\
  \fcomb_{1,1} (\x)& = (4, 6, 7, 8) \tensor \y_1 + I \tensor \y_2,\\
  \fcomb_{\infty,0} (\y_1)& = (45, 6) \tensor \z_1 + (7, 8) \tensor \z_1,\\
  \fcomb_{\infty,0} (\y_2)& = (4, 6) \tensor \z_1 + (4, 56, 2) \tensor \z_2,\\
  \fcomb_{\infty,0} (\y_3)& = I \tensor \z_2,\\
  \fcomb_{\infty,1} (\y_1)& = I \tensor \z_1,\\
  \fcomb_{\infty,1} (\y_2)& = (12, 3, 4, 6) \tensor \z_1 + (2) \tensor \z_2 + 
  (12, 3, 4, 56, 2) \tensor \z_2,\\
  \fcomb_{\infty,1} (\y_3)& = (1, 23) \tensor \z_2,\\
  \fcomb_{0,0} (\z_1)& = (5) \tensor \x + (5, 12, 3, 4, 56) \tensor \x + (45, 
  6, 7, 8, 5) \tensor \x,\\
  \fcomb_{0,0} (\z_2)& = (1, 3, 4, 56) \tensor \x,\\
  \fcomb_{0,1} (\z_1)& = (5, 12, 3) \tensor \x + (7, 8, 5) \tensor \x,\\
  \fcomb_{0,1} (\z_2)& = (1, 3) \tensor \x.
\end{align*}
See \fullref{fig:comm_diag_f} for a graphical representation of $\fcombki$.

\begin{figure}[!htbp]
  \vspace*{\scractualfontsize}
  \labellist
  \scriptsize\hair 2pt
   \pinlabel {$(4,6,7,8,5)$} [b] at 100 261

   {\color{bluearr}%
     \pinlabel {$(4,6)$} [b] at 21 231
     \pinlabel {$(4,56)$} [r] at 82 215
     \pinlabel {$(12,3)$} [l] at 117 215
     \pinlabel {$(2)$} [l] at 165 233
   }
   {\color{redarr}%
     \pinlabel {$(4,6,7,8)$} [r] at 49 226
   }
   \pinlabel {$(5)$} [b] at 53 193
   \pinlabel {$(1,3)$} [b] at 145 190 
   \pinlabel {$(7,8,5,12,3)$} [b] at 55 181
   \pinlabel {$(7,8,5,2)$} [t] at 95 172
   {\color{bluearr}%
     \pinlabel {$(7,8)$} [r] at 4 151
     \pinlabel {$(45,6)$} [l] at 22 156
     \pinlabel {$(4,6)$} [l] at 68 158
     \pinlabel {$(4,56,2)$} [l] at 120 163
   }
   {\color{redarr}%
     \pinlabel {$(12,3,4,6)$} [r] at 62 166
     \pinlabel {$(2)$} [r] at 119 157
     \pinlabel {$(12,3,4,56,2)$} [l] at 142 170
     \pinlabel {$(1,23)$} [l] at 185 152
   }
   \pinlabel {$(5,12,3,4,6)$} [r] at 27 120
   \pinlabel {$(1,3,4,6)$} [b] at 99 137
   \pinlabel {$(5,2)$} [b] at 99 129
   \pinlabel {$(5,12,3,4,56,2)$} [b] at 99 120
   \pinlabel {$(1,3,4,56,2)$} [l] at 167 120
   {\color{bluearr}%
     \pinlabel {$(45,6,7,8,5)$} [r] at 34 88
     \pinlabel {$(5)$} [r] at 63 100
     \pinlabel {$(5,12,3,4,56)$} [l] at 85 110
     \pinlabel {$(1,3,4,56)$} [l] at 145 96
   }
   {\color{redarr}%
     \pinlabel {$(7,8,5)$} [r] at 54 96
     \pinlabel {$(5,12,3)$} [l] at 66 100
     \pinlabel {$(1,3)$} [l] at 127 99
   }
   \pinlabel {$(4,6,7,8,5)$} [l] at 132 72
   {\color{bluearr}%
     \pinlabel {$(4,6)$} [r] at 22 52
     \pinlabel {$(4,56)$} [r] at 79 43
     \pinlabel {$(12,3)$} [l] at 115 43
     \pinlabel {$(2)$} [l] at 163 49
   }
   {\color{redarr}%
     \pinlabel {$(4,6,7,8)$} [r] at 47 46
   }
   \pinlabel {$(5)$} [b] at 53 20
   \pinlabel {$(7,8,5,12,3)$} [b] at 55 8
   \pinlabel {$(7,8,5,2)$} [b] at 101 1
   \pinlabel {$(1,3)$} [b] at 145 18

   \pinlabel {\large $\x$} [ ] at 100 244
   \pinlabel {\large $\y_1$} [ ] at 12 185
   \pinlabel {\large $\y_2$} [ ] at 101 185
   \pinlabel {\large $\y_3$} [ ] at 186 185
   \pinlabel {\large $\z_1$} [ ] at 57 128
   \pinlabel {\large $\z_2$} [ ] at 145 128
   \pinlabel {\large $\x$} [ ] at 100 69
   \pinlabel {\large $\y_1$} [ ] at 15 12
   \pinlabel {\large $\y_2$} [ ] at 101 12
   \pinlabel {\large $\y_3$} [ ] at 186 12

  \endlabellist
  \includegraphics[scale=1.35]{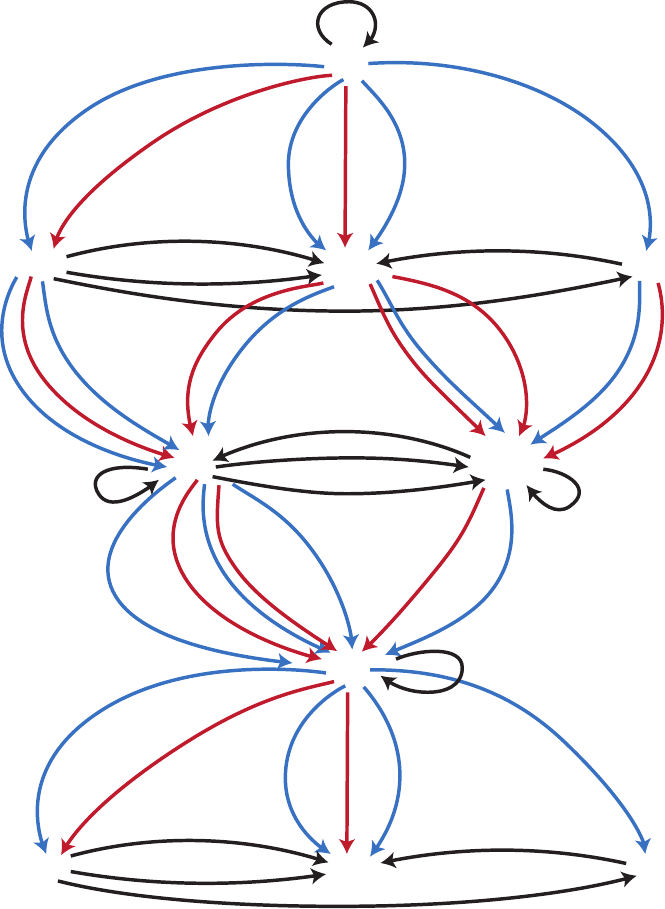}
  \caption{A graphical representation of the morphisms $\fcombki$. Black arrows 
    represent $\delta$, blue arrows represent $\fcomb_{k,0}$, and red arrows 
    represent $\fcomb_{k,1}$. The repeated level, from $\BSDh (\B_1)$ to $\BSDh 
    (\B_\infty)$, is drawn here to facilitate checking 
    \fullref{lem:cond2_comp}.}
  \label{fig:comm_diag_f}
\end{figure}

\begin{remark}
  \label{rmk:blocked_disks}
  Note that we have had to observe that $(12, 23) = (45, 56) = 0$ in this 
  calculation; for instance, $\fcomb_{\infty,0} (\y_1)$ in fact has a term 
  $(45, 56, 2) \tensor \z_2 = 0$. To see the corresponding immersed disk, which 
  has multiplicity $2$ at two non-adjacent regions, one may, for example, 
  concatenate the domain for the term $(5) \tensor \y_2$ in $\delta (\y_1)$ 
  with the domain for the term $(4, 56, 2) \tensor \z_2$ in $\fcomb_{\infty,0} 
  (\y_2)$; see \fullref{fig:mult_hd_4}. In the holomorphic interpretation (see 
  \fullref{prop:poly_comp} below), these cases represent situations where we 
  already know that $\chi (\source) = 1$ (cf.\ \fullref{rmk:complication}, 
  where we may have $\chi (\source) < 1$), but no partition $\vec{P}$ supports 
  a contribution to $f_k$.
\end{remark}

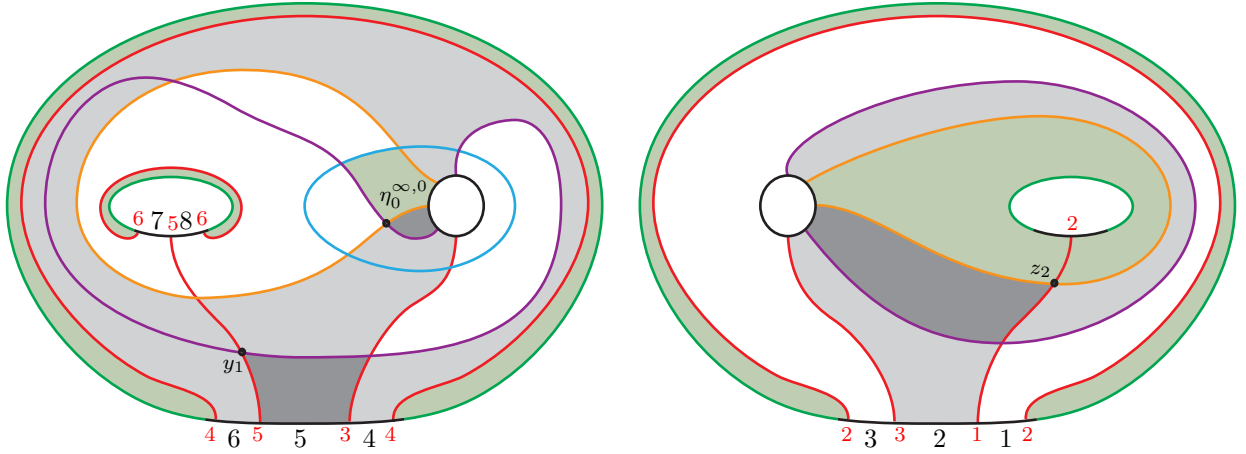
\begin{figure}[!htbp]
  \captionsetup{aboveskip={\dimexpr10pt+2pt+\sactualfontsize\relax}}
  \labellist
	\scriptsize\hair 2pt
	\pinlabel \textcolor{red}{4} [t] at 77 1
	\pinlabel \textcolor{red}{5} [t] at 95 1
	\pinlabel \textcolor{red}{3} [t] at 128 1
	\pinlabel \textcolor{red}{4} [t] at 145 1
	\pinlabel \textcolor{red}{2} [t] at 317 1
	\pinlabel \textcolor{red}{3} [t] at 337 1
	\pinlabel \textcolor{red}{1} [t] at 366 1
	\pinlabel \textcolor{red}{2} [t] at 385 1
	\pinlabel \textcolor{red}{6} at 50 78
	\pinlabel \textcolor{red}{5} at 63 76
	\pinlabel \textcolor{red}{6} at 74 78
	\pinlabel \textcolor{red}{2} at 402 76

  	\pinlabel {\small 6} [t] at 86 0
	\pinlabel {\small 5} [t] at 111 0
	\pinlabel {\small 4} [t] at 137 0
	\pinlabel {\small 3} [t] at 326 0
	\pinlabel {\small 2} [t] at 352 0
	\pinlabel {\small 1} [t] at 377 0
	\pinlabel {\small 7} at 57 77
	\pinlabel {\small 8} at 68 77

  \pinlabel $\etainfty_0$ at 150 87
  \pinlabel $y_1$ at 86 22
	\pinlabel $z_2$ at 390 58
	
	\endlabellist
  \includegraphics{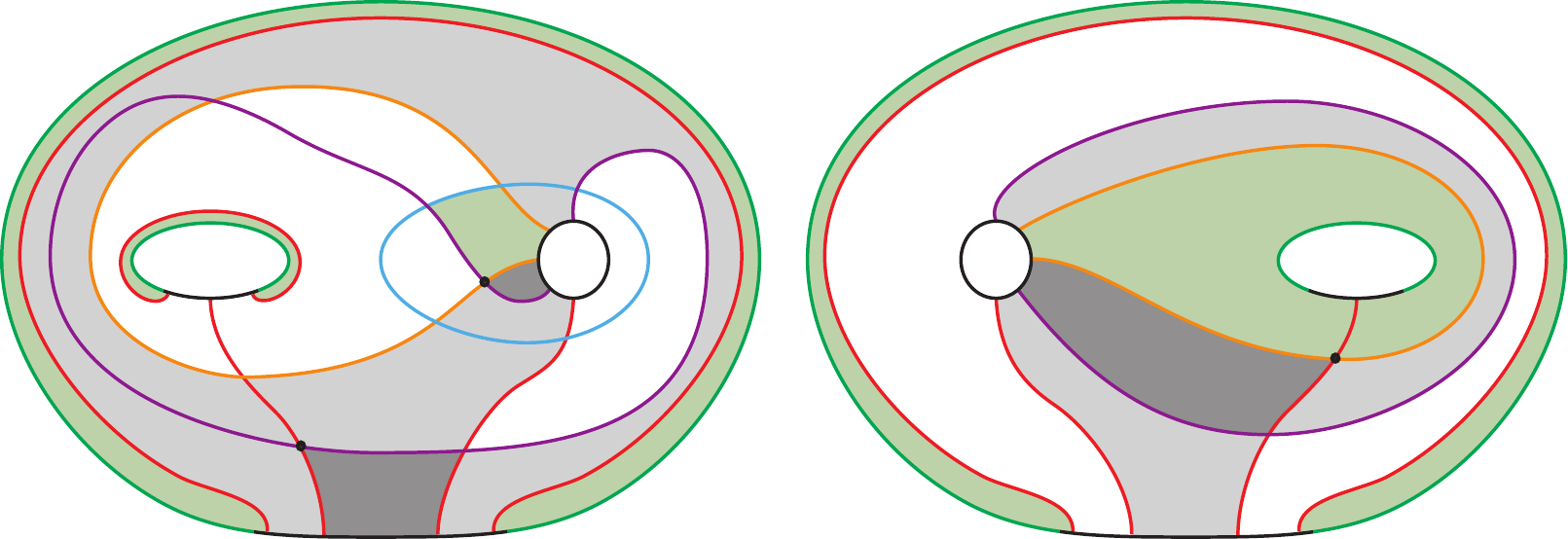}
  \caption{An immersed disk in $\Tri_0 (\z_2, \y_1; (-\rho_{45}, -\rho_{56}, 
    -\rho_2))$ that does not contribute to $\fcomb_{\infty,0} (\y_1)$ because 
    $(45, 56, 2) = 0$. Regions with multiplicity $1$ are shaded in light gray, 
    and regions with multiplicity $2$ in dark gray.}
    \label{fig:mult_hd_4}
\end{figure}

\begin{lemma}
  \label{lem:cond1_comp}
  The morphisms $\fcombk \colon \BSDh (\Bk) \to \BSDh (\Bnext)$ are type~$D$ 
  homomorphisms and thus satisfy Condition~(1) of \fullref{lem:hom_alg}.
\end{lemma}

\begin{proof}
  In fact, $\fcombki$ is a homomorphism for $i \in \set{0, 1}$. The proof is a 
  direct computation, and the reader may find \fullref{fig:comm_diag_f} helpful 
  in checking this computation. We provide a sample calculation here: 
  \begin{align*}
    \delta \comp \fcomb_{\infty,0} (\y_1) & = \deltat ((45, 6) \tensor \z_1 + 
    (7, 8) \tensor \z_1)\\
    & = (45, 6, 5, 12, 3, 4, 6) \tensor \z_1 + (45, 6, 5, 2) \tensor \z_2 + 
    (45, 6, 5, 12, 3, 4, 56, 2) \tensor \z_2\\
    & \quad + (7, 8, 5, 12, 3, 4, 6) \tensor \z_1 + (7, 8, 5, 2) \tensor \z_2 + 
    (7, 8, 5, 12, 3, 4, 56, 2) \tensor \z_2\\
    & = (45, 6, 5, 2) \tensor \z_2 + (7, 8, 5, 12, 3, 4, 6) \tensor \z_1 + (7, 
    8, 5, 2) \tensor \z_2\\
    & \quad + (7, 8, 5, 12, 3, 4, 56, 2) \tensor \z_2,\\
    \fcomb_{\infty,0} \comp \delta (\y_1) & = \paren{\mu \tensor \id_{\BSDh 
        (\B_1)}} \comp (\id_A \tensor \fcomb_{\infty,1}) ((5) \tensor \y_2 + 
    (7, 8, 5, 12, 3) \tensor \y_2 + (7, 8, 5, 2) \tensor \y_3)\\
    & = (5, 4, 6) \tensor \z_1 + (5, 4, 56, 2) \tensor \z_2 + (7, 8, 5, 12, 3, 
    4, 6) \tensor \z_1\\
    & \quad + (7, 8, 5, 12, 3, 4, 56, 2) \tensor \z_2 + (7, 8, 5, 2) \tensor 
    \z_2,\\
    d \fcomb_{\infty,0} (\y_1) & = \paren{d \tensor \id_{\BSDh (\B_1)}} (\y_1) 
    = d (45, 6) \tensor \z_1 + d (7, 8) \tensor \z_1 = (5, 4, 6) \tensor \z_1.
  \end{align*}
  Observing that $(45, 6, 5, 2) = (5, 4, 56, 2)$ (see \fullref{fig:alg_mult} 
  and the surrounding text), we see that the sum of the three terms above is 
  zero.
\end{proof}

We may also directly compute:
\begin{align*}
  \vphicomb_{\infty,00} (\y_2)& = (4, 56) \tensor \x,\\
  \vphicomb_{\infty,10} (\y_2)& = I \tensor \x,\\
  \vphicomb_{\infty,11} (\y_2)& = (12, 3) \tensor \x,\\
  \vphicomb_{0,00} (\z_1)& = (45, 6) \tensor \y_1,\\
  \vphicomb_{0,01} (\z_1)& = I \tensor \y_1,\\
  \vphicomb_{0,11} (\z_1)& = (7, 8) \tensor \y_1,\\
  \vphicomb_{0,10} (\z_2)& = (1, 23) \tensor \y_3,\\
  \vphicomb_{0,11} (\z_2)& = I \tensor \y_3.
\end{align*}
All other combinations of $i, j, k, \w$ yield $\vphicombkij (\w) = 0$. See 
\fullref{fig:comm_diag_phi} for a graphical representation of the sums 
$\vphicombk$.

\begin{figure}[!htbp]
  \labellist
  \scriptsize\hair 2pt
   \pinlabel {$(4,6,7,8,5)$} [t] at 108 208
   \pinlabel {$(5)$} [b] at 47 171
   \pinlabel {$(7,8,5,12,3)$} [b] at 47 160
   \pinlabel {$(7,8,5,2)$} [t] at 141 155
   \pinlabel {$(1,3)$} [b] at 138 169
   {\color{greenarr}%
     \pinlabel {$(4,56)$} [r] at 70 133
     \pinlabel {$(12,3)$} [l] at 114 133
   }
   \pinlabel {$(5,12,3,4,6)$} [r] at 15 111
   \pinlabel {$(1,3,4,6)$} [b] at 92 123
   \pinlabel {$(5,2)$} [t] at 92 116
   \pinlabel {$(5,12,3,4,56,2)$} [t] at 92 106
   \pinlabel {$(1,3,4,56,2)$} [l] at 167 111
   {\color{greenarr}%
     \pinlabel {$(7,8)$} [r] at 12 75
     \pinlabel {$(45,6)$} [l] at 46 63
     \pinlabel {$(1,23)$} [l] at 169 65
   }
   \pinlabel {$(4,6,7,8,5)$} [t] at 113 54
   \pinlabel {$(5)$} [b] at 47 20
   \pinlabel {$(7,8,5,12,3)$} [b] at 47 10
   \pinlabel {$(7,8,5,2)$} [b] at 94 1
   \pinlabel {$(1,3)$} [b] at 138 17
   
   \pinlabel {\large $\x$} [ ] at 93 217
   \pinlabel {\large $\y_1$} [ ] at 5 165
   \pinlabel {\large $\y_2$} [ ] at 94 165
   \pinlabel {\large $\y_3$} [ ] at 180 165
   \pinlabel {\large $\z_1$} [ ] at 47 115
   \pinlabel {\large $\z_2$} [ ] at 137 115
   \pinlabel {\large $\x$} [ ] at 93 63
   \pinlabel {\large $\y_1$} [ ] at 7 15
   \pinlabel {\large $\y_2$} [ ] at 94 15
   \pinlabel {\large $\y_3$} [ ] at 180 15

  \endlabellist
  \includegraphics[scale=1.35]{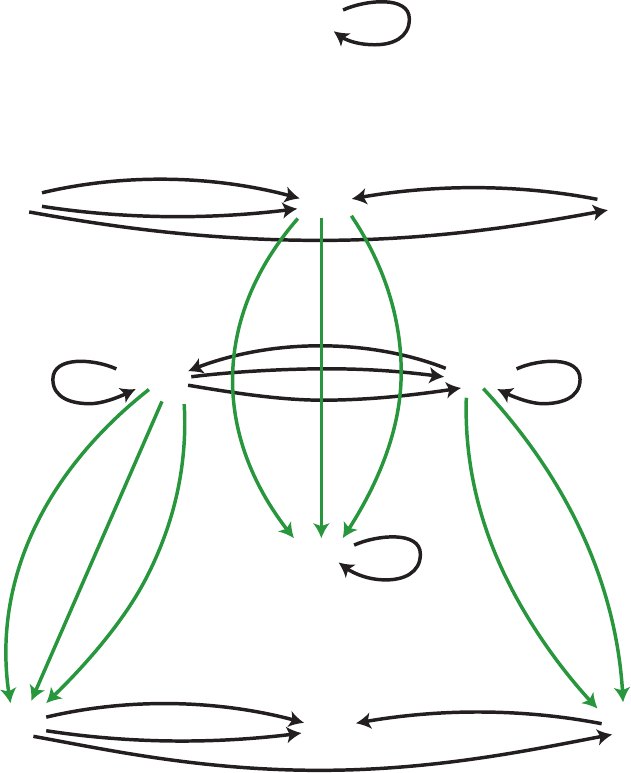}
  \caption{A graphical representation of the morphisms $\vphicombk$. Black 
    arrows represent $\delta$ and green arrows represent $\vphicombk$.}
  \label{fig:comm_diag_phi}
\end{figure}
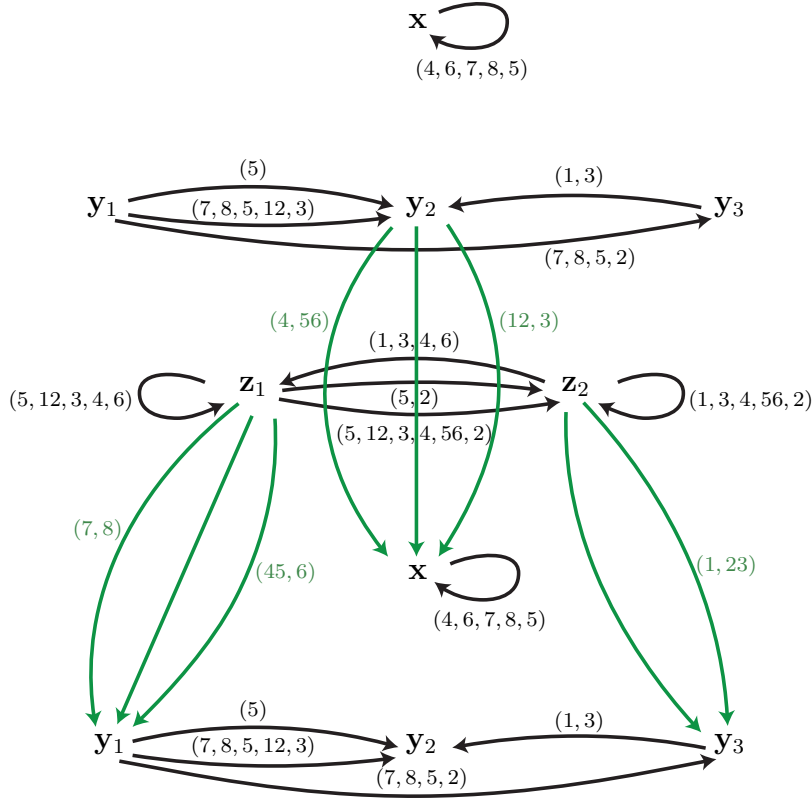

\begin{lemma}
  \label{lem:cond2_comp}
  The morphisms $\fcombk \colon \BSDh (\Bk) \to \BSDh (\Bnext)$ and $\vphicombk 
  \colon \BSDh (\Bk) \to \BSDh (\Bprev)$ satisfy Condition~(2) of 
  \fullref{lem:hom_alg}.
\end{lemma}

\begin{proof}
  The proof is a direct computation analogous to that of 
  \fullref{lem:cond1_comp}. The reader may find it helpful to peruse both 
  \fullref{fig:comm_diag_f} and \fullref{fig:comm_diag_phi}.
\end{proof}

\begin{lemma}
  \label{lem:cond3_comp}
  The morphisms $\fcombk \colon \BSDh (\Bk) \to \BSDh (\Bnext)$ and $\vphicombk 
  \colon \BSDh (\Bk) \to \BSDh (\Bprev)$ satisfy Condition~(3) of 
  \fullref{lem:hom_alg}, with $\kappa_k \equiv 0$.
\end{lemma}

\begin{proof}
  The proof is a direct computation analogous to that of 
  \fullref{lem:cond1_comp}. Again, it may be helpful to peruse both 
  \fullref{fig:comm_diag_f} and \fullref{fig:comm_diag_phi}.
\end{proof}

\begin{proof}[Proof of \fullref{prop:elementary}, via direction computation]
  We apply to $\fcombk$ and $\vphicombk$ \fullref{lem:hom_alg}, the conditions 
  of which are satisfied according to \fullref{lem:cond1_comp}, 
  \fullref{lem:cond2_comp}, and \fullref{lem:cond3_comp}.
\end{proof}

We end this section by relating this proof with the one in 
\fullref{sec:skein_via_polygon}.

\begin{proposition}
  \label{prop:poly_comp}
  The morphisms $\fcombk$ and $\vphicombk$ coincide with $f_k$ and $\vphi_k$ 
  respectively.
\end{proposition}

\begin{proof}
  The morphisms $f_k$ are defined using $\delta_2$, which in turn is defined by 
  counting holomorphic polygons, i.e.\ by considering the moduli space 
  $\embmoduli^B (\w^{k+1}, \Etaknext, \w^k; \source, \vec{P})$; the morphisms 
  $\vphi_k$ are defined similarly. As in the proof of \fullref{lem:m_123}, 
  these maps may be computed from considerations similar to that in 
  \fullref{sssec:BSDY1}, summarized in \fullref{rmk:computations}, with the 
  index formula \cite[Proposition~5.3.5]{Zar11:BSFH} replaced by the one for 
  polygons \cite[Proposition~4.9]{LipOzsThu16:BFHBranchedDoubleSSII}.   
  \fullref{rmk:no_full_reeb} and \fullref{rmk:complication} both apply. This 
  analysis yields the fact that the holomorphic polygon counts are in fact 
  counts of immersed disks, and consideration of the term $\sum \bdy_j (B) 
  \cdot \bdy_\ell (B)$ in 
  \cite[Proposition~4.9]{LipOzsThu16:BFHBranchedDoubleSSII} implies that the 
  angles at the intersection points are acute.
\end{proof}

\section{Comparison with the skein relation for knot Floer homology}
\label{sec:comparison}

In this section, we prove that the skein exact triangle in 
\fullref{cor:arbitrary} agrees with the one in \cite{Man07:HFKUnorientedSkein}, 
using a pairing theorem for triangles adapted from 
\cite{LipOzsThu16:BFHBranchedDoubleSSII}. Precisely, we prove the following, 
elaborated version of \fullref{thm:agree}.

\begin{theorem}
  \label{thm:agree_2}
  There exists a connected link projection of $L_\infty$ and a choice of edges, 
  such that the special Heegaard diagram $\HDMan_\infty$, when once 
  quasi-stabilized, can be obtained by gluing together two bordered--sutured 
  Heegaard diagrams $\HD'$ and $\HD_\infty$, where $\HD'$ is admissible, and 
  $\HD_\infty$ is as described in \fullref{ssec:tancomp} and used in the proof 
  of \fullref{thm:main}; moreover, for each $k \in \set{\infty, 0, 1}$, the 
  diagram
  \[
    \xymatrix{
      \BSAh (\HD') \boxtensor \BSDh (\HD_k) \ar[rrr]^{%
        \mathcenter{\scriptstyle \Id_{\BSAh (\HD')}} \, \boxtensor \, f_k} 
      \ar[d]_{\homeq} & & & \BSAh (\HD') \boxtensor \BSDh (\HD_{k+1}) 
      \ar[d]^{\homeq}\\
      \CFLt (\HDMank) \tensor V \ar[rrr]^{\FMank \, \tensor \, \id_V} & & & 
      \CFLt (\HDMannext) \tensor V,
    }
  \]
  where $V$ is a $2$-dimensional vector space, and where the vertical arrows 
  are induced by the pairing theorem \cite[Theorem~7.6.1]{Zar11:BSFH}, commutes 
  up to homotopy.
\end{theorem}

\begin{remark}
  \label{rmk:dir_arrows}
  Due to a difference in the orientation convention, the arrows in 
  \fullref{thm:agree_2} point in the direction opposite to those in 
  \cite{Man07:HFKUnorientedSkein, ManOzs08:QALinks}. As in 
  \cite{Won17:GHUnorientedSkein, PetWon18:TFHSkein}, we follow the convention 
  in \cite{OzsSza04:HFK, Zar11:BSFH}, where the Heegaard surface is the 
  oriented boundary of the $\alpha$-handlebody.
\end{remark}

\subsection{Pairing theorem for triangles}
\label{ssec:pairing_polygons}

In \fullref{ssec:bordered_polygons}, we briefly summarized a theory of 
holomorphic polygon counting for bordered Floer homology developed by Lipshitz, 
Ozsv\'ath, and Thurston \cite[Section~4]{LipOzsThu16:BFHBranchedDoubleSSII}. In 
fact, they also provide a pairing theorem for these bordered polygons in 
\cite[Section~5]{LipOzsThu16:BFHBranchedDoubleSSII}, which, roughly speaking, 
establishes that polygon counting in the bordered setting commutes with the 
(usual) operation of pairing with another bordered diagram, up to homotopy.  
Below, we will translate a special case, a pairing theorem for triangles, to 
the bordered--sutured Floer setting. To avoid confusion, we shall refer to the 
usual pairing theorems (e.g.\ \cite[Theorem~7.6.1 and 
Theorem~8.5.1]{Zar11:BSFH}) as \emph{pairing theorems for bigons}, to 
distinguish them from the pairing theorem for triangles.

Let $(\Sigma, \alphas, \set{\betask}_{k=1}^m, -\arcdiag, \psi)$ be a 
provincially admissible bordered--sutured multi-diagram. As in 
\fullref{ssec:bordered_polygons}, the bordered--sutured diagram $\HD_k = 
(\Sigma, \alphas, \betask, -\arcdiag, \psi)$ is provincially admissible, and 
the sutured diagram $\HD_{k_1, k_2} = (\Sigma, \betas^{k_1}, \betas^{k_2})$ is 
admissible. Here, we are concerned with the type~$D$ modules
\[
  \itensor[^{\AZ}]{\BSDh (\HD_k)}{}.
\]
Let $\Eta^{k_1} \in \SFC (\HD_{k_1, k_2})$ be a cycle; then \fullref{prop:ainf} 
implies that the morphism
\[
  \delta_2 (\Eta^{k_1} \tensor \blank) \colon \BSDh (\HD_{k_1}) \to \BSDh 
  (\HD_{k_2})
\]
is a type~$D$ homomorphism.

Now let $\HD' = (\Sigma', \alphas', \betas', \arcdiag_1 \union \arcdiag_2 
\union \arcdiag, \psi')$ be an admissible bordered--sutured diagram. Let 
$\betaspH$ be
a small Hamiltonian perturbation of $\betas'$; then $\HDpH = (\Sigma', 
\alphas', \betaspH, \arcdiag_1 \union \arcdiag_2 \union \arcdiag, \psi')$ is 
also an admissible bordered--sutured diagram. We will be concerned with the 
type~$\DAA$ trimodules
\[
  \itensor[^{\A (- \arcdiag_1)}]{\BSDAAh (\HD')}{_{\A (\arcdiag_2), \AZ}}, 
  \itensor[^{\A (- \arcdiag_1)}]{\BSDAAh (\HDpH)}{_{\A (\arcdiag_2), \AZ}}.
\]

Noting that there is an obvious relative $\Z$-grading on $\SFC (\betas', 
\betaspH)$, let $\Theta \in \genset (\betas', \betaspH)$ be the top-graded 
generator.  As mentioned at the end of \fullref{ssec:bordered_polygons}, there 
are maps $\itensor[_n]{\delta}{_{\ell_1, \ell_2}}$
for type~$\DAA$ trimodules analogous to $\delta_n$. Observing that the 
differential vanishes on $\SFC (\betas', \betaspH)$, we see that $\Theta$ is 
also a cycle, and so by a type~$\DAA$ version of \fullref{prop:ainf}, the 
morphism given by the collection
\[
  \set{
    \begin{aligned}
      \itensor[_2]{\delta}{_{\ell_1, \ell_2}} (\Theta \tensor \blank \tensor 
      a_1' \tensor \dotsb \tensor a_{\ell_1}' \tensor a_1 \tensor \dotsb 
      \tensor a_{\ell_2}) \colon \SFC (\betas', \betaspH) \tensor \BSDAAh 
      (\alphas', \betas')\\
      \tensor \underbrace{\A (\arcdiag_2) \tensor \dotsb \tensor \A 
        (\arcdiag_2)}_{\ell_1} \tensor \underbrace{\AZ \tensor \dotsb \tensor 
        \AZ}_{\ell_2} \to \A (- \arcdiag_1) \tensor \BSDAAh (\alphas', 
      \betaspH)
    \end{aligned}
  }
\]
is also a type~$\DAA$ homomorphism; we denote this homomorphism by
\[
  \itensor[_2]{\delta}{} (\Theta \tensor \blank) \colon \BSDAAh (\HD') \to 
  \BSDAAh (\HDpH).
\]

Finally, we may glue the Heegaard diagrams $\HD'$ (resp.\ $\HDpH$) and $\HD_k$ 
along $\arcdiag$ to obtain a bordered--sutured diagram $\HD' \union \HD_k$ 
(resp.\ $\HDpH \union \HD_k$). This gives rise to type~$\DA$ bimodules
\[
  \itensor[^{\A (- \arcdiag_1)}]{\BSDAh (\HD' \union \HD_k)}{_{\A 
      (\arcdiag_2)}}, \itensor[^{\A (- \arcdiag_1)}]{\BSDAh (\HDpH \union 
    \HD_k)}{_{\A (\arcdiag_2)}}.
\]
Noting that
\[
  \Theta \tensor \Eta^{k_1} \in \SFC (\betas' \union \betas^{k_1}, \betaspH 
  \union \betas^{k_2})
\]
is a cycle, considerations similar to the above yield a type~$\DA$ homomorphism
\[
  \itensor[_2]{\delta}{} ((\Theta \tensor \Eta^{k_1}) \tensor \blank) \colon 
  \BSDAh (\HD' \union \HD_{k_1}) \to \BSDAh (\HDpH \union \HD_{k_2}).
\]

The discussion in \cite[Section~5]{LipOzsThu16:BFHBranchedDoubleSSII} now 
carries over to our context to give us the following proposition, a pairing 
theorem for triangles.

\begin{proposition}[cf.\ 
  {\cite[Proposition~5.35]{LipOzsThu16:BFHBranchedDoubleSSII}}]
  \label{prop:comm_gen}
  The diagram
  \[
    \xymatrix{
      \BSDAAh (\HD') \boxtensor \BSDh (\HD_{k_1}) 
      \ar[rrrr]^{\itensor[_2]{\delta}{} (\Theta \tensor \blank) \, \boxtensor 
        \, \itensor{\delta}{_2} (\Eta^{k_1} \tensor \blank)} \ar[d]_{\homeq} & 
      & & & \BSDAAh (\HDpH) \boxtensor \BSDh (\HD_{k_2}) \ar[d]^{\homeq}\\
      \BSDAh (\HD' \union \HD_{k_1}) \ar[rrrr]^{\itensor[_2]{\delta}{} ((\Theta 
        \tensor \Eta^{k_1}) \tensor \blank)} & & & & \BSDAh (\HDpH \union 
      \HD_{k_2}),
    }
  \]
  where the vertical arrows are induced by the pairing theorem for bigons 
  \cite[Theorem~8.5.1]{Zar11:BSFH}, commutes up to homotopy.
\end{proposition}

While this proposition may be generalized by replacing $\BSDh$ with trimodules 
$\BSDDAh$, or by slightly relaxing the admissibility of the diagrams, the 
version stated suffices for our purposes.

The discussion in \cite[Section~5]{LipOzsThu16:BFHBranchedDoubleSSII} in fact 
gives a pairing theorem for general polygons, which would be a necessary 
ingredient if we wished to identify the cube-of-resolutions spectral sequence 
\cite[Section~5]{BalLev12:HFKUnorientedSkeinIterate} (see also  
\cite[Section~5]{Won17:GHUnorientedSkein}) arising from the iteration of the 
exact triangle in \cite{Man07:HFKUnorientedSkein} with one arising from 
\fullref{thm:main} in the bordered context. However, we will not be concerned 
with this identification.

For clarity in the sequel, we provide a variant of \fullref{prop:comm_gen} 
here.

\begin{proposition}
  \label{prop:comm_gen_2}
  The diagram
  \[
    \xymatrix{
      \BSDAAh (\HD') \boxtensor \BSDh (\HD_{k_1}) 
      \ar[rrrr]^{\mathcenter{\scriptstyle \Id_{\BSDAAh (\HD')}} \, \boxtensor \, 
        \itensor{\delta}{_2} (\Eta^{k_1} \tensor \blank)} \ar[d]_{\homeq} & & & 
      & \BSDAAh (\HD') \boxtensor \BSDh (\HD_{k_2}) \ar[d]^{\homeq}\\
      \BSDAh (\HD' \union \HD_{k_1}) \ar[rrrr]^{\itensor[_2]{\delta}{} ((\Theta 
        \tensor \Eta^{k_1}) \tensor \blank)} & & & & \BSDAh (\HDpH \union 
      \HD_{k_2}),
    }
  \]
  where the left vertical arrow is induced by the pairing theorem for bigons 
  \cite[Theorem~8.5.1]{Zar11:BSFH}, and the right vertical arrow is induced by 
  the deformation of almost complex structure relating $\HD'$ and $\HDpH$ and 
  the pairing theorem, commutes up to homotopy.
\end{proposition}

\begin{proof}
  This follows directly from \fullref{prop:comm_gen} by precomposing with the 
  commutative diagram that shows the definition of the box tensor morphism 
  $\itensor[_2]{\delta}{} (\Theta \tensor \blank) \boxtensor 
  \itensor{\delta}{_2} (\Eta^{k_1} \tensor \blank)$ (see 
  \cite[Section~2.3.2]{LipOzsThu15:BFHBimodules}),
  \[
    \xymatrix{
      & & & & \BSDAAh (\HD') \boxtensor \BSDh (\HD_{k_2}) 
      \ar[d]^{\itensor[_2]{\delta}{} (\Theta \tensor \blank) \, \boxtensor \, 
        \mathcenter{\scriptstyle \Id_{\BSDh (\HD_{k_2})}}}\\
      \BSDAAh (\HD') \boxtensor \BSDh (\HD_{k_1}) 
      \ar[rrrru]^(0.6){\mathcenter{\scriptstyle \Id_{\BSDAAh (\HD')}} \, 
        \boxtensor \, \itensor{\delta}{_2} (\Eta^{k_1} \tensor \blank) 
        \hspace*{0.5in}} \ar[rrrr]^{\itensor[_2]{\delta}{} (\Theta \tensor 
        \blank) \, \boxtensor \, \itensor{\delta}{_2} (\Eta^{k_1} \tensor 
        \blank)} & & & & \BSDAAh (\HDpH) \boxtensor \BSDh (\HD_{k_2}),
    }
  \]
  and adapting \cite[Proposition~11.4]{Lip06:HFCylindrical} to identify 
  $\itensor[_2]{\delta}{} (\Theta \tensor \blank)$ with the homotopy 
  equivalence induced by the deformation of almost complex structure, as in the 
  last paragraph of the proof of \fullref{lem:cond3_poly}.

  Alternatively, one could directly unpack the definitions in 
  \cite[Theorem~7]{LipOzsThu16:BFHBranchedDoubleSSII}, which was used to prove 
  \cite[Proposition~5.35]{LipOzsThu16:BFHBranchedDoubleSSII} and hence 
  \fullref{prop:comm_gen}.
\end{proof}

We are now ready to prove that the maps $F_k$ in \fullref{thm:main} could be 
obtained by counting holomorphic polygons in a bordered--sutured multi-diagram 
that encodes $\Yk$, instead of doing so for $\Bk$ and invoking the pairing 
theorem.  

\begin{theorem}
  \label{thm:comm_spec}
  Let $\HD'$ be the admissible diagram for $\Y'$ in the proof of 
  \fullref{thm:main} in \fullref{sec:setup}, and define $\betaspH$ and $\HDpH$ 
  as in this section. For $k \in \set{\infty, 0, 1}$, the diagram
  \[
    \xymatrix{
      \BSDAAh (\HD') \boxtensor \BSDh (\HD_{k}) \ar[rrrr]^{%
        \mathcenter{\scriptstyle \Id_{\BSDAAh (\HD')}} \, \boxtensor \, f_{k}} 
      \ar[d]_{\homeq} & & & & \BSDAAh (\HD') \boxtensor \BSDh (\HD_{k+1}) 
      \ar[d]^{\homeq}\\
      \BSDAh (\HD' \union \HD_{k}) \ar[rrrr]^{\itensor[_2]{\delta}{} ((\Theta 
        \tensor \Etaknext) \tensor \blank)} & & & & \BSDAh (\HDpH \union 
      \HD_{k+1}),
    }
  \]
  where the left vertical arrow is induced by the pairing theorem for bigons 
  \cite[Theorem~8.5.1]{Zar11:BSFH}, and the right vertical arrow is induced by 
  the deformation of almost complex structure relating $\HD'$ and $\HDpH$ and 
  the pairing theorem, commutes up to homotopy.
\end{theorem}

\begin{proof}
  This is a direct application of \fullref{prop:comm_gen_2}.
\end{proof}

The significance of \fullref{thm:comm_spec} is the following. By combining the 
diagrams $\HD' \union \HD_k$, $\HDpH \union \HD_{k+1}$, and the analogous 
diagram for $k+2$, one obtains an admissible bordered--sutured multi-diagram 
encoding $\set{\Yk}_{k \in \set{\infty, 0, 1}}$. By counting holomorphic 
polygons in this multi-diagram, one could directly prove \fullref{thm:main}.  
(Of course, in this approach, one would not in general be able to obtain an 
explicit description of the maps involved as in 
\fullref{sec:skein_via_computation}.)
The resulting type~$\DA$ homomorphisms are given exactly by 
$\itensor[_2]{\delta}{} ((\Theta \tensor \Etaknext) \tensor \blank)$.  
\fullref{thm:comm_spec} shows that these homomorphisms in fact agree with the 
ones we used in our proof of \fullref{thm:main}.

\subsection{The exact triangles agree}
\label{ssec:triangles_agree}

As mentioned in \fullref{sec:introduction}, Manolescu 
\cite{Man07:HFKUnorientedSkein} constructs from a connected link projection a 
special Heegaard diagram for $(S^3, L)$. We do not repeat the complete 
definitions here, and instead direct the reader to 
\cite[Section~2]{Man07:HFKUnorientedSkein}; we provide a quick summary in this 
paragraph.  In this construction, the Heegaard surface is the boundary of a 
regular neighborhood of the singularization of the link projection, with an 
$\alpha$-circle for each bounded region of the link diagram, and a 
$\beta$-circle $\beta_v$ for each crossing. (The unbounded region is denoted 
$A_0$.) We denote by $\piece_v$ the piece of the Heegaard diagram associated to 
the crossing $v$, as shown in \cite[Figure~2]{Man07:HFKUnorientedSkein}. It 
also involves a choice of a distinguished edge $e$, as well as a collection of 
edges $\set{s_i}$, in the link projection.  At each of these edges, the local 
diagram is modified as in \cite[Figure~3]{Man07:HFKUnorientedSkein}: A 
meridional $\beta$-curve is added, with a puncture on each side; we further add 
an $\alpha$-curve encircling the two punctures for each edge $s_i$.

In accordance with \fullref{rmk:dir_arrows}, we must slightly modify the 
construction in \cite{Man07:HFKUnorientedSkein} to fit with our convention, as 
follows. Since the difference in the conventions switches the roles of the 
$\alpha$- and $\beta$-handlebodies, the special Heegaard diagram for $(S^3, L)$ 
in \cite{Man07:HFKUnorientedSkein} in fact encodes $(- S^3, L) \isom (S^3, m 
(L))$ in our convention, where $m (L)$ denotes the mirror of a link $L$. To 
compensate for the difference, we replace $\piece_v$ with the piece $\piece_v'$ 
of Heegaard surface constructed for $v$ with the opposite crossing information; 
in other words, we replace the $\beta$-circle $\beta_v$ with the circle 
$\beta_v'$, unique up to homotopy, that is not homotopic to $\beta_v$ in the 
$4$-punctured sphere, and that separates the four punctures as $\beta_v$ does.  
The rest of the construction is unchanged.

\fullref{thm:agree} now follows from \fullref{thm:comm_spec}:

\begin{proof}[Proof of \fullref{thm:agree}, without gradings]
  In \cite{Man07:HFKUnorientedSkein}, the maps $\FMank$ are constructed as 
  follows. First, a special Heegaard diagram $\HDMan_\infty$ for $L_\infty$ is 
  constructed from a connected link projection with a choice of $e$ and 
  $\set{s_i}$, with $\abs{\set{s_i}} = m - 1$.  Then $\HDMan_0$ and $\HDMan_1$ 
  are constructed by changing the $\beta$-circle $\beta_v'$ in $\HDMan_\infty$ 
  near the crossing $v$ in question, and modifying all other $\beta$-circles by 
  a small Hamiltonian perturbation. The maps $\FMank$ are then defined by 
  holomorphic triangle counts between $\HDMank$ and $\HDMannext$, similar to 
  the morphisms $f_k$ in \fullref{ssec:proof_prop}.

  In light of \fullref{thm:comm_spec}, we will find a link projection for 
  $L_\infty$ and a choice of $e$ and $\set{s_i}$ with $\abs{\set{s_i}} = m$, 
  such that the associated special Heegaard diagram $\HDManpinfty$ can be 
  decomposed into two bordered--sutured Heegaard diagrams $\HD'$ and 
  $\HD_\infty$, where $\HD'$ is admissible and $\HD_\infty$ is the diagram in 
  \fullref{sec:setup}. We illustrate how to do so.

  Fix a connected link projection $D$ for $L_\infty$, with a crossing $v$ at 
  which the resolutions of $D$ are projections for $L_0$ and $L_1$. Using 
  Reidemeister moves, find another connected link projection $D'$ for 
  $L_\infty$ such that exactly one of the four regions adjacent to $v$ is the 
  unbounded region $A_0$ in $\R^2$. In fact, we may require that, after 
  rotating the diagram to match $v$ with the left of 
  \fullref{fig:special_diag}, the top region is $A_0$ and the other three 
  regions are bounded.

  \begin{figure}[!htbp]
    \centering
    \labellist
      \footnotesize\hair 2pt
      \pinlabel {\small $A_0$} at 32 133
      \pinlabel {\small $e$} [t] at 4 132
      \pinlabel {\small $s_3$} [t] at 60 131
 	    \pinlabel {\small $s_1$} [b] at 5 69
 	    \pinlabel {\small $s_2$} [b] at 60 68
 	    \pinlabel \textcolor{red}{$1$} [l] at 334 122
 	    \pinlabel \textcolor{red}{$6$} [r] at 401 132
 	    \pinlabel \textcolor{red}{$5$} [b] at 393 107
 	    \pinlabel {\scriptsize $7$} [b] at 416 135
 	    \pinlabel {\scriptsize $8$} [r] at 413 152
 	    \pinlabel \textcolor{red}{$2$} [b] at 338 62
 	    \pinlabel \textcolor{red}{$3$} [b] at 377 62
 	    \pinlabel \textcolor{red}{$4$} [b] at 398 62
 	    \pinlabel {\scriptsize $2$} [l] at 318 51
 	    \pinlabel {\scriptsize $1$} [t] at 319 64
 	    \pinlabel {\scriptsize $3$} [r] at 335 47
 	    \pinlabel {\scriptsize $5$} [b] at 413 45
 	    \pinlabel {\scriptsize $4$} [l] at 400 47
 	    \pinlabel {\scriptsize $6$} [l] at 412 59
    \endlabellist
    \includegraphics{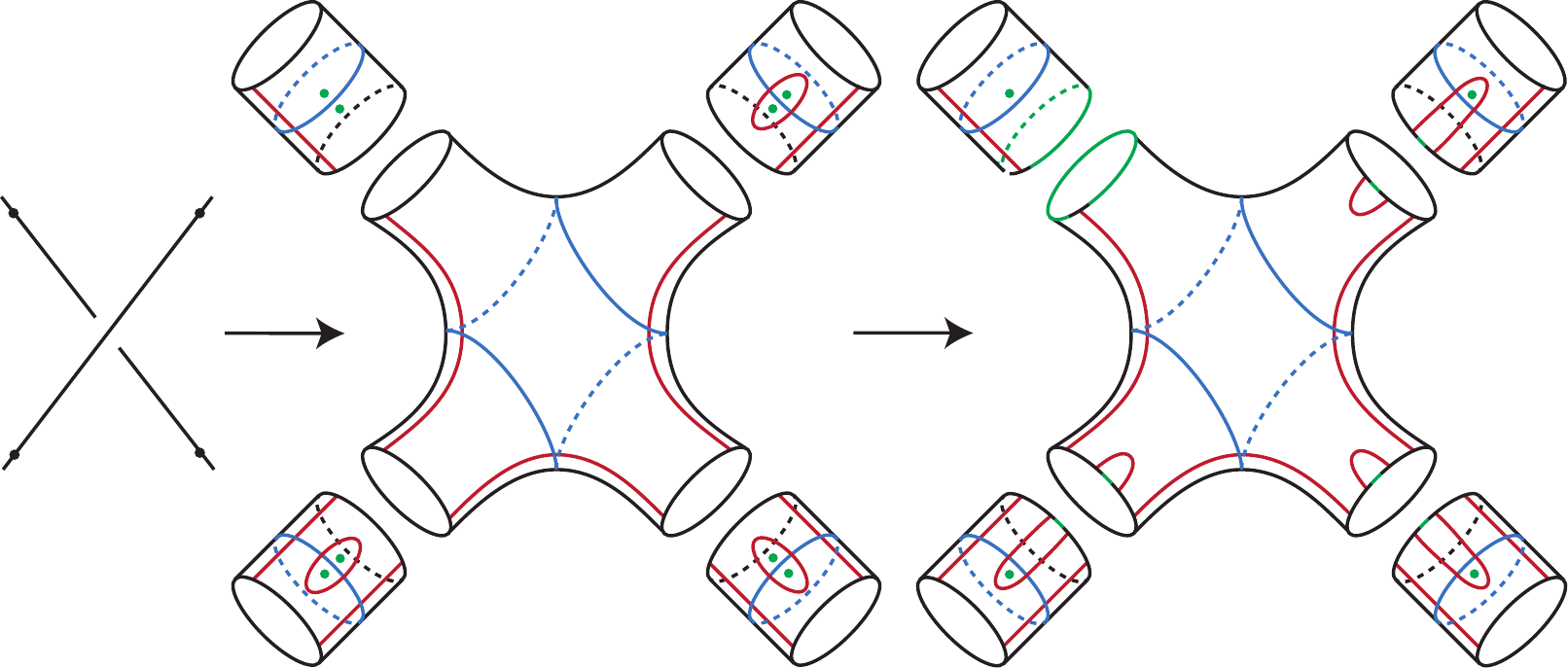}
    \caption{Left: The crossing $v$ in a connected projection of $L_\infty$, 
      with all adjacent edges marked.  Center: The piece $\piece_v'$ associated 
      with $v$ in the special Heegaard diagram $\HDManpinfty$, together with 
      the local modifications according to the choice of $e$ and $\set{s_i}$.  
      Right: After a small isotopy of $\HDManpinfty$, we may identify 
      $\piece_v'$ with the bordered--sutured diagram $\HD_\infty$ in 
      \fullref{sec:setup}.}
    \label{fig:special_diag}
  \end{figure}
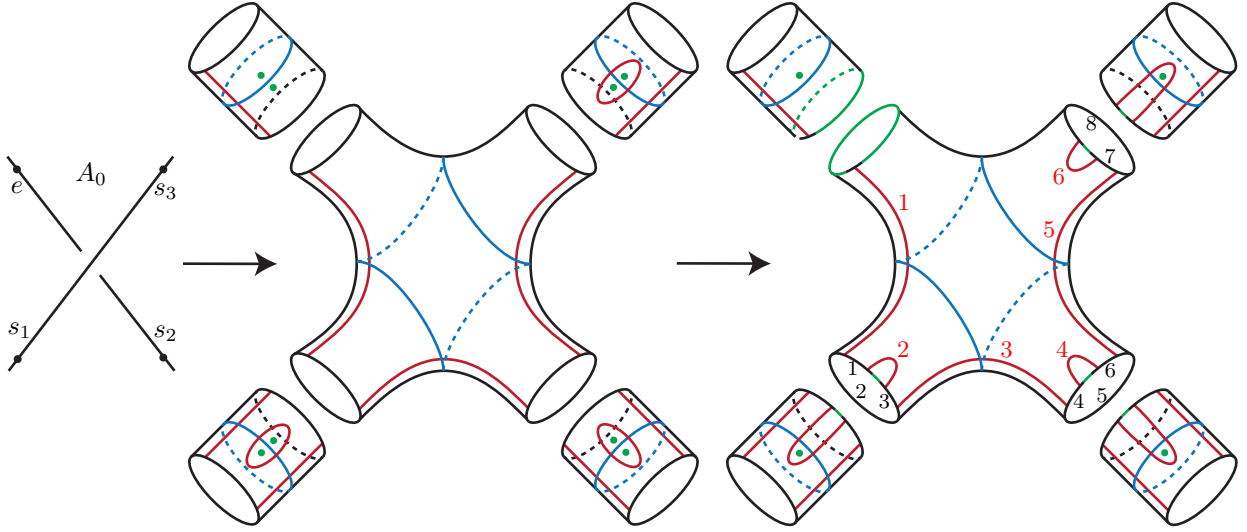

  Viewing $D'$ as a $4$-valent graph, we choose the distinguished edge $e$ to 
  be that to the top left of $v$, and choose $s_1$, $s_2$, and $s_3$ to be the 
  other three edges adjacent to $v$; see the left of 
  \fullref{fig:special_diag}. Complete this to a minimal collection 
  $\set{s_i}$; since the two strands at $v$ belong to at most two components, 
  the cardinality of $\set{s_i}$ will be one larger than the minimum, meaning 
  $\abs{\set{s_i}} = m$. Thus, $\HDManpinfty$, and hence $\HDManpzero$ and 
  $\HDManpone$, will be once quasi-stabilized when compared to the Heegaard 
  diagrams $\HDMank$ used in \cite{Man07:HFKUnorientedSkein}.

  Now using $D'$, construct a special Heegaard diagram $\HDManpinfty$ according 
  to \cite[Section~2]{Man07:HFKUnorientedSkein}. Under this construction, the 
  piece $\piece_v'$ of the Heegaard diagram associated with $v$ is as in the 
  center of \fullref{fig:special_diag}. In the same figure, the local diagrams 
  associated to the edges $e$, $s_1$, $s_2$, and $s_3$ are also shown.  
  $\HDManpinfty$ is formed by attaching these pieces, as well as other pieces 
  not shown, to each other.

  After a small isotopy of $\HDManpinfty$ (or equivalently, by choosing a 
  different place to cut it apart), we obtain the local diagram as in the right 
  of \fullref{fig:special_diag}. In this diagram, the $4$-punctured sphere in 
  the center, which we also denote by $\piece_v'$ by a slight abuse of 
  notation, intersects with three more $\alpha$-curves than before. Moreover, 
  four of the punctures in the previous picture are now each manifested as an 
  arc on $\bdy \piece_v'$ together with a corresponding arc on the boundary of 
  another piece.  Now it is easy to check that $\piece_v'$ is a provincially 
  admissible bordered--sutured diagram, with the arcs mentioned above serving as 
  $\bdy \Sigma \setminus \arcdiag$. In fact, $\piece_v'$ now coincides with 
  $\HD_\infty$ (see \fullref{fig:hd_infty}); in the right of 
  \fullref{fig:special_diag}, we have included the numbering of the 
  $\alpha$-arcs and some of the Reeb chords to help the reader verify this 
  identification. We may now declare the rest of $\HDManpinfty$ to be $\HD'$; 
  then the admissibility of $\HD'$ follows from that of $\HDManpinfty$. Note 
  that, gluing $\HD'$ with $\piece_v' \cong \HD_\infty$ to obtain 
  $\HDManpinfty$, each arc in $\bdy \Sigma \setminus \arcdiag$ combines with 
  the corresponding arc to form the boundary of the corresponding puncture in 
  $\HDManpinfty$.

  Writing $\HDManp_k$ as $\HD' \union \HD_k$, we see that $\HDManp_{k+1}$ is 
  $\HDpH \union \HD_{k+1}$. Thus, we may directly apply \fullref{thm:comm_spec} 
  to conclude the proof.
\end{proof}

\section{Gradings}
\label{sec:gradings}

We have avoided any discussion of gradings so far, primarily because they are 
somewhat more complicated than usual in (bordered Floer theory and) 
bordered--sutured Floer theory.

In the following, we first give a brief overview of these gradings in 
\fullref{ssec:gradings_bsfh}, before proceeding to prove graded versions of 
\fullref{thm:main} and \fullref{cor:arbitrary}.  The strategy we employ 
reflects that of our proof of the ungraded versions in \fullref{sec:setup}; in 
particular, we will first endow \fullref{prop:elementary} with gradings, and 
then use the pairing theorem to extend these gradings to the general case.  
Thus, we will compute the gradings on the algebra $\AZ$ in 
\fullref{ssec:gradings_alg} and the gradings on the modules $\BSDh (\Bk)$ in 
\fullref{ssec:gradings_mod}, for $\arcdiag$ and $\Bk$ as defined in 
\fullref{sec:setup}. Finally, we state the graded statements and complete the 
proofs in \fullref{ssec:gradings_proof}.

\subsection{Gradings in bordered--sutured Floer theory}
\label{ssec:gradings_bsfh}

Since \fullref{prop:elementary} is stated for type~$D$ modules, we focus on the 
gradings on type~$D$ structures here. Again, we omit many details, and direct 
the interested reader to \cite{Zar11:BSFH, LipOzsThu18:BFH}. Note that we 
mostly follow the notation in \cite{Zar11:BSFH}, which differs from that in 
\cite{LipOzsThu18:BFH}. (See, however, \fullref{rmk:half_grading}.)

Let $\Y = (Y, \Gamma, -\arcdiag, \phi)$ be a bordered--sutured manifold, and 
$\HD = (\Sigma, \alphas, \betas, -\arcdiag, \psi)$ be a bordered--sutured 
diagram for $\Y$. First of all, the type~$D$ module $\BSDh (\Y) \homeq \BSDh 
(\HD)$ over $\AZ$ decomposes into direct-sum components corresponding to 
relative $\SpinC$-structures:
\[
  \BSDh (\HD) = \bigdirsum_{\spinc \in \SpinC (Y, \bdy Y \setminus F (- 
    \arcdiag))} \BSDh (\HD, \spinc)
\]
Here, the space $\SpinC (Y, \bdy Y \setminus F (-\arcdiag))$ of relative 
$\SpinC$-structures is an affine copy of $H^2 (Y, \bdy Y \setminus F (- 
\arcdiag)) \isom H_1 (Y, F (-\arcdiag))$. This decomposition arises from an 
assignment $\x \mapsto \spinc (\x)$ of an Euler structure to each generator (by 
the standard construction of a vector field), and the observation that the 
differential $\delta_{\BSDh (\HD)}$ respects this assignment. The set of all 
generators $\x$ with $\spinc (\x) = \spinc$ is denoted $\genset (\HD, \spinc)$.  
This direct sum decomposition is invariant for different Heegaard diagrams and 
choices of complex structures.

The structure of gradings in bordered--sutured Floer theory is as follows. The 
differential algebra $\AZ$ is graded by a possibly non-Abelian group $G$ with a 
fixed central element $\lambda \in G$, via a map $g \colon \AZ \to G$, such 
that for homogeneous elements $a$ and $b$, we have
\begin{itemize}
  \item $g (a b) = g (a) g (b)$; and
  \item $g (d (a)) = \lambda^{-1} g (a)$.
\end{itemize}
Then, each $\BSDh (\HD, \spinc)$ has a (separate) \emph{relative grading in a 
  $G$-set}, by which we mean a left action by $G$ on some set $S$, and a map $g 
\colon \genset (\HD, \spinc) \to S$, such that, for a homogeneous element $a 
\in \AZ$ and a generator $\x \in \genset (\HD, \spinc)$, we have
\begin{itemize}
  \item $g (a \tensor \x) = g (a) \cdot g (\x)$; and
  \item $g (\delta (\x)) = \lambda^{-1} \cdot g (\x)$.
\end{itemize}
Two relative gradings $g$ and $g'$, respectively in $G$-sets $S$ and $S'$, are 
equivalent if there exists a bijective $G$-set map $\phi \colon S \to S'$ that 
commutes with both $g$ and $g'$.

Here, the story gets one more twist. There are in fact two related but distinct 
gradings on $\AZ$, and two corresponding ways to grade each $\BSDh (\HD, 
\spinc)$. More concretely, choose a generator $\w_0 \in \genset (\HD, \spinc)$.  
One defines the (possibly non-Abelian) group $\Gr (\arcdiag)$ and subgroup 
$\Grr (\arcdiag)$, and further defines (not necessarily normal) subgroups $\grP 
(\w_0) \subset \Gr (\arcdiag)$ and $\grPr_r (\w_0) \subset \Grr (\arcdiag)$.  
Denoting the left cosets by
\begin{align*}
  \Gr (\HD, \spinc)& = \Gr (\arcdiag) / \grP (\w_0),\\
  \Grr_r (\HD, \spinc)& = \Grr (\arcdiag) / \grPr_r (\w_0),
\end{align*}
there is a natural left action by $\Gr (\arcdiag)$ on $\Gr (\HD, \spinc)$, and 
by $\Grr_r (\arcdiag)$ on $\Grr_r (\HD, \spinc)$. One then defines
\begin{multicols}{2}
  \begin{itemize}
    \item A map $\gr \colon \AZ \to \Gr (\arcdiag)$; and
    \item A map $\gr \colon \genset (\HD, \spinc) \to \Gr (\HD, \spinc)$,
  \end{itemize}
  \columnbreak
  \begin{itemize}
    \item A map $\grr_r \colon \AZ \to \Grr (\arcdiag)$; and
    \item A map $\grr_r \colon \genset (\HD, \spinc) \to \Grr_r (\HD, \spinc)$,
  \end{itemize}
\end{multicols}
\noindent such that $\gr$ (resp.\ $\grr_r$) provides a grading on $\AZ$ and a 
relative grading on $\BSDh (\HD, \spinc)$. Unfortunately, while $\gr$ is 
independent of all choices made in its construction, the grading group $\Gr 
(\arcdiag)$ is too large to support a graded version of the pairing theorem; on 
the other hand, while $\grr_r$ allows for a graded pairing theorem, it depends 
on some \emph{refinement data} $r$ on $\AZ$ (known as a \emph{grading 
  reduction} in \cite{Zar11:BSFH}), written as a subscript above.  Since we 
will use the pairing theorem, we will focus on the $(\grr_r)$-graded versions 
when we compute the gradings below.

To end this subsection, we mention here the structure of the graded pairing 
theorem.  Let $\Y' = (Y', \Gamma', \arcdiag_1 \union \arcdiag_2 \union 
\arcdiag, \phi')$ be another bordered--sutured manifold, and let $\Y''$ be the 
bordered--sutured manifold obtained by gluing $\Y'$ and $\Y$ along $\sutF 
(\arcdiag)$, with underlying manifold $Y''$. Now let $\HD'$ be a bordered 
sutured diagram for $\Y'$, and let $\spinc'$ be a relative $\SpinC$-structure 
on $(Y', \bdy Y' \setminus F (\arcdiag_1 \union \arcdiag_2 \union \arcdiag))$.  
Choosing a generator $\w_0' \in \genset (\HD', \spinc')$, then, $\BSDAAh (\HD', 
\spinc')$ has a relative grading in
\[
  \Grr_{r_1, r_2, r} (\HD', \spinc') = \grPr_{r_1, r_2, r} (\w_0') \rcoset \Grr 
  (\arcdiag_1 \union \arcdiag_2 \union \arcdiag),
\]
which has a left action by $\Grr (-\arcdiag_1)$ (equivalent to a right action 
by $\Grr (\arcdiag_1)$) and a right action by $\Grr (\arcdiag_2)$, and most 
pertinently, a right action by $\Grr (\arcdiag)$.
 
Suppose that $\spinc'$ is compatible with $\spinc \in \SpinC (Y, \bdy Y 
\setminus F (- \arcdiag))$, i.e.\ that there exists some $\spinc'' \in \SpinC 
(Y'', \bdy Y'' \setminus F (\arcdiag_1 \union \arcdiag_2))$ whose restrictions 
to $Y'$ and $Y$ are $\spinc'$ and $\spinc$ respectively. The set of all such 
$\SpinC$-structures $\spinc''$, which we denote by $\SpinC (Y'', \bdy Y'' 
\setminus F (\arcdiag_1 \union \arcdiag_2), \spinc', \spinc)$, is an affine 
copy of
\[
  \Im (\iota_* \colon H_1 (F (\arcdiag)) \to H_1 (Y'', F (\arcdiag_1 \union 
  \arcdiag_2))).
\]
(Here, we have implicitly used Poincar\'e duality to regard the set of 
$\SpinC$-structures as an affine space of first homology, rather than the usual 
second cohomology.) Now Let $\w_0' \in \genset (\HD', \spinc')$ and $\w_0 \in 
\genset (\HD, \spinc)$.  Then, viewing $\Grr (\arcdiag)$ as a subgroup of $\Grr 
(\arcdiag_1 \union \arcdiag_2 \union \arcdiag)$, the box tensor product 
$\BSDAAh (\HD', \spinc') \boxtensor \BSDh (\HD, \spinc)$ has a relative grading 
$\grbox_{r_1, r_2, r}$ with values in
\[
  \Grr_{r_1, r_2, r} (\HD', \HD, \spinc', \spinc) = \grPr_{r_1, r_2, r} (\w_0') 
  \rcoset \Grr (\arcdiag_1 \union \arcdiag_2 \union \arcdiag) / \grPr_r (\w_0).
\]
The key feature of this grading, as shown in \cite{LipOzsThu18:BFH, 
  Zar11:BSFH}, is the following. Suppose the idempotents of $\w_0'$ and $\w_0$ 
match along $\arcdiag$; then there is a projection $\pi \colon \Grr_{r_1, r_2, 
  r} (\HD', \HD, \spinc', \spinc) \to \SpinC (Y'', \bdy Y'' \setminus F 
(\arcdiag_1 \union \arcdiag_2))$, such that
\begin{itemize}
  \item $\pi$ distinguishes different $\SpinC$-structures $\spinc''$ in $(Y'', 
    \bdy Y'' \setminus F (\arcdiag_1 \union \arcdiag_2))$ that restrict to 
    $\spinc'$ in $Y'$ and $\spinc$ in $Y$;
  \item Each fiber $\pi^{-1} (\spinc'')$ can be identified with $\Grr_{r_1, 
      r_2} (\HD' \union \HD, \spinc'')$; and
  \item The homotopy equivalence
    \[
      \Phi \colon \BSDAAh (\HD', \spinc') \boxtensor \BSDh (\HD, \spinc) \to
      \bigdirsum_{\substack{\spinc'' \in \SpinC (Y'', \bdy Y'' \setminus 
          F(\arcdiag_1 \union \arcdiag_2))\\ \spinc''\vert_{Y'} = \spinc', \;\; 
          \spinc''\vert_{Y} = \spinc}} \BSDAh (\HD' \union \HD, \spinc'')
    \]
    respects the identification above, in the sense that, if $\Phi (\x) \in 
    \BSDAh (\HD' \union \HD, \spinc'')$, then $\grbox_{r_1, r_2, r} (\x)$ is 
    valued in $\pi^{-1} (\spinc'')$, and the grading in this fiber is 
    equivalent to the grading $\grr_{r_1, r_2} (\Phi (\x))$ valued in 
    $\Grr_{r_1, r_2} (\HD' \union \HD, \spinc'')$.
\end{itemize}

\subsection{Grading on the algebra}
\label{ssec:gradings_alg}

We now turn to computing the gradings on the algebra $\AZ = (\orarcs, 
\matchedpts, M)$. In fact, recall that in \fullref{ssec:4puncS2}, we defined 
the index sets $J'$ and $J$.  Since we work with a type~$D$ module, it is 
sufficient for our purposes to compute the gradings of $(j) = a_5 
(\set{\rho_j}) \in \AZfive$ for $j \in J'$. (We may then compute the gradings 
of $(j_1, \dotsc, j_\ell)$ by group multiplication.)

\subsubsection{The unrefined grading on the algebra}
\label{sssec:grZ}

The unrefined $\Gr (\arcdiag)$-grading on $\AZ$ is defined as follows.  For 
$a_i \in \matchedpts$ and $[\rhos] \in H_1 (\orarcs, \matchedpts)$, let $m 
(a_i, [\rhos])$ be the average multiplicity with which $[\rhos]$ covers the 
regions on either side of $a_i$, and extend this linearly to a map $m \colon 
H_0 (\matchedpts) \cross H_1 (\orarcs, \matchedpts) \to (1/2) \Z$.  Define now 
$L \colon H_1 (\orarcs, \matchedpts) \cross H_1 (\orarcs, \matchedpts) \to 
(1/2) \Z$ by
\[
  L ([\rhos], [\rhos']) = m (\bdy [\rhos], [\rhos']),
\]
where $\bdy$ is the (signed) connecting homomorphism in the long exact sequence 
for a pair. For $[\rhos] \in H_1 (\orarcs, \matchedpts)$, we also define 
$\epsilon ([\rhos]) \in ((1/2) \Z) / \Z$ by
\[
  \epsilon ([\rhos]) = \paren{\frac{1}{4} \# \setc{a_i \in \matchedpts}{m (a_i, 
      [\rhos]) \equiv \frac{1}{2} \mod 1}} \mod 1.
\]
We now let
\[
  \Gr (\arcdiag) = \setc{(n, [\rhos]) \in \frac{1}{2} \Z \cross H_1 (\orarcs, 
    \matchedpts)}{n \equiv \epsilon ([\rhos]) \mod 1}.
\]
For $(n, [\rhos]) \in \Gr (\arcdiag)$, $n$ is called the \emph{Maslov 
  component}, and $[\rhos]$ the \emph{homological component}.

\begin{remark}
  \label{rmk:half_grading}
  The definitions of $\Gr (\arcdiag)$ here, and consequently of $\Grr 
  (\arcdiag)$ below, conform to the more recent definitions in 
  \cite{LipOzsThu18:BFH}. In \cite{Zar11:BSFH}, $\Gr (\arcdiag)$ is simply 
  defined to be $(1/2) \Z \cross H_1 (\orarcs, \matchedpts)$, which is twice as 
  large as presented above and in \cite{LipOzsThu18:BFH}. By 
  \cite[Proposition~3.39]{LipOzsThu18:BFH}, $\gr (a)$ defined below is 
  guaranteed to be in the smaller $\Gr (\arcdiag)$ as defined here.
\end{remark}

To define the grading of an element $a = \Istart \cdot a_p (\rhos) \cdot \Iend 
\in \A (\arcdiag, p)$, begin by choosing a value of $m_i \in M^{-1} (i) 
\subset \matchedpts$ for each unoccupied arc $e_i$ that is in both $\Istart$ 
and $\Iend$.  (Informally, this is equivalent to making one of each matched 
pair of horizontal dotted lines solid in the strands diagram associated to 
$a$.) Let $\phi$ be the corresponding strands diagram, and let $\inv (\phi)$ be 
the minimum number of crossings in $\phi$.  Further, let
\[
  S = \setc{m_i}{e_i \text{ is in both } \Istart \text{ and } \Iend} \union 
  \rhos^- \subset \matchedpts
\]
be the set of all starting points of $\phi$ (consisting of exactly $p$ 
elements), and let $[S]$ denote the corresponding element of $H_0 
(\matchedpts)$. The grading of $a$ is then defined to be
\[
  \gr (a) = \paren{\inv (\phi) - m ([S], [\rhos]), [\rhos]} \in \Gr (\arcdiag).
\]
According to \cite[Proposition~2.4.5]{Zar11:BSFH}, $\gr (a)$ is independent of 
the choice of values of $m_i$, and $\gr$ gives a well-defined grading on $\AZ$, 
with central element $\lambda = (1, 0)$. Furthermore, $\gr (a)$ turns out to be 
independent of $p$, $\Istart$, and $\Iend$.

Focusing on our context, consider now some $(j) = a_5 (\set{\rho_j}) \in 
\AZfive$ with $j \in J'$. By the last sentence of the previous paragraph, we 
may easily compute
\[
  \gr ((j)) = \gr (a_5 (\set{\rho_j})) = \gr (a_1 (\set{\rho_j})),
\]
where the strands diagram $\phi$ for $a_1 (\set{\rho_j})$ has no horizontal 
dotted lines at all. This means that $\inv (\phi) = 0$ and $S = \set{\rho^-}$, 
allowing us to compute
\[
  \gr ((j)) = \gr (a_1 (\set{\rho_j})) = \paren{-\frac{1}{2}, [\rho_j]}.
\]
For economy and clarity, from now on we will write $\gr (j)$ instead of $\gr 
((j))$.

\subsubsection{The refined grading on the algebra}
\label{sssec:grrZ}

For $\AZ = (\orarcs, \matchedpts, M)$, the matching $M$ induces a map $M_* 
\colon H_0 (\matchedpts) \to H_0 (\set{1, \dotsc, 6})$. Again denoting by 
$\bdy$ the connecting homomorphism, let $\bdy' = M_* \comp \bdy$. The refined 
grading group is defined to be the subgroup
\[
  \Grr (\arcdiag) = \setc{(n, [\rhos]) \in \Gr (\arcdiag)}{\bdy' [\rhos] = 0} 
  \subset \Gr (\arcdiag),
\]
consisting of elements with homological component in $\ker \bdy' \isom H_1 (F 
(\arcdiag))$. In fact, this allows us to view $\Grr (\arcdiag)$ as a central 
extension of $H_1 (F (\arcdiag))$.

In our context, $\ker \bdy'$ is generated by $[\rho_{123}]$, $[\rho_{456}]$, 
and $[\rho_{78}]$. Moreover, it is easy to see that $L$ is identically zero on 
$\ker \bdy' \cross \ker \bdy'$, and so $\Grr (\arcdiag)$ is Abelian and in fact 
isomorphic to $\Z \dirsum \Z^3$, with generators
\[
  \lambda = \paren{1, 0}, \quad A_1 = \paren{-\frac{1}{2}, [\rho_{123}]}, \quad 
  A_2 = \paren{-\frac{1}{2}, [\rho_{456}]}, \quad A_3 = \paren{-\frac{3}{2}, 
    [\rho_{78}]}.
\]
The generators here are chosen to facilitate our discussion below. (Also, 
observe that, for example, $(0, [\rho_{123}]) \notin \Gr (\arcdiag)$; see 
\fullref{rmk:half_grading}.)

In \cite{Zar11:BSFH}, in order to define the refinement data necessary to 
define $\grr_r$, the notion of \emph{connected components} of idempotents is 
introduced:  Two idempotents $I$ and $I'$ are connected if $I - I'$ is in the 
image of $\bdy'$. Now let $\pihom \colon \Gr (\arcdiag) \to H_1 (\orarcs, 
\matchedpts)$ be the projection homomorphism; the refinement data then consists 
of a \emph{base idempotent} $I_0$ in each connected component, and an element 
$r (I) \in (\bdy' \comp \pihom)^{-1} (I - I_0) \subset \Gr (\arcdiag)$ for each 
$I$ in the same component as $I_0$. Then $\grr_r \colon \AZ \to \Grr 
(\arcdiag)$ is defined by
\[
  \grr_r (a) = r (\Istart) \gr (a) r (\Iend)^{-1},
\]
for a generator $a = \Istart \cdot a \cdot \Iend \in \AZ$.

For our $\arcdiag$, it is obvious that there is exactly one connected component 
of idempotents. Thus, we may arbitrarily choose any idempotent in $\IZfive$ to 
be $I_0$. We choose $I_0 = \Iocc{6}$ for convenience of computation, and choose
\begin{align*}
  r (\Iocc{1})& = \paren{- \frac{1}{2}, [\rho_2] + [\rho_5] + [\rho_8]},\\
  r (\Iocc{2})& = \paren{- \frac{1}{2}, [\rho_1] + [\rho_2] + [\rho_5] + 
    [\rho_8]},\\
  r (\Iocc{3})& = (0, [\rho_5] + [\rho_8]),\\
  r (\Iocc{4})& = \paren{0, [\rho_4] + [\rho_5] + [\rho_8]},\\
  r (\Iocc{5})& = \paren{- \frac{1}{2}, [\rho_8]},\\
  r (\Iocc{6})& = (0, 0).
\end{align*}
(Again, see \fullref{rmk:half_grading}.) Then, as an example, we may compute 
$\grr_r (7, 8)$ as follows.  Since $(7, 8) = \Iocc{5} \cdot (78) \cdot 
\Iocc{5}$, we have
\[
  \grr_r (7, 8) = r (\Iocc{5}) \gr (78) r (\Iocc{5})^{-1} = 
  \paren{-\frac{1}{2}, [\rho_8]} \paren{-\frac{1}{2}, [\rho_{78}]} 
  \paren{\frac{1}{2}, -[\rho_8]} = \paren{-\frac{3}{2}, [\rho_{78}]}.
\]
Similarly, we have $\grr_r (1, 23) = \grr_r (12, 3) = (-1/2, [\rho_{123}])$ and 
$\grr_r (4, 56) = \grr_r (45, 6) = (-1/2, [\rho_{456}])$.

\subsubsection{A skein reduction}
\label{sssec:grrZsk}

We now make a further reduction in the grading set, which is necessary for 
stating a graded version of \fullref{thm:main}: We let
\begin{gather*}
  \grrPsk = \subgrp{A_1, A_2, A_3},\\
  \Grrsk = \Grr (\arcdiag) / \grrPsk \isom \Z.
\end{gather*}
We will denote an element in $\Grrsk$ by an integer, given by sending the class 
$\lambda \cdot \grrPsk$ to $1$. (The reason we chose the symbol $\grrPsk$ will 
become clear in the next subsection.)

Composing $\grr_r$ with the canonical projection, we thus obtain a map 
$\grrsk_r \colon \AZ \to \Grrsk$, which we call the \emph{skein-reduced 
  grading}.  We may now compute all colored algebra elements appearing in 
\fullref{fig:comm_diag_f}:
\begin{align*}
  \grrr (I)& = (0, 0), & \grrrsk (I)& = 0,\\
  \grrr (4, 6)& = (-1/2, [\rho_{456}]), & \grrrsk (4, 6)& = 0,\\
  \grrr (4, 6, 7, 8)& = (-2, [\rho_{456}] + [\rho_{78}]),& \grrrsk (4, 6, 7, 8) 
  & = 0,\\
  \grrr (4, 56)& = (-1/2, [\rho_{456}]), & \grrrsk (4, 56)& = 0,\\
  \grrr (12, 3)& = (-1/2, [\rho_{123}]), & \grrrsk (12, 3)& = 0,\\
  \grrr (2)& = (0, 0), & \grrrsk (2)& = 0,\\
  \grrr (45, 6)& = (-1/2, [\rho_{456}]), & \grrrsk (45, 6)& = 0,\\
  \grrr (7, 8)& = (-3/2, [\rho_{78}]), & \grrrsk (7, 8)& = 0,\\
  \grrr (12, 3, 4, 6)& = (-1, [\rho_{123}] + [\rho_{456}]), & \grrrsk (12, 3, 
  4, 6)& = 0,\\
  \grrr (4, 56, 2)& = (-1/2, [\rho_{456}]), & \grrrsk (4, 56, 2)& = 0,\\
  \grrr (12, 3, 4, 56, 2)& = (-1, [\rho_{123}] + [\rho_{456}]), & \grrrsk (12, 
  3, 4, 56, 2)& = 0,\\
  \grrr (1, 23)& = (-1/2, [\rho_{123}]), & \grrrsk (1, 23)& = 0,\\
  \grrr (45, 6, 7, 8, 5)& = (-3, [\rho_{456}] + [\rho_{78}]), & \grrrsk (45, 6, 
  7, 8, 5)& = -1,\\
  \grrr (7, 8, 5)& = (-5/2, [\rho_{78}]), & \grrrsk (7, 8, 5)& = -1,\\
  \grrr (5)& = (-1, 0), & \grrrsk (5)& = -1,\\
  \grrr (5, 12, 3)& = (-3/2, [\rho_{123}]), & \grrrsk (5, 12, 3)& = -1,\\
  \grrr (5, 12, 3, 4, 56)& = (-2, [\rho_{123}] + [\rho_{456}]), & \grrrsk (5, 
  12, 3, 4, 56)& = -1,\\
  \grrr (1, 3)& = (-3/2, [\rho_{123}]), & \grrrsk (1, 3)& = -1,\\
  \grrr (1, 3, 4, 56)& = (-2, [\rho_{123}] + [\rho_{456}]), & \grrrsk (1, 3, 4, 
  56)& = -1.
\end{align*}

\subsection{Gradings on the modules}
\label{ssec:gradings_mod}

We now compute the gradings on $\BSDh (\Bk) \homeq \BSDh (\HD_k)$.

\subsubsection{Unrefined and refined gradings on the modules}
\label{sssec:GrH}

The gradings on the modules are defined in \cite{Zar11:BSFH} as follows.  Let 
$\w_1, \w_2 \in \genset (\HD, \spinc)$. First, for each homology class $B \in 
\pi_2 (\w_2, \w_1)$ with domain $[B]$, its unrefined and refined gradings are 
defined by
\begin{gather*}
  \gr (B) = (- e ([B]) - n_{\w_1} ([B]) - n_{\w_2} ([B]), \bdybdy [B]) \in \Gr 
  (\arcdiag),\\
  \grr_r (B) = r (I_D (\w_2)) \gr (B) r (I_D (\w_1))^{-1} \in \Grr (\arcdiag),
\end{gather*}
where $e ([B])$ denotes the Euler measure of $[B]$, $n_{\w} ([B])$ the sum of 
the average local multiplicities of $[B]$ in the four regions adjacent to the 
intersection points of $\w$, and $I_D (\w)$ the type~$D$ idempotent of $\w$.  
(For the reader that is comparing the above with 
\cite[Definition~6.1.1]{Zar11:BSFH}, the order in which $r (I_D (\w_1))$ and $r 
(I_D (\w_2))$ appear in the second equation arises from the fact that we are 
working with $\Grr (\arcdiag)$ instead of $\Grr (- \arcdiag)$; see 
\cite[Section~10.4 and Section~10.5]{LipOzsThu18:BFH}.)

Then, for each $\spinc \in \SpinC (Y, \bdy Y \setminus F (- \arcdiag))$, one 
picks a generator $\w_0 \in \genset (\HD, \spinc)$, called a \emph{base 
  generator}, and defines the \emph{stabilizer subgroups}
\begin{gather*}
  \grP (\w_0) = \setc{\gr (B)}{B \in \pi_2 (\w_0, \w_0)} \subset \Gr 
  (\arcdiag),\\
  \grPr_r (\w_0) = \setc{\grr_r (B)}{B \in \pi_2 (\w_0, \w_0)} \subset \Grr 
  (\arcdiag).
\end{gather*}
As mentioned in \fullref{ssec:gradings_bsfh}, the grading sets $\Gr (\HD, 
\spinc)$ and $\Grr_r (\HD, \spinc)$ are respectively defined as the set of left 
cosets $\Gr (\arcdiag) / \grP (\w_0)$ and $\Grr (\arcdiag) / \grPr_r (\w_0)$.  
For a generator $\w \in \genset (\HD, \spinc)$, then, one picks a homology 
class $B_0 \in \pi_2 (\w, \w_0)$, and the gradings of $\w$ are defined by
\begin{gather*}
  \gr (\w) = \gr (B_0) \cdot \grP (\w_0) \in \Gr (\HD, \spinc),\\
  \grr_r (\w) = \grr_r (B_0) \cdot \grPr_r (\w_0) \in \Grr_r (\HD, \spinc).
\end{gather*}
The map $\grr_r \colon \genset (\HD, \spinc) \to \Grr_r (\HD, \spinc)$ gives a 
relative grading by a $\Grr (\arcdiag)$-set, which is independent of the choice 
of $\w_0$ and $B_0$ (but not $r$ in general).

In our context, first note that, for each $k \in \set{\infty, 0, 1}$, we have 
$H_1 (B_k, F (- \arcdiag)) = 0$, and so there is a unique $\SpinC$-structure 
$\spinc_k \in \SpinC (B_k, \bdy B_k \setminus F (- \arcdiag))$. This means that 
$\BSDh (\HD_k) = \BSDh (\HD_k, \spinc_k)$; in other words, $\y_1, \y_2, \y_3 
\in \genset (\HD_\infty, \spinc_{\infty})$ share the same grading set, and so 
do $\z_1, \z_2 \in \genset (\HD_0, \spinc_0)$. In the following, we focus on 
the refined gradings.

Arbitrarily, we pick $\x \in \genset (\HD_1)$, $\y_3 \in \genset (\HD_\infty)$, 
and $\z_1 \in \genset (\HD_0)$ to define our grading sets. Since there is only 
one region $R$ in $\HD_1$ not adjacent to $\bdy \Sigma \setminus \orarcs$, we 
see that $\pi_2 (\x, \x)$ is generated by the homology class $B$ with domain 
$[B] = R$. Seeing $I_D (\x) = \Iocc{3}$, we may thus compute $e ([B]) = -2$ and 
$n_{\x} ([B]) = 1/2$ to obtain
\[
  \grr_r (B) = (0, [\rho_5] + [\rho_8]) (1, -[\rho_{456}] - [\rho_{78}]) (0, 
  -[\rho_5] -[\rho_8]) = (2, -[\rho_{456}] - [\rho_{78}]) = - A_2 - A_3.
\]
This means that $\grPr_r (\x) = \subgrp{- A_2 - A_3} \subset \Grr (\arcdiag)$, 
and so $\Grr_r (\HD_1, \spinc_1) = \subgrp{\lambda, A_1, A_2, A_3} / \subgrp{- 
  A_2 - A_3}$.  Similarly, we obtain that $\grPr_r (\y_3) = \subgrp{- A_1 - 
  A_3}$, and $\grPr_r (\z_1) = \subgrp{- A_1 - A_2}$.

Clearly, $\grr_r (\x) = \grPr_r (\x) \in \Grr_r (\HD_1, \spinc_1)$, and 
similarly for $\y_3$ and $\z_1$.  To compute $\grr_r (\y_2)$, for example, one 
might choose $B_0 \in \pi_2 (\y_2, \y_3)$ to be the homology class that 
corresponds to the term $(1, 3) \tensor \y_2$ in $\delta (\y_3)$, to see that
\[
  \grr_r (\y_2) = \grr_r (B_0) \cdot \grPr_r (\y_3) = \paren{\frac{1}{2}, 
    -[\rho_{123}]} \cdot \grPr_r (\y_3) = - A_1 \cdot \subgrp{- A_1 - A_3}.
\]
Alternatively, one could \emph{use} the fact that $(1, 3) \tensor \y_2$ appears 
in $\delta (\y_3)$ and that the grading shift of $\delta$ is $- \lambda$ to 
arrive at the same conclusion.

\subsubsection{Skein-reduced gradings}
\label{sssec:GrHsk}

Our goal is to show that the maps $f_k$, $\vphi_k$, and $\kappa_k$ respect 
gradings in a suitable sense. Towards this goal, we now recall the definitions 
of $\grrPsk$ and $\Grrsk$ in \fullref{sssec:grrZsk}:
\begin{gather*}
  \grrPsk = \subgrp{A_1, A_2, A_3},\\
  \Grrsk = \Grr (\arcdiag) / \grrPsk \isom \Z.
\end{gather*}
Since $\grPr_r (\x)$, $\grPr_r (\y_3)$, and $\grPr_r (\z_1)$ are all subgroups 
of $\grrPsk$, we obtain maps $\grrrsk \colon \genset (\HD_k, \spinc_k) \to 
\Grrsk$ for each $k \in \set{\infty, 0, 1}$, which are relative gradings in the 
$\Grr (\arcdiag)$-set $\Grrsk$. In other words, we have united the three 
gradings on $\HD_k$. We also call these gradings the \emph{skein-reduced 
  gradings}.

We record the gradings of all generators below. Again, we denote elements of 
$\Grrsk$ by the integer obtained by sending $\lambda \cdot \grrPsk$ to $1$.
\begin{align*}
  \grr_r (\x)& = \subgrp{- A_2 - A_3}, & \grrrsk (\x) = 0,\\
  \grr_r (\y_1)& = A_3 \cdot \subgrp{- A_1 - A_3}, & \grrrsk (\y_1) = 0,\\
  \grr_r (\y_2)& = - A_1 \cdot \subgrp{- A_1 - A_3}, & \grrrsk (\y_2) = 0,\\
  \grr_r (\y_3)& = \subgrp{- A_1 - A_3}, & \grrrsk (\y_3) = 0,\\
  \grr_r (\z_1)& = \subgrp{- A_1 - A_2}, & \grrrsk (\z_1) = 0,\\
  \grr_r (\z_2)& = \subgrp{- A_1 - A_2}, & \grrrsk (\z_2) = 0.
\end{align*}

Finally, we mention that, for $a \in \AZ$ and $\w \in \genset (\HD_k)$, we have
\[
  \grrrsk (a \tensor \w) = \grr_r (a) \cdot \grrrsk (\w) = \grrrsk (a) + 
  \grrrsk (\w),
\]
since the action by $\grr_r (a) \in \Grr (\arcdiag)$ on $\Grrsk$ is exactly 
given by its value $\grrrsk (a)$ under the canonical projection.

\subsection{Graded versions of results and their proofs}
\label{ssec:gradings_proof}

We are now ready to state and prove a graded version of 
\fullref{prop:elementary}.

\begin{proposition}
  \label{prop:elementary_graded}
  With our choice of refinement data $r$, the morphisms $f_k \colon \BSDh (\Bk) 
  \to \BSDh (\Bnext)$ in \fullref{prop:elementary} preserve the relative 
  gradings $\grrrsk$ in the $\Grr (\arcdiag)$-set $\Grrsk \isom \Z$, with 
  grading shifts given by the following diagram:
  \[
    \xymatrix{
      \BSDh (\B_0) \ar[rr]^{-1} & & \BSDh (\B_1) \ar[ld]^{0}\\
      & \BSDh (\B_\infty) \ar[lu]^{0}
    }
  \]
  (The reader is reminded that there is a $\Grr (\arcdiag)$-action on the 
  grading set $\Grrsk$ for each of the three modules $\BSDh (\Bk)$, shifting by 
  any integer; accordingly, the diagram above should be interpreted up to 
  simultaneous and opposite integer shifts on any two consecutive arrows.) The 
  homotopy equivalence
  \[
    \BSDh (\Bk) \homeq \Cone (f_{k+1} \colon \BSDh (\Bnext) \to \BSDh (\Bprev))
  \]
  also respects the relative gradings $\grrrsk$.
\end{proposition}

\begin{proof}
  Combining \fullref{fig:comm_diag_f}, \fullref{fig:comm_diag_phi}, and the 
  grading computations in \fullref{sssec:grrZsk} and \fullref{sssec:GrHsk}, we 
  see that the maps $\fcombk$ in \fullref{sec:skein_via_computation} preserve 
  the relative gradings $\grrrsk$, and have grading shifts indicated by the 
  diagram.  Now $\vphicombk$ can also be easily checked to be homogeneous, and 
  have the appropriate grading shifts, using the same grading computations.  
  Recalling that $\kappa_k \equiv 0$ in \fullref{lem:cond3_comp}, this means 
  that all morphisms in \fullref{lem:hom_alg} are graded, implying that the 
  homotopy equivalence also respects the gradings.  (Note that 
  \fullref{prop:poly_comp} implies that the gradings are the same for $f_k$ and 
  $\vphi_k$ as defined in \fullref{sec:skein_via_polygon}.)
\end{proof}

\begin{remark}
  \label{rmk:gr_indep_on_r}
  A version of the statements in the proposition above should be true for every 
  choice of refinement data $r$. Here, one would possibly have to change the 
  basis of $\Grr (\arcdiag) \isom \Z \dirsum \Z^3$, i.e.\ to pick different 
  generators $A_1$, $A_2$, and $A_3$, in order to define $\grrPsk$ and 
  $\Grrsk$. In fact, one could first work with the unrefined gradings $\gr$ to 
  obtain an unrefined version of the proposition, and only then, refine the 
  gradings. For economy and clarity, we elected not to present this approach; 
  this saved us, for example, from computations in the quotient of $\Gr 
  (\arcdiag)$ by the normal closure of $\set{A_1, A_2, A_3}$.
\end{remark}

\begin{remark}
  \label{rmk:gr_abs}
  One could also remove the indeterminacy of simultaneous and opposite integer 
  shifts as follows. Instead of considering domain gradings for bigons in the 
  three Heegaard diagrams $\HD_1$, $\HD_\infty$, and $\HD_0$ in 
  \fullref{fig:hd} separately, one may introduce domain gradings for polygons 
  in the bordered--sutured multi-diagram $\HDab$ in \fullref{fig:mult_hd_3}.  
  Again, for simplicity, we opted not to take this approach.
\end{remark}

\begin{remark}
  \label{rmk:gr_01}
  The astute reader might have noticed in \fullref{sssec:GrHsk} that $\grrPsk$ 
  is not the smallest subgroup $\grrPmin$ containing $\grPr_r (\x)$, $\grPr_r 
  (\y_3)$, and $\grPr_r (\z_1)$. Indeed, while $\Grr (\arcdiag) / \grrPsk \isom 
  \Z$, we have $\Grr (\arcdiag) / \grrPmin \isom \Z \dirsum \cycgrp{2}$. In 
  fact, it is easy to check that, for each $k \in \set{\infty, 0, 1}$, the 
  morphisms $\fcomb_{k,0}$ and $\fcomb_{k,1}$ are both homogeneous with respect 
  to the gradings in $\Grr (\arcdiag) / \grrPmin$, but their grading shifts 
  differ by the generator of the $\cycgrp{2}$ factor.  If $\Yk = \Y' 
  \glue{\sutF (\arcdiag)} \Bk$ is the sutured manifold associated to a link $L 
  \subset S^3$, then the $\cycgrp{2}$ factor may be seen as the Alexander 
  grading on $\HFKh (S^3, L)$ modulo~2.  Assuming that \fullref{prop:comm_gen} 
  holds with gradings, one could use this fact, as well as the fact that 
  $\Etaknext_0$ and $\Etaknext_1$ differ in mod-2 Maslov gradings, to obtain 
  the mod-2 grading reduction of the $\delta$-graded skein relation in 
  \cite{ManOzs08:QALinks} (with gradings suitably shifted).
\end{remark}

Next, we turn to a graded version of \fullref{thm:general}. Let $\Y' = (Y', 
\Gamma', \arcdiag_1 \union \arcdiag_2 \union \arcdiag, \phi')$ be the 
bordered--sutured manifold defined in \fullref{ssec:proof_main}, with $\Y' 
\glue{F (\arcdiag)} \Bk = \Yk$, and let $\HD'$ be a bordered--sutured diagram 
for $\Y'$; then
\[
  \BSDAAh (\Y') \boxtensor \BSDh (\Bk) \homeq \BSDAh (\Yk).
\]
For $\spinc' \in \SpinC (Y', \bdy Y' \setminus F (\arcdiag_1 \union \arcdiag_2 
\union \arcdiag))$,
pick a generator $\w_0' \in \genset (\HD', \spinc')$. By the graded pairing 
theorem mentioned in \fullref{ssec:gradings_bsfh}, $\BSDAAh (\HD', \spinc') 
\boxtensor \BSDh (\HD_k)$ has a relative grading $\grbox_{r_1, r_2, r}$ with 
values in
\[
  \Grr_{r_1, r_2, r} (\HD', \HD_k, \spinc', \spinc_k) = \grPr_{r_1, r_2, r} 
  (\w_0') \rcoset \Grr (\arcdiag_1 \union \arcdiag_2 \union \arcdiag) / \grPr_r 
  (\w_{0,k}),
\]
where $\w_{0,k}$ is $\x$, $\y_3$, and $\z_1$, for $k = 1$, $0$, and $\infty$, 
respectively. We may, as before, take a skein reduction to get a relative 
grading $\grboxsk_{r_1, r_2, r}$ with values in the set
\[
  \Grrrrsk (\HD', \HD_k, \spinc', \spinc_k) = \grPr_{r_1, r_2, r} (\w_0') 
  \rcoset \Grr (\arcdiag_1 \union \arcdiag_2 \union \arcdiag) / \grrPsk.
\]

\begin{proposition}
  \label{prop:box_graded}
  For each $\spinc' \in \SpinC (Y', \bdy Y' \setminus F (\arcdiag_1 \union 
  \arcdiag_2 \union \arcdiag))$, and with our choice of refinement data $r$, 
  the morphisms
  \[
    \Id_{\BSDAAh (\Y')} \boxtensor f_k \colon \BSDAAh (\Y') \boxtensor \BSDh 
    (\Bk) \to \BSDAAh (\Y') \boxtensor \BSDh (\Bnext)
  \]
  preserve the relative gradings $\grrrrboxsk$ in the $\Grr (\arcdiag_1 \union 
  \arcdiag_2)$-set $\Grrrrsk (\HD', \HD_k, \spinc', \spinc_k)$, with grading 
  shifts given by the following diagram:
  \[
    \xymatrix{
      \BSDAAh (\HD', \spinc') \boxtensor \BSDh (\HD_0) \ar[rr]^{\lambda^{-1} = 
        (-1, 0)} & & \BSDAAh (\HD', \spinc') \boxtensor \BSDh (\HD_1) 
      \ar[ld]^{(0, 0)}\\
      & \BSDAAh (\HD', \spinc') \boxtensor \BSDh (\HD_\infty) \ar[lu]^{(0, 0)}
    }
  \]
  (Again, the diagram above should be interpreted up to simultaneous and 
  inverse (right) actions by elements of $\Grr (\arcdiag_1 \union \arcdiag_2)$ 
  on any two consecutive arrows.) The homotopy equivalence
  \[
    \BSDAAh (\Y') \boxtensor \BSDh (\Bk) \homeq \Cone (\Id_{\BSDAAh (\Y')} 
    \boxtensor f_{k+1})
  \]
  also respects the relative gradings $\grrrrboxsk$.
\end{proposition}

\begin{proof}
  Given a $G$-set-graded type~$D$ morphism $f \colon \M_1 \to \M_2$, with 
  grading shifts in $G / \grP$, it is straightforward to check that, for any 
  $G$-set-graded type~$A$ module $\module{N}$ with gradings in $\grP' \rcoset 
  G$, we have that $\Id_{\module{N}} \boxtensor f \colon \module{N} \boxtensor 
  \M_1 \to \module{N} \boxtensor \M_2$ is a $\Z$-set graded chain map with 
  grading shifts given by the obvious reduction to $\grP' \rcoset G / \grP$.  
  (More generally, the functor $(\module{N} \boxtensor \blank, \Id_{\module{N}} 
  \boxtensor \blank)$ can be upgraded to a $\dg$ functor in an appropriate 
  sense; see \cite[Example~2.5.35]{LipOzsThu15:BFHBimodules}.)
  
  The statement above generalizes in an obvious manner to the case where 
  $\module{N}$ is a type~$\DAA$ trimodule. Specializing this generalization to 
  our context would mean the following: Given the $\Grr (\arcdiag_1 \union 
  \arcdiag_2 \union \arcdiag)$-set-graded morphism $f_k \colon \BSDh (\HD_k) 
  \to \BSDh (\HD_{k+1})$, with gradings in $\Grr (\arcdiag) / \grrPsk \subset 
  \Grr (\arcdiag_1 \union \arcdiag_2 \union \arcdiag) / \grrPsk$, the morphism 
  $\Id_{\BSDAAh (\HD', \spinc')} \boxtensor f_k$ is a $\Grr (\arcdiag_1 \union 
  \arcdiag_2)$-set-graded morphism, with gradings given by the reduction to 
  $\grPr_{r_1, r_2, r} (\w_0') \rcoset \Grr (\arcdiag_1 \union \arcdiag_2 
  \union \arcdiag) / \grrPsk$.

  Recalling that the grading shift $- 1$ is a shorthand for $\lambda^{-1} \in 
  \Grr (\arcdiag) / \grrPsk \subset \Grr (\arcdiag_1 \union \arcdiag_2 \union 
  \arcdiag) / \grrPsk$, the proposition now follows directly from 
  \fullref{prop:elementary_graded}.
\end{proof}

We now work towards interpreting the result above in terms of the 
$\SpinC$-structures and intrinsic gradings on $\Yk = \Y' \glue{\sutF 
  (\arcdiag)} \Bk$. To do so, we must first understand the grading set 
$\grPr_{r_1, r_2, r} (\w_0') \rcoset \Grr (\arcdiag_1 \union \arcdiag_2 \union 
\arcdiag) / \grrPsk$, or in fact
\[
  \Grrrr (\HD', \HD_k, \spinc', \spinc_k) = \grPr_{r_1, r_2, r} (\w_0') \rcoset 
  \Grr (\arcdiag_1 \union \arcdiag_2 \union \arcdiag) / \grPr_r (\w_{0, k}),
\]
in terms of the discussion in \fullref{ssec:gradings_bsfh}.

First, assume for the moment that there exist generators $\wocc{1}, \wocc{3}, 
\wocc{5} \in \genset (\HD', \spinc')$ in idempotents $\Iocc{1}$, $\Iocc{3}$, 
and $\Iocc{5}$ respectively; then
\begin{itemize}
  \item $\wocc{1}$ is complementary to $\y_3 = \w_{0, \infty}$ and $\z_2$;
  \item $\wocc{3}$ is complementary to $\x = \w_{0, 1}$ and $\y_2$; and
  \item $\wocc{5}$ is complementary to $\y_1$ and $\z_1 = \w_{0, 0}$.
\end{itemize}
Let $\w_{0, k}'$ be $\wocc{3}$, $\wocc{1}$, and $\wocc{5}$, for $k = 1$, $0$, 
and $\infty$, respectively, so that $\w_{0, k}'$ is complementary to $\w_{0, 
  k}$. Let $B_{0, k}' \in \pi_2 (\w_0', \w_{0, k}')$ be any homology class from 
$\w_{0, k}'$ to $\w_0'$.  Adapting  \cite[Corollary~10.16]{LipOzsThu18:BFH}, 
one could choose to change the base generator in $\HD'$ from $\w_0'$ to $\w_{0, 
  k}'$, and identify the grading in $\Grr_{r_1, r_2, r} (\HD', \spinc') = 
\grPr_{r_1, r_2, r} (\w_0') \rcoset \Grr (\arcdiag_1 \union \arcdiag_2 \union 
\arcdiag)$ with the grading in $\grPr_{r_1, r_2, r} (\w_{0, k}') \rcoset \Grr 
(\arcdiag_1 \union \arcdiag_2 \union \arcdiag)$, such that the grading of a 
generator $\w' \in \genset (\HD', \spinc')$ with respect to $\w_{0, k}'$ is
\[
  \grr_{r_1, r_2, r} (B_{0, k}') \cdot \grr_{r_1, r_2, r} (\w').
\]
This correspondence carries directly to the double-coset space, allowing one to 
view gradings in $\Grr_{r_1, r_2, r} (\HD', \HD_k, \spinc', \spinc_k)$ as 
gradings in $\grPr_{r_1, r_2, r} (\w_{0, k}') \rcoset \Grr (\arcdiag_1 \union 
\arcdiag_2 \union \arcdiag) / \grPr_r (\w_{0, k})$. Similarly, one may view 
gradings in $\Grrrrsk (\HD', \HD_k, \spinc', \spinc_k)$ as gradings in 
$\grPr_{r_1, r_2, r} (\w_{0, k}') \rcoset \Grr (\arcdiag_1 \union \arcdiag_2 
\union \arcdiag) / \grrPsk$.

Denote by $\pi_k$ the projection
\[
  \pi \colon \grPr_{r_1, r_2, r} (\w_{0, k}') \rcoset \Grr (\arcdiag_1 \union 
  \arcdiag_2 \union \arcdiag) / \grPr_r (\w_{0, k}) \to \SpinC (Y_k, \bdy Y_k 
  \setminus F (\arcdiag_1 \union \arcdiag_2), \spinc', \spinc_k)
\]
as mentioned in \fullref{ssec:gradings_bsfh}. According to 
\cite{LipOzsThu18:BFH}, this projection to the set of $\SpinC$-structures on 
$(Y_k, \bdy Y_k \setminus F (\arcdiag_1 \union \arcdiag_2)$ that restrict to 
$\spinc'$ and $\spinc_k$, is defined as follows.  The quotient of the 
double-coset space by the action by $\Z$ on the Maslov component is naturally 
identified with the image of $H_1 (F (\arcdiag))$ in $H_1 (Y_k, F (\arcdiag_1 
\union \arcdiag_2))$; let $[ \blank ]$ denote this quotient map.  More 
concretely, for $g \in \Grr (\arcdiag_1 \union \arcdiag_2 \union \arcdiag)$, 
let $[\grPr_{r_1, r_2, r} (\w_{0, k}') \cdot g \cdot \grPr_r (\w_{0, k})]$ 
denote the image of the $H_1 (F (\arcdiag))$-component of $g$ in $H_1 (Y_k, F 
(\arcdiag_1 \union \arcdiag_2))$. Then $\pi_k$ is defined by
\[
  \pi_k (\grPr_{r_1, r_2, r} (\w_{0, k}') \cdot g \cdot \grPr_r (\w_{0, k})) = 
  \spinc (\w_{0, k}' \cross \w_{0, k}) + [\grPr_{r_1, r_2, r} (\w_{0, k}') 
  \cdot g \cdot \grPr_r (\w_{0, k})],
\]
where $\w_{0, k}' \cross \w_{0, k}$ is viewed as a generator in $\genset (\HD' 
\union \HD_k)$. (Recall that $\SpinC (Y_k, \bdy Y_k \setminus F (\arcdiag_1 
\union \arcdiag_2), \spinc', \spinc_k)$ is an affine copy of the image of $H_1 
(F (\arcdiag))$ in $H_1 (Y_k, F (\arcdiag_1 \union \arcdiag_2))$.) Each fiber 
$\pi_k^{-1} (\spinc'')$ is then isomorphic as a $\Grr (\arcdiag_1 \union 
\arcdiag_2)$-set to $\Grr_{r_1, r_2} (\HD' \union \HD_k, \spinc'')$.

Since \fullref{prop:box_graded} pertains to the skein-reduced gradings in 
$\Grrrrsk (\HD', \HD_k, \spinc', \spinc_k)$, which we identify with 
$\grPr_{r_1, r_2, r} (\w_{0, k}') \rcoset \Grr (\arcdiag_1 \union \arcdiag_2 
\union \arcdiag) / \grrPsk$, we must find an analogous correspondence for these 
grading sets too.  To do so, recall that $\grrPsk$ is generated by $A_1$, 
$A_2$, and $A_3$; the homological components of these elements are 
$[\rho_{123}]$, $[\rho_{456}]$, and $[\rho_{78}]$ respectively. Direct 
inspection of \fullref{fig:arc_diag} shows that these elements, viewed in $H_1 
(F (\arcdiag))$, are represented by meridian loops of three of the four 
punctures of $F (\arcdiag)$; we denote the images of these meridian loops in 
$H_1 (Y_k, F (\arcdiag_1 \union \arcdiag_2))$ by $\mu_1$, $\mu_2$, and $\mu_3$ 
respectively.  Letting $\Msk_k \subset H_1 (Y_k, F (\arcdiag_1 \union 
\arcdiag_2))$ be the subgroup generated by $\mu_1$, $\mu_2$, and $\mu_3$, we 
may thus define a projection
\[
  \pisk_k \colon \grPr_{r_1, r_2, r} (\w_{0, k}') \rcoset \Grr (\arcdiag_1 
  \union \arcdiag_2 \union \arcdiag) / \grrPsk \to \SpinC (Y_k, \bdy Y_k 
  \setminus F (\arcdiag_1 \union \arcdiag_2), \spinc', \spinc_k) / \eqrelsk,
\]
where the equivalence relation $\eqrelsk$ is defined by
\[
  \spinc_1'' \sim_{\mathrm{sk}} \spinc_2'' \qquad \text{if } \PD (\spinc_1'' - 
  \spinc_2'') \in \Msk_k.
\]
Here, $\PD$ denotes the isomorphism under Poincar\'e duality. We denote an 
equivalence class of $\SpinC$-structures by $[\spinc'']$, and the set of 
equivalence classes in $\SpinC (Y_k, \bdy Y_k \setminus F (\arcdiag_1 \union 
\arcdiag_2), \spinc', \spinc_k)$ by $E_k$. Recalling that the grading sets 
$\Grrrrsk (\HD', \HD_k, \spinc', \spinc_k)$ are in fact the same for $k \in 
\set{\infty, 0, 1}$, the projections $\pisk_k$ provide a family of bijections 
$e_k \colon E_k \to E_{k+1}$ such that $e_{k+2} \comp e_{k+1} \comp e_k$ is the 
identity. Furthermore, $(\pisk_k)^{-1} ([\spinc''])$ is isomorphic to 
$(\pisk_{k+1})^{-1} (e_k ([\spinc'']))$.

In order to understand the fibers of $\pisk_k$, we first note the following.

\begin{lemma}
  \label{lem:G-set_isom}
  For $k \in \set{\infty, 0, 1}$ and $i \in \set{1, 2, 3}$, the map
  \[
    \grPr_{r_1, r_2, r} (\w_{0, k}') \cdot g \cdot \grPr_r (\w_{0, k}) \mapsto 
    \grPr_{r_1, r_2, r} (\w_{0, k}') \cdot g A_i \cdot \grPr_r (\w_{0, k})
  \]
  defines an isomorphism $\pi_k^{-1} (\spinc'') \isom \pi_k^{-1} (\spinc'' + 
  \PD (\mu_i))$ as $\Grr (\arcdiag_1 \union \arcdiag_2)$-sets. This extends to 
  multiplication by $A_1^{n_1} A_2^{n_2} A_3^{n_3}$, giving isomorphisms 
  $\pi_k^{-1} (\spinc'') \isom \pi_k^{-1} (\spinc'' + \PD (h))$ for some $h \in 
  \Msk_k$.
\end{lemma}

\begin{proof}
  First, note that $\Grr (\arcdiag)$ is contained in the center of $\Grr 
  (\arcdiag_1 \union \arcdiag_2 \union \arcdiag)$, since $\Grr (\arcdiag)$ is 
  Abelian, and $\arcdiag_1$ and $\arcdiag_2$ are disjoint from $\arcdiag$. The 
  fact that the given map is well defined and bijective, and respects the 
  action of $\Grr (\arcdiag_1 \union \arcdiag_2)$, follows from the observation 
  that $A_i$ is central for each $i$.
\end{proof}

\begin{remark}
  \label{rmk:G-set_isom_self}
  It is possible that $\mu_i = 0$ (or, more generally, $h = 0$), in which case 
  \fullref{lem:G-set_isom} provides an isomorphism of $\pi_k^{-1} (\spinc'')$ 
  with itself, which may not be the identity isomorphism.
\end{remark}

\begin{remark}
  \label{rmk:G-set_isom_closed}
  When $\arcdiag_1 = \arcdiag_2 = \eset$, i.e.\ when $\Yk$ is a sutured 
  manifold for each $k$, the isomorphism in \fullref{lem:G-set_isom} takes on a 
  particularly simple form.  By \cite[Proposition~6.3.2]{Zar11:BSFH}, $\Grr 
  (\HD' \union \HD_k, \spinc'')$, and hence the fiber $\pi_k^{-1} (\spinc'')$, 
  is (non-canonically) isomorphic to $\Z / \div (\spinc'') \Z$. (Recall that 
  the divisibility $\div (\spinc'')$  of $\spinc''$ is the integer 
  $\gcd_{\alpha \in H_2 (Y_k)} \eval{c_1 (\spinc'')}{\alpha}$, where $c_1 
  \colon \SpinC (Y_k, \bdy Y_k \setminus F (\arcdiag_1 \union \arcdiag_2))$ 
  sends $\spinc''$ to the Euler class of the orthogonal bundle of a vector 
  field representative of $\spinc''$.) Thus, the lemma states that the fibers 
  $\pi_k^{-1} (\spinc'')$ and $\pi_k^{-1} (\spinc'' + \PD (\mu_i))$ are 
  isomorphic as $\Z$-sets, or, equivalently, that $\div (\spinc'') = \div 
  (\spinc'' + \PD (\mu_i))$. Indeed, under the map $H_1 (Y_k, F (\arcdiag_1 
  \union \arcdiag_2)) = H_1 (Y_k) \to H_1 (Y_k, \bdy Y_k)$ from the long exact 
  sequence for a pair, the images of $\mu_i$ are trivial.  Under Poincar\'e 
  duality, this is equivalent to saying that the connecting homomorphism 
  $\delta \colon H^2 (Y_k, \bdy Y_k \setminus F (\arcdiag_1 \union \arcdiag_2)) 
  = H^2 (Y_k, \bdy Y_k) \to H^2 (Y_k)$ sends $\PD (\mu_i)$ to zero.  This 
  implies that
  \[
    c_1 (\spinc'' + \PD (\mu_i)) = c_1 (\spinc'') + 2 \delta (\PD (\mu_i)) = c_1 
    (\spinc''),
  \]
  which in turn implies that $\div (\spinc'') = \div (\spinc'' + \PD (\mu_i))$.
\end{remark}

\fullref{lem:G-set_isom} allows us to understand the fibers of $\pisk_k$ as 
follows.  Viewing $\grPr_{r_1, r_2, r} (\w_{0, k}') \rcoset \Grr (\arcdiag_1 
\union \arcdiag_2 \union \arcdiag) / \grPr_r (\w_{0, k})$ as the disjoint union 
of fibers $\pi_k^{-1} (\spinc'')$, each isomorphic to $\Grr_{r_1, r_2} (\HD' 
\union \HD_k, \spinc'')$, the reduction to $\grPr_{r_1, r_2, r} (\w_{0, k}') 
\rcoset \Grr (\arcdiag_1 \union \arcdiag_2 \union \arcdiag) / \grrPsk$ may be 
seen as obtained by identifying the fibers $\pi_k^{-1} (\spinc'')$ for all 
$\spinc''$ in the same equivalence class under $\eqrelsk$, using the relators 
$A_1 = A_2 = A_3 = e$ (where $e$ is the identity element). We denote the 
resulting $\Grr (\arcdiag_1 \union \arcdiag_2)$-set by $\Grr_{r_1, r_2} (\HD' 
\union \HD_k, [\spinc''])$.  Note that, for each $k$, at least one of the three 
relators is redundant, and so we may consider only two relators; for example, 
setting $A_1 = e$ in $\grPr_r (\w_{0, 0}') \rcoset \Grr (\arcdiag_1 \union 
\arcdiag_2 \union \arcdiag) / \grPr_r (\w_{0, 0})$ immediately implies $A_2 = 
e$ (since $A_1 A_2 = e$).  We illustrate this identification process with two 
examples; as we shall see, $\Grr_{r_1, r_2} (\HD' \union \HD_k, [\spinc''])$ 
may or may not be isomorphic to $\Grr_{r_1, r_2} (\HD' \union \HD_k, 
\spinc'')$.

\begin{example}
  \label{eg:iden_fibers_1}
  Let $\arcdiag_1 = \arcdiag_2 = \eset$, so that $\Yk$ is a sutured manifold 
  for each $k$. Suppose $\grPr_{r} (\w_0') = \subgrp{A_1 + A_3}$. Since $\Grr 
  (\arcdiag)$ is Abelian, the conjugate $\grPr_r (\w_{0, k}')$ is simply 
  $\grPr_r (\w_0')$. Then
  \begin{align*}
    \grPr_{r} (\w_{0, 1}') \rcoset \Grr (\arcdiag) / \grPr_r (\w_{0, 1})& = 
    \subgrp{A_1 + A_3} \rcoset \subgrp{\lambda, A_1, A_2, A_3} / \subgrp{A_2 + 
      A_3} \isom \Z \dirsum \Z,\\
    \grPr_{r} (\w_{0, \infty}') \rcoset \Grr (\arcdiag) / \grPr_r (\w_{0, 
      \infty})& = \subgrp{A_1 + A_3} \rcoset \subgrp{\lambda, A_1, A_2, A_3} / 
    \subgrp{A_1 + A_3} \isom \Z \dirsum \Z \dirsum \Z,\\
    \grPr_{r} (\w_{0, 0}') \rcoset \Grr (\arcdiag) / \grPr_r (\w_{0, 0})& = 
    \subgrp{A_1 + A_3} \rcoset \subgrp{\lambda, A_1, A_2, A_3} / \subgrp{A_1 + 
      A_2} \isom \Z \dirsum \Z,\\
    \grPr_{r} (\w_0') \rcoset \Grr (\arcdiag) / \grrPsk& = \subgrp{A_1 + A_3} 
    \rcoset \subgrp{\lambda, A_1, A_2, A_3} / \subgrp{A_1, A_2, A_3} \isom \Z,
  \end{align*}
  where the first summand on the right hand side is the Maslov component. This 
  situation may arise, for example, in the familiar setting in 
  \cite{ManOzs08:QALinks}, when $\Yk$ is the complement of $L_k \subset S^3$ 
  with meridional sutures, where $L_0$ and $L_1$ are knots and $L_\infty$ is a 
  $2$-component link. The $\Z$-summands in the homological components are then 
  in correspondence with the link components, and may be interpreted as the 
  relative Alexander gradings. On each fixed $\SpinC$-structure (i.e.\ in each 
  Alexander grading), there is a relative Maslov $\Z$-grading. For $k = 
  \infty$, the two relators $A_1 = 0$ (or equivalently, $A_3 = 0$) and $A_2 = 
  0$ identify the $(\Z \dirsum \Z)$-worth of fibers $\pi_\infty^{-1} 
  (\spinc'')$, each isomorphic to $\Z$, with each other uniquely.  For $k = 0, 
  1$, one of the two relators identifies the $\Z$-worth of fibers uniquely, and 
  the other relator is redundant.  Thus, in all three cases, the skein 
  reduction is the projection onto the Maslov component.  Although we do not 
  prove it here, we expect the skein reduction to be related to the relative 
  $\delta$-grading in \cite{ManOzs08:QALinks}, and the graded exact sequence in 
  \fullref{thm:gradings} below to be related to that in 
  \cite{ManOzs08:QALinks}.
\end{example}

\begin{example}
  \label{eg:iden_fibers_2}
  Let $\arcdiag_1 = \arcdiag_2 = \eset$. Suppose $\grPr_r (\w_0') = \subgrp{12 
    \lambda + A_1 + A_2, 18 \lambda + A_1 + A_3}$. Then
  \begin{align*}
    \grPr_{r} (\w_{0, 1}') \rcoset \Grr (\arcdiag) / \grPr_r (\w_{0, 1})
    & \isom \Z \dirsum \cycgrp{2},\\
    \grPr_{r} (\w_{0, \infty}') \rcoset \Grr (\arcdiag) / \grPr_r (\w_{0, 
      \infty})
    & \isom \cycgrp{18} \dirsum \Z,\\
    \grPr_{r} (\w_{0, 0}') \rcoset \Grr (\arcdiag) / \grPr_r (\w_{0, 0})
    & \isom \cycgrp{12} \dirsum \Z,\\
    \grPr_{r} (\w_0') \rcoset \Grr (\arcdiag) / \grrPsk
    & \isom \cycgrp{6},
  \end{align*}
  where again the first summand on the right is the Maslov component. For $k = 
  1$, there are exactly two $\SpinC$-structures in $\SpinC (Y_1, \bdy Y_1, 
  \spinc', \spinc_1)$, with a relative Maslov $\Z$-grading on each.  As the 
  ($\cycgrp{2}$)-summand is generated by $15 \lambda + A_1 = 3 \lambda - A_2$, 
  we may identify $A_1$ with $(-15, 1)$, and $A_2$ with $(3, 1)$.  The relator 
  $A_1 = 0$ does more than identify the two fibers $\pi_1^{-1} (\spinc'')$; 
  since $2 A_1 = (-30, 0) = 0$, it identifies each fiber with itself by a 
  degree-$30$ shift, reducing each fiber to $\cycgrp{30}$.  Similarly, $A_2 = 
  0$  further reduces each fiber to $\cycgrp{6}$. This describes the skein 
  reduction for $k = 1$. In contrast, for $k = \infty$, the element $A_1$ has 
  infinite order, and so the relator $A_1 = 0$ only identifies distinct fibers 
  uniquely and does not reduce the fibers; $A_2 = 0$ also identifies the fibers 
  uniquely. However, since the isomorphisms given by $A_1$ and $A_2$ are 
  different, the two relators together force a reduction of the fibers from 
  $\cycgrp{18}$ to $\cycgrp{6}$.  The skein reduction for $k = 0$ is similar.
  While we do not provide a diagram $\HD'$ where $\grPr_r (\w_0')$ arises, the 
  phenomena this example illustrates could occur in general.
\end{example}

Note that the discussion above remains valid if we choose some other pair of 
complementary base generators $\w_{0, k}'$ and $\w_{0, k}$ in $\HD'$ and 
$\HD_k$. For example, we could have chosen to work with $\wocc{5} \cross \y_1$ 
instead of $\wocc{1} \cross \y_3$ for $k = \infty$. Indeed, the correspondence 
between gradings with respect to different base generators respects both the 
projection $\pi_k$ and the $\Grr (\arcdiag_1 \union \arcdiag_2)$-action on the 
fibers, and commutes with the isomorphism in \fullref{lem:G-set_isom} (since 
$A_i$ is central). This means that the identification process in the skein 
reduction is in fact independent of the pair of complementary base generators.  
This allows us to deal with the case when one or more of $\wocc{1}$, 
$\wocc{3}$, and $\wocc{5}$ does not exist: For example, if $\wocc{5}$ does not 
exist (i.e.\ if there does not exist a generator $\genset (\HD', \spinc')$ in 
idempotent $\Iocc{5}$), it would serve us well to let $\w_{0, 0}' \cross \w_{0, 
  0}$ be $\wocc{1} \cross \z_2$ instead of $\wocc{5} \cross \z_1$ as defined 
above. (In the case that no generator in $\genset (\HD', \spinc')$ is 
compatible with one of the generators of $\genset (\HD_k)$, then the situation 
is in fact simpler: The summand $\BSDAh (\HD' \union \HD_k, \spinc'')$ is 
trivial for all $\spinc'' \in \SpinC (Y_k, \bdy Y_k \setminus F (\arcdiag_1 
\union \arcdiag_2), \spinc', \spinc_k)$, and the skein relation, when 
restricted to the summand associated to $\spinc'$, is an isomorphism between 
modules associated to $\Ynext$ and $\Yprev$.)

We are now ready to state the interpretation of \fullref{prop:box_graded} in 
terms of the identification above.

\begin{maintheorem}
  \label{thm:gradings}
  Let $\Yk = \Y' \glue{\sutF (\arcdiag)} \Bk$ be as in \fullref{thm:general}, 
  and fix $\spinc' \in \SpinC (Y', \bdy Y' \setminus F (\arcdiag_1 \union 
  \arcdiag_2))$. Viewing $\Grrrr (\HD', \HD_k, \spinc', \spinc_k)$ as the 
  disjoint union of fibers $\pi_k^{-1} (\spinc'') \isom \Grr_{r_1, r_2} (\HD' 
  \union \HD_k, \spinc'')$, the isomorphisms in \fullref{lem:G-set_isom} 
  provide a way to identify $\Grr_{r_1, r_2} (\HD' \union \HD_k, \spinc'')$ for 
  $\SpinC$-structures $\spinc''$ that differ by a linear combination of the 
  Poincar\'e duals of the meridians of $\elT_k$, in the process possibly 
  reducing the grading set, such that the resulting grading set $\Grr_{r_1, 
    r_2} (\HD' \union \HD_k, [\spinc''])$ satisfies
  \[
    \Grr_{r_1, r_2} (\HD' \union \HD_k, [\spinc'']) \isom \Grr_{r_1, r_2} (\HD' 
    \union \HD_{k+1}, e_k ([\spinc'']))
  \]
  as $\Grr (\arcdiag_1 \union \arcdiag_2)$-sets. Then, interpreting 
  $\grrrrboxsk$ as assigning to each generator in $\BSDAh (\Yk)$ an equivalence 
  class $[\spinc'']$ and a relative grading in $\Grr_{r_1, r_2} (\HD' \union 
  \HD_k, [\spinc''])$, the morphisms
  \[
    F_k \colon \BSDAh (\Yk) \to \BSDAh (\Ynext)
  \]
  in \fullref{thm:general}, and hence \fullref{thm:main}, send the summand 
  associated to $[\spinc'']$ to $e_k ([\spinc''])$, and preserve the relative 
  gradings in $\Grr_{r_1, r_2} (\HD' \union \HD_k, [\spinc''])$, with grading 
  shifts given by the following diagram:
  \[
    \xymatrix{
      \BSDAh (\HD' \union \HD_0, [\spinc'']) \ar[rr]^{\lambda^{-1} = (-1, 0)} & 
      & \BSDAh (\HD' \union \HD_1, e_0 ([\spinc''])) \ar[ld]^{(0, 0)}\\
      & \BSDAh (\HD' \union \HD_\infty, e_{\infty}^{-1} ([\spinc'']))
      \ar[lu]^{(0, 0)}
    }
  \]
  (Again, the diagram above should be interpreted up to simultaneous and 
  inverse (right) actions by elements of $\Grr (\arcdiag_1 \union \arcdiag_2)$ 
  on any two consecutive arrows.) The homotopy equivalence
  \[
    \BSDAh (\Yk) \homeq \Cone (F_{k+1} \colon \BSDAh (\Ynext) \to \BSDAh 
    (\Yprev))
  \]
  also respects the relative gradings in $\Grr_{r_1, r_2} (\HD' \union \HD_k, 
  [\spinc''])$.
\end{maintheorem}

\begin{proof}
  \fullref{prop:box_graded} and the discussion above implies that the morphisms
  \[
    \Id_{\BSDAAh (\Y')} \boxtensor f_{k} \colon \BSDAAh (\Y') \boxtensor \BSDh 
    (\Bk) \to \BSDAAh (\Y') \boxtensor \BSDh (\Bnext)
  \]
  satisfy the description in the theorem. Since the homotopy equivalence
  \[
    \BSDAAh (\Y') \boxtensor \BSDh (\Bk) \homeq \BSDAh (\Yk)
  \]
  is graded, the theorem follows.
\end{proof}

\begin{remark}
  \label{rmk:gradings_closed}
  For the result of applying \fullref{thm:gradings} to \fullref{cor:arbitrary}, 
  the reader is invited to consider \fullref{eg:iden_fibers_2} and other 
  similar cases.
\end{remark}


\bibliographystyle{mwamsalphack}
\bibliography{Bibliography}

\providecommand{\bysame}{\leavevmode\hbox to3em{\hrulefill}\thinspace}
\providecommand{\MR}{\relax\ifhmode\unskip\space\fi MR }
\providecommand{\MRhref}[2]{%
  \href{http://www.ams.org/mathscinet-getitem?mr=#1}{#2}
}
\providecommand{\href}[2]{#2}
\begin{thebibliography}{FOOO09b}

\bibitem[AL17]{AliLip17:BFHCompressingDisks}
Akram Alishahi and Robert Lipshitz, \emph{Bordered {F}loer homology and
  incompressible surfaces}, preprint, version 2, 2017,
  \href{http://arxiv.org/abs/1708.05121}{\texttt{arXiv:1708.05121}}.

\bibitem[Bao17]{Bao17:HFG}
Yuanyuan Bao, \emph{Heegaard {F}loer homology for embedded bipartite graphs},
  preprint, version 3, 2017,
  \href{http://arxiv.org/abs/1401.6608}{\texttt{arXiv:1401.6608}}.

\bibitem[BL12]{BalLev12:HFKUnorientedSkeinIterate}
John~A. Baldwin and Adam~Simon Levine, \emph{A combinatorial spanning tree
  model for knot {F}loer homology}, Adv. Math. \textbf{231} (2012), no.~3-4,
  1886--1939. \MR{2964628}

\bibitem[FOOO09a]{FukOhOht09:LFHBookI}
Kenji Fukaya, Yong-Geun Oh, Hiroshi Ohta, and Kaoru Ono, \emph{Lagrangian
  intersection {F}loer theory: anomaly and obstruction. {P}art {I}}, AMS/IP
  Studies in Advanced Mathematics, vol.~46, American Mathematical Society,
  Providence, RI; International Press, Somerville, MA, 2009. \MR{2553465}

\bibitem[FOOO09b]{FukOhOht09:LFHBookII}
\bysame, \emph{Lagrangian intersection {F}loer theory: anomaly and obstruction.
  {P}art {II}}, AMS/IP Studies in Advanced Mathematics, vol.~46, American
  Mathematical Society, Providence, RI; International Press, Somerville, MA,
  2009. \MR{2548482}

\bibitem[HO17]{HarODo17:HFG}
Shelly Harvey and Danielle O'Donnol, \emph{Heegaard {F}loer homology of spatial
  graphs}, Algebr. Geom. Topol. \textbf{17} (2017), no.~3, 1445--1525.
  \MR{3677933}

\bibitem[Hom14]{Hom14:Epsilon}
Jennifer Hom, \emph{Bordered {H}eegaard {F}loer homology and the tau-invariant
  of cable knots}, J. Topol. \textbf{7} (2014), no.~2, 287--326. \MR{3217622}

\bibitem[Juh06]{Juh06:SFH}
Andr{{\'a}}s Juh{{\'a}}sz, \emph{Holomorphic discs and sutured manifolds},
  Algebr. Geom. Topol. \textbf{6} (2006), 1429--1457. \MR{2253454
  (2007g:57024)}

\bibitem[Kho00]{Kho00:Kh}
Mikhail Khovanov, \emph{A categorification of the {J}ones polynomial}, Duke
  Math. J. \textbf{101} (2000), no.~3, 359--426. \MR{1740682 (2002j:57025)}

\bibitem[Lam17]{Lam17:GHMutation}
Peter Lambert-Cole, \emph{On {C}onway mutation and link homology}, preprint,
  2017, \href{http://arxiv.org/abs/1701.00880}{\texttt{arXiv:1701.00880}}.

\bibitem[Lip06]{Lip06:HFCylindrical}
Robert Lipshitz, \emph{A cylindrical reformulation of {H}eegaard {F}loer
  homology}, Geom. Topol. \textbf{10} (2006), 955--1097. \MR{2240908
  (2007h:57040)}

\bibitem[LOSSz09]{LisOzsSti09:LOSS}
Paolo Lisca, Peter Ozsv{{\'a}}th, Andr{{\'a}}s~I. Stipsicz, and Zolt{{\'a}}n
  Szab{{\'o}}, \emph{Heegaard {F}loer invariants of {L}egendrian knots in
  contact three-manifolds}, J. Eur. Math. Soc. (JEMS) \textbf{11} (2009),
  no.~6, 1307--1363. \MR{2557137 (2010j:57016)}

\bibitem[LOT14]{LipOzsThu14:BFHBranchedDoubleSSI}
Robert Lipshitz, Peter~S. Ozsv{\'a}th, and Dylan~P. Thurston, \emph{Bordered
  {F}loer homology and the spectral sequence of a branched double cover {I}},
  J. Topol. \textbf{7} (2014), no.~4, 1155--1199. \MR{3286900}

\bibitem[LOT15]{LipOzsThu15:BFHBimodules}
\bysame, \emph{Bimodules in bordered {H}eegaard {F}loer homology}, Geom. Topol.
  \textbf{19} (2015), no.~2, 525--724. \MR{3336273}

\bibitem[LOT16]{LipOzsThu16:BFHBranchedDoubleSSII}
\bysame, \emph{Bordered {F}loer homology and the spectral sequence of a
  branched double cover {II}: the spectral sequences agree}, J. Topol.
  \textbf{9} (2016), no.~2, 607--686. \MR{3509974}

\bibitem[LOT18]{LipOzsThu18:BFH}
Robert Lipshitz, Peter~S. Ozsvath, and Dylan~P. Thurston, \emph{Bordered
  {H}eegaard {F}loer homology}, Mem. Amer. Math. Soc. \textbf{254} (2018),
  no.~1216, viii+279. \MR{3827056}

\bibitem[Man07]{Man07:HFKUnorientedSkein}
Ciprian Manolescu, \emph{An unoriented skein exact triangle for knot {F}loer
  homology}, Math. Res. Lett. \textbf{14} (2007), no.~5, 839--852. \MR{2350128
  (2008m:57074)}

\bibitem[MO08]{ManOzs08:QALinks}
Ciprian Manolescu and Peter Ozsv{{\'a}}th, \emph{On the {K}hovanov and knot
  {F}loer homologies of quasi-alternating links}, Proceedings of {G}{\"o}kova
  {G}eometry-{T}opology {C}onference 2007, G{\"o}kova Geometry/Topology
  Conference (GGT), G{\"o}kova, 2008, pp.~60--81. \MR{2509750 (2010k:57029)}

\bibitem[MOS09]{ManOzsSar09:GH}
Ciprian Manolescu, Peter Ozsv{{\'a}}th, and Sucharit Sarkar, \emph{A
  combinatorial description of knot {F}loer homology}, Ann. of Math. (2)
  \textbf{169} (2009), no.~2, 633--660. \MR{2480614 (2009k:57047)}

\bibitem[Ni07]{Ni07:HFKFibered}
Yi~Ni, \emph{Knot {F}loer homology detects fibred knots}, Invent. Math.
  \textbf{170} (2007), no.~3, 577--608. \MR{2357503}

\bibitem[NOT08]{NgOzsThu08:GRIDEffective}
Lenhard Ng, Peter Ozsv{{\'a}}th, and Dylan Thurston, \emph{Transverse knots
  distinguished by knot {F}loer homology}, J. Symplectic Geom. \textbf{6}
  (2008), no.~4, 461--490. \MR{2471100 (2009j:57014)}

\bibitem[OSSz17]{OzsStiSza17:Upsilon}
Peter~S. Ozsv{\'a}th, Andr{\'a}s~I. Stipsicz, and Zolt{\'a}n Szab{\'o},
  \emph{Concordance homomorphisms from knot {F}loer homology}, Adv. Math.
  \textbf{315} (2017), 366--426. \MR{3667589}

\bibitem[OSz03]{OzsSza03:HFKg4}
Peter Ozsv{{\'a}}th and Zolt{{\'a}}n Szab{{\'o}}, \emph{Knot {F}loer homology
  and the four-ball genus}, Geom. Topol. \textbf{7} (2003), 615--639.
  \MR{2026543 (2004i:57036)}

\bibitem[OSz04a]{OzsSza04:HFThurstonNormHFKGenus}
\bysame, \emph{Holomorphic disks and genus bounds}, Geom. Topol. \textbf{8}
  (2004), 311--334. \MR{2023281 (2004m:57024)}

\bibitem[OSz04b]{OzsSza04:HFK}
\bysame, \emph{Holomorphic disks and knot invariants}, Adv. Math. \textbf{186}
  (2004), no.~1, 58--116. \MR{2065507 (2005e:57044)}

\bibitem[OSz04c]{OzsSza04:HF}
\bysame, \emph{Holomorphic disks and topological invariants for closed
  three-manifolds}, Ann. of Math. (2) \textbf{159} (2004), no.~3, 1027--1158.
  \MR{2113019 (2006b:57016)}

\bibitem[OSz05]{OzsSza05:HFBranchedDoubleSS}
\bysame, \emph{On the {H}eegaard {F}loer homology of branched double-covers},
  Adv. Math. \textbf{194} (2005), no.~1, 1--33. \MR{2141852 (2006e:57041)}

\bibitem[OSz08]{OzsSza08:HFL}
\bysame, \emph{Holomorphic disks, link invariants and the multi-variable
  {A}lexander polynomial}, Algebr. Geom. Topol. \textbf{8} (2008), no.~2,
  615--692. \MR{2443092 (2010h:57023)}

\bibitem[PV16]{PetVer16:TFH}
Ina Petkova and Vera V{\'e}rtesi, \emph{Combinatorial tangle {F}loer homology},
  Geom. Topol. \textbf{20} (2016), no.~6, 3219--3332. \MR{3590353}

\bibitem[PW18]{PetWon18:TFHSkein}
Ina Petkova and C.-M.~Michael Wong, \emph{Skein relations for tangle {F}loer
  homology}, preprint, version 2, 2018,
  \href{http://arxiv.org/abs/1611.04065}{\texttt{arXiv:1611.04065}}.

\bibitem[Ras03]{Ras03:HFK}
Jacob~Andrew Rasmussen, \emph{Floer homology and knot complements}, Ph.D.
  thesis, Harvard University, 2003, 126~pages,
  \url{https://search.proquest.com/docview/305332635}. \MR{2704683}

\bibitem[Ras05]{Ras05:HFKKhExpo}
Jacob Rasmussen, \emph{Knot polynomials and knot homologies}, Geometry and
  topology of manifolds, Fields Inst. Commun., vol.~47, Amer. Math. Soc.,
  Providence, RI, 2005, pp.~261--280. \MR{2189938 (2006i:57029)}

\bibitem[Sei08]{Sei08:FukBook}
Paul Seidel, \emph{Fukaya categories and {P}icard-{L}efschetz theory}, Zurich
  Lectures in Advanced Mathematics, European Mathematical Society (EMS),
  Z{\"u}rich, 2008. \MR{2441780}

\bibitem[SW10]{SarWan10:Nice}
Sucharit Sarkar and Jiajun Wang, \emph{An algorithm for computing some
  {H}eegaard {F}loer homologies}, Ann. of Math. (2) \textbf{171} (2010), no.~2,
  1213--1236. \MR{2630063 (2012f:57032)}

\bibitem[Vir04]{Vir04:KhUnorientedSkein}
Oleg Viro, \emph{Khovanov homology, its definitions and ramifications}, Fund.
  Math. \textbf{184} (2004), 317--342. \MR{2128056}

\bibitem[Won17]{Won17:GHUnorientedSkein}
C.-M.~Michael Wong, \emph{Grid diagrams and {M}anolescu's unoriented skein
  exact triangle for knot {F}loer homology}, Algebr. Geom. Topol. \textbf{17}
  (2017), no.~3, 1283--1321. \MR{3677929}

\bibitem[Zar11]{Zar11:BSFH}
Rumen Zarev, \emph{Bordered {S}utured {F}loer {H}omology}, Ph.D. thesis,
  Columbia University, 2011, 189~pages,
  \url{https://search.proquest.com/docview/868276640}. \MR{2941830}

\bibitem[Zib17]{Zib17:Dissertation}
Claudius~Bodo Zibrowius, \emph{On a {H}eegaard {F}loer theory for tangles},
  Ph.D. thesis, University of Cambridge, 2017, 135~pages.

\end{thebibliography}


\end{document}